\documentclass[11pt, reqno]{amsart}

\usepackage{amsthm, amsmath, amssymb,booktabs,xcolor,graphicx}
\colorlet{refkey}{orange!20}
\colorlet{labelkey}{blue!60}

\usepackage[margin=1in]{geometry}
\usepackage[colorinlistoftodos]{todonotes}
\usepackage{enumitem}
\usepackage{scalerel}
\usepackage{makecell}
\usepackage{multicol}
\usepackage{mathtools}

\usepackage[hidelinks]{hyperref}

\usepackage[noabbrev,capitalize]{cleveref}
\crefname{equation}{}{}

\usepackage{environ}
\usepackage{cancel}
\usepackage{graphicx}

\usepackage{diagbox}

\usepackage{xcolor}

\newtheorem{theorem}{Theorem}[section]
\newtheorem{proposition}[theorem]{Proposition}
\newtheorem{lemma}[theorem]{Lemma}
\newtheorem{claim}[theorem]{Claim}
\newtheorem{corollary}[theorem]{Corollary}

\newtheorem{observation}[theorem]{Observation}
\newtheorem*{question*}{Question}
\newtheorem{fact}[theorem]{Fact}

\theoremstyle{definition}
\newtheorem{definition}[theorem]{Definition}

\newtheorem{question}[theorem]{Question}
\newtheorem*{definition*}{Definition}
\newtheorem{example}[theorem]{Example}
\newtheorem{setup}[theorem]{Setup}
\theoremstyle{remark}
\newtheorem{remark}[theorem]{Remark}

\newcommand{\norm}[1]{\left\lVert#1\right\rVert}

\newcommand{\floor}[1]{\left\lfloor #1 \right\rfloor}

\newcommand{\EE}{\mathbb{E}}
\newcommand{\CC}{\mathbb{C}}

\newcommand{\RR}{\mathbb{R}}
\newcommand{\PP}{\mathbb{P}}
\newcommand{\NN}{\mathbb{N}}
\newcommand{\FF}{\mathbb{F}}
\newcommand{\ZZ}{\mathbb{Z}}
\newcommand{\QQ}{\mathbb{Q}}
\newcommand{\cA}{\mathcal{A}}

\newcommand{\cH}{\mathcal{H}}

\newcommand{\edit}[1]{{\color{red}#1}}

\title{Uncommon linear systems of two equations}

\begin{document}

\author{Dingding Dong}
\address{Dong: Department of Mathematics, Harvard University, Cambridge, MA 02138, USA.}
\email{ddong@math.harvard.edu}

\author{Anqi Li}
\address{Li: Department of Pure Mathematics and Mathematical Statistics (Centre for Mathematical Sciences), University of Cambridge, Cambridge CB30WB, England, UK.}
\email{aqli@stanford.edu}
\author{Yufei Zhao}
\address{Zhao: Department of Mathematics, Massachusetts Institute of Technology, Cambridge, MA 02139, USA.}
\email{yufeiz@mit.edu}

\thanks{Li was supported by the Trinity Studentship in Mathematics. Most of this work was completed while Li was an undergraduate at MIT. Zhao was supported in part by NSF CAREER Award DMS-2044606}

\begin{abstract}
A system of linear equations $L$ is common over $\FF_p$ if, as $n\to\infty$, any 2-coloring of $\FF_p^n$ gives asymptotically at least as many monochromatic solutions to $L$ as a random 2-coloring.

The notion of common linear systems is analogous to that of common graphs, i.e., graphs whose monochromatic density in 2-edge-coloring of cliques is asymptotically minimized by the random coloring. Saad and Wolf initiated a systematic study on identifying common linear systems, built upon the earlier work of Cameron--Cilleruelo--Serra. When $L$ is a single equation, Fox--Pham--Zhao gave a complete characterization of common linear equations. When $L$ consists of two equations, Kam\v{c}ev--Liebenau--Morrison showed that irredundant $2\times 4$ linear systems are always uncommon. In this work, (1) we determine commonness of all $2\times 5$ linear systems up to a small number of cases, and (2) we show that all $2\times k$ linear systems with $k$ even and girth (minimum number of nonzero coefficients of a nonzero equation spanned by the system) $k-1$ are uncommon, answering a question of Kamčev--Liebenau--Morrison.
\end{abstract}

\maketitle

\section{Introduction}

One of the major open problems in extremal graph theory is \emph{Sidorenko's conjecture} \cite{S93}, which roughly states that for bipartite graphs $H$, a random graph $G$ on $n$ vertices asymptotically minimizes the density of $H$ in $G$ among all $G$ with $n$ vertices and fixed edge density.
Sidorenko's conjecture has been verified for various classes of bipartite graphs $H$ (e.g. \cite{BR65,CFS10,CKLL18,CL17,CL21,H10,KLL16,L11,S91,S14} and the references therein), although a complete proof of Sidorenko's conjecture still eludes us. A closely related graph property is the following coloring variant known as \emph{commonness}: a graph $H$ is \emph{common} if the random coloring minimizes the density of monochromatic copies of $H$ in $K_n$ among all possible two-colorings of the edges of $K_n$. An early result of Goodman \cite{Goo59} showed that $K_3$ is common. Motivated by this result, Burr and Rosta \cite{BR80} conjectured that all graphs $H$ are common, which was disproved by Sidorenko (who showed that a triangle with a pendant edge is uncommon) and Thomason \cite{Tho89} (who showed that $K_4$ is uncommon). In fact, any graph containing a $K_4$ is uncommon \cite{CST96}. We do not even have a good guess at what a complete characterization of common graphs might look like.

In this paper, we are interested in an arithmetic analogue (over finite fields) of this notion of commonness. To motivate the setup, we observe that for a given set $A \subseteq \FF_p^{n}$ satisfying $-A = A$, the additive quadruples in $A$ (that is, $(x,y,z,w) \in A^4$ such that $ x+y - z - w = 0$) correspond to 4-cycles in the Cayley graph on $A$ (that is, the graph with vertex set indexed by $\FF_p^n$ where $x \sim y$ if and only if $x - y \in A$). As such, a natural arithmetic analogue of commonness would be to ask for which linear equations in $\FF_p$ is it true that the random two-coloring of $\FF_p^n$ minimizes the density of monochromatic solutions to the linear equation. We can similarly define an arithmetic analogue of Sidorenko's conjecture. The study of such arithmetic analogues was first posed by Saad and Wolf \cite{SW17}. In this paper, we are interested in the commonness of certain linear \emph{systems} of equations. Before we formalize this setup, we begin by introducing some definitions.

\begin{definition}We say that $L$ is an $m\times k$ \textit{linear system} if it is a collection of $m$ linear equations $L_1,\dots,L_m$ on $k$ variables with integer coefficients. In other words,  $L$ is a collection of linear equations
\begin{align*}
    L_1 \mathbf x &=a_{1,1}x_1+\dots+a_{1,k}x_k,\\
    &\vdots\\
    L_m \mathbf x&=a_{m,1}x_1+\dots+a_{m,k}x_k,
\end{align*}
with $\{a_{i,j} \in \ZZ :i\in[m],j\in[k]\}$. 

We note that there is a one-to-one correspondence between $m\times k$ linear systems and matrices in $\ZZ^{m\times k}$; this means we can also represent $L$ as the matrix
\begin{align*}
    L=\begin{pmatrix}
        a_{1,1} & a_{1,2} & \dots & a_{1,k}\\
        \vdots & \vdots & \ddots & \vdots\\
        a_{m,1} & a_{m,2} & \dots & a_{m,k}\\
    \end{pmatrix}\in\ZZ^{m\times k}.
\end{align*}
\end{definition}

We say that an $m \times k $ linear system $L$ is \emph{common over $\FF_p$} if for every 2-coloring of $\FF_p^n$, the number of monochromatic solutions $\mathbf{x} \in \FF_p^n$ to $L \mathbf{x} = 0$ is at least $1-o(1)$ times the expected number of monochromatic solutions given by a random 2-coloring, asymptotically as $n\to\infty$. Below, we will state an equivalent definition as a precise inequality.

In the following, we restrict our attention to \emph{irredundant linear systems}. Roughly speaking, a system $L$ is irredundant if it does not have ``repeated'' equations or variables. 

\begin{definition}Let $L$ be a linear system.
    We say that $L$ is \textit{irredundant} if:
    \begin{enumerate}
        \item the rows of $L$ are linearly independent,
        \item no column of $L$ is the zero vector,
        \item the row span of $L$ does not contain vectors of the form $(0,\dots,0,1,0,\dots,0,-1,0,\dots,0)$.
    \end{enumerate}
    If $L$ is not irredundant, we say that it is \emph{redundant}. 
\end{definition}

Such a restriction is without loss of generality: if $L$ is redundant, then by removing extraneous equations and variables from $L$ we can obtain an irredundant linear system  $L'$ such that $L$ is common if and only if $L'$ is common. Therefore, we restrict our attention to irredundant linear systems for the rest of this paper (see also \cite[Section 2]{KLM21a} for a  discussion on this). By restricting our attention to irredundant linear systems, we can also pass to a ``functional'' analogue of commonness which we introduce next.

Earlier on, we defined the commonness property of a linear system $L$ in terms of minimizing the number of monochromatic solutions to $L\mathbf{x} = 0$ in a 2-coloring of $\FF_p^n$. It turns out that there is a ``functional'' variant of this definition which is more convenient for us to work with. This is the version of the definition of common and Sidorenko systems that we now state.
We refer the interested reader to \cite[Section 2]{FPZ21} for a proof of the equivalence of the functional and set definitions of the properties of being common and Sidorenko.

\begin{definition}\label{def:common-sid}

Let $L$ be an irredundant $m\times k$ linear system and $p$ be prime. For every $n\in\NN$ and  $f:\FF_p^n \to\RR$, define the \textit{density function}
\[
t_L(f)=\EE_{\mathbf x \in \FF_p^n : L\mathbf x=0}f(x_1)\cdots f(x_k).
\]
\begin{enumerate}
    \item[(1)] We say that $L$ is \emph{common} over $\FF_p$ if for all $n\in\NN$ and $f\colon \FF_p^n\to[0,1]$, we have \[
    t_L(f)+t_L(1-f)\geq 2^{1-k}.
    \]
    If this does not hold, then we say that $L$ is \emph{uncommon} over $\FF_p$. 
\item[(2)] We say that $L$ is \emph{Sidorenko} over $\FF_p$ if for all $n\in\NN$ and $f\colon \FF_p^n\to[0,1]$, we have 
\[
t_L(f)\geq (\EE f)^k.
\]
\end{enumerate}
\end{definition}

\begin{remark}
If $L$ is Sidorenko over $\FF_p$, then it is common over $\FF_p$.
Indeed, if $L$ is Sidorenko, then $t_L(f) + t_L(1-f) \ge (\EE f)^k + (1-\EE f)^k \ge 2^{-k+1}$ by convexity.
\end{remark}

Several recent works address the question of characterizing common and Sidorenko linear systems.
When $L$ is itself a single linear equation (i.e., $m=1$),  the \textit{length} of $L$ is defined to be the number of its nonzero coefficients.
Fox, Pham and Zhao \cite{FPZ21}  showed that a linear equation $L$ is Sidorenko if and only if its coefficients can be partitioned into cancelling pairs; furthermore, it is common and not Sidorenko if and only if it has odd length. Their results give a complete characterziation of common and Sidorenko linear equations. 

No complete characterization is known for linear systems of higher rank. Some partial results in this direction include the following: Versteegen \cite{Ver21} proved that 4-AP is uncommon -- this establishes the arithmetic analogue of Thomason's result on the uncommonness of $K_4$. Kam\v{c}ev, Liebenau and Morrison proved that if a linear system consists of Sidorenko linear equations aligned in a ``tree template'', then it is Sidorenko \cite{KLM21b}; they also proved that any linear system that has an irredundant $2\times 4$  subsystem is uncommon \cite{KLM21a}. Altman \cite{Alt22a} explored the notion of local Sidorenko property and showed that any linear system of two equations, with every $2\times 2$ minor invertible, is not Sidorenko. On the other hand, Altman \cite{Alt22b} also established that there is a non-Sidorenko linear system that becomes common after adding sufficiently many free variables, answering a question of \cite{SW17}.

In this paper, we determine commonness of all $2\times 5$ linear systems up to a small number of cases. As a byproduct, we answer \cite[Question 5.2]{KLM21b} by constructing an uncommon $2\times 5$ linear system that contains an additive quadruple. We also prove \cite[Conjecture 6.1]{KLM21a} that all $2\times k$ linear systems, with $k$ even and ``girth" (which we define in \cref{def:girth}) $k-1$, are uncommon over $\FF_p$ for all sufficiently large $p$. Before we formally state our results, we need to introduce some definitions.

\begin{definition}\label{def:subsystem}
Let $L$ be an $m\times k$ linear system on variables $x_1,\dots,x_k$, and $L'$ be an $m'\times k'$ linear system on a subset of the variables $x_{i_1},\dots,x_{i_{k'}}$. Write $\mathbf x=(x_1,\dots,x_k)$ and $\mathbf x'=(x_{i_1},\dots,x_{i_{k'}})$.

\begin{enumerate}
    \item We say that $L'$ is a \textit{subsystem} of $L$ if $L \mathbf x =0$ implies $L' \mathbf x' = 0$.
    
    \item Suppose $k'=k$ (so $\mathbf x'=\mathbf x$). We say that $L$ and $L'$ are \textit{isomorphic} if we can permute the columns of $L'$ so that the obtained matrix has the same row span as that of $L$.
\end{enumerate}
\end{definition}

\begin{example}The four-term arithmetic progression
\begin{align*}
    L'(x_1,x_2,x_3,x_4)=\begin{pmatrix} 1 & -2 & 1 & 0 \\ 0 & 1 & -2 & 1 \end{pmatrix}
\end{align*} 
is a subsystem of the five-term arithmetic progression
\begin{align*}
    L(x_1,x_2,x_3,x_4,x_5)=\begin{pmatrix} 1 & -2 & 1 & 0 & 0 \\ 0 & 1 & -2 & 1 & 0 \\ 0 & 0 & 1 & -2 & 1 \end{pmatrix}.
\end{align*}

\end{example}

Finally, we recall that a graph with odd girth is not Sidorenko (see for instance \cite[Theorem 3.1]{CLS23}). In \cite{KLM21b}, the authors defined a natural notion of \emph{girth} for linear systems and proved a parallel result that a linear system with odd girth is not Sidorenko. We next give the definition of \emph{girth}  for a linear system.

\begin{definition}\label{def:girth}
    Let $L\in\ZZ^{m\times k }$ be an $m\times k$ linear system. 
    The \textit{girth} of $L$, denoted as $s(L)$, is the minimum possible size of the  support of a nonzero vector in the row span of $L$. 
\end{definition}

We are now ready to formally state our main results. While odd girth was enough to guarantee that a linear system $L$ is not Sidorenko, we cannot expect it alone to give the stronger conclusion that $L$ is uncommon: we have seen earlier that single, odd-length equations are common. But what if we restrict our attention to systems of at least two equations? Would we be able to deduce the uncommonness of such a system from its girth? Our first result gives a partial answer to this question and shows that all two-equation linear systems with even number of variables and maximum girth are uncommon.

\begin{theorem}\label{thm:2*k-even}
    Suppose $k\geq 4$ is even. Let $L$ be an irredundant $2\times k$ system with $s(L)=k-1$. Then there exists $p_L$ such that $L$ is uncommon over $\FF_p$ for all $p\geq p_L$.
\end{theorem}

By combining the proof strategy of \cref{thm:2*k-even} and some additional ideas from  \cite{KLM21a} and \cite{Alt22a}, we obtain the following more general result, which is also a generalization of \cite[Corollary 1.2]{KLM21a}.

\begin{theorem}\label{thm:2*k-even-ext}
    Suppose $k\geq 4$ is even. Let $L$ be a linear system, with $s(L)=k-1$, that has an irredundant $2\times k$ subsystem. Then there exists $p_L$ such that $L$ is uncommon over $\FF_p$ for all $p\geq p_L$.
\end{theorem}

We remark that an independent proof of \cref{thm:2*k-even,thm:2*k-even-ext} was given by Altman and Liebenau \cite{AL24}.

Next, one can ask about the commonness of $2 \times k$ linear systems for $k$ odd. This turns out to be much more subtle than the $k$ even case; even for $k=5$, there are common $2 \times 5$ linear systems with $s(L) = 3$ and $s(L) = 4$. In fact, for each $2\times 5$ linear system that is not one of the two unknown cases, we are able to determine whether it is common over $\FF_p$ for all $p$, or it is uncommon over $\FF_p$ for all sufficiently large $p$.

\begin{theorem}\label{thm:2*5}
    Let $L$ be an irredundant $2\times 5$ linear system that is not isomorphic to
    \begin{align*}
        \begin{pmatrix}
    1 & 1 & -1 & -1 & 0\\
    1 & -1 & 3 & 0 & -3
\end{pmatrix}\quad\text{or}\quad \begin{pmatrix}
            1 & 1 & -1 & -1 & 0\\
            2 & -2 & 3 & 0 & -3
        \end{pmatrix}.
    \end{align*}
    \begin{enumerate}
        \item If $L$ is isomorphic to 
        \[
        \begin{pmatrix}
            1 & 0 & -1 & 2 & -2\\
            0 & 1 & 2 & -1 & -2
        \end{pmatrix}
        \text{
        or\, }
        \begin{pmatrix}
            1 & -1 & 1 & -1 & 0\\
            1 & 2 & -1 & 0 & -2
        \end{pmatrix}
        \text{
        or\, }
        \begin{pmatrix}
            1 & 1 & -1 & -1 & 0\\
            2 & -2 & 1 & 0 & -1
        \end{pmatrix}
        \]
        or
        \begin{align*}
            &\begin{pmatrix}
            a & b & 0 & 0 & c\\
            0 & 0 & a & b & c
        \end{pmatrix}\qquad \text{with $a,b,c\in\ZZ\setminus\{0\}$}
        \end{align*}
        or

        \begin{align*}
            &
        \begin{pmatrix}
        c & -c & d & -d & 0\\
        a & b & 0 & 0 & -a-b
    \end{pmatrix}
    \quad \text{with $a,b,a+b,c,d\in\ZZ\setminus\{0\}$}\\
    &\quad \text{such that }\{|a/b|,|a/(a+b)|,|b/(a+b)|\}\cap \{|c/d|,|d/c|\}\neq\emptyset,
        \end{align*}
        
        \noindent then $L$ is common over $\FF_p$ for all primes $p$.
        \item If $L$ is not isomorphic to any of the above, then there exists $p_L$ such that $L$ is uncommon over $\FF_p$ for all $p\geq p_L$.
    \end{enumerate}
\end{theorem}

The proofs of uncommonness in \cref{thm:2*k-even}, \cref{thm:2*k-even-ext} and \cref{thm:2*5} utilize suitably designed \emph{Fourier templates} which we introduce in \cref{sec:fourier}; these Fourier templates can be thought of as functions with suitably chosen Fourier coefficients  that certificate uncommonness over $\FF_p$ for sufficiently large $p$. We are unable to construct a suitable Fourier template for the two unknown cases in \cref{thm:2*5}, although we believe that our main constraints are essentially computational barriers which limit the size of the support of the Fourier templates we are able to consider. It is also plausible that the systems are actually common, although we are also unable to prove that. 

\begin{question}
    Are the two systems not covered by \cref{thm:2*5} common or uncommon?
\end{question}

There also remains the interesting question of classifying $2 \times k$ systems for $k$ odd and $k \geq 7$, and it is natural to ask if there is a general way to do this classification for all odd $k$ at once. 

\begin{question}
    What is the classification of the commonness of $2 \times k$ systems for $k \geq 7$ odd?
\end{question}

\paragraph{\textbf{Outline}.} In Section 2, we introduce the Fourier analytic approach to solve problems on identifying common and Sidorenko linear systems. In Section 3, we define the notion of a Fourier template and demonstrate its utility by proving  (\cref{thm:fourier-template}) that if we can find an appropriate Fourier template corresponding to a given linear system $L$, then $L$ is uncommon over all $\FF_p$ for sufficiently large $p$. Next, we prove \cref{thm:2*k-even} and \cref{thm:2*k-even-ext} in Section 4 and \cref{thm:2*5} in Sections 5 and 6. We begin each of Sections 4 (\cref{app:Even-intro}) and 5 (\cref{subsec:5-intro}) with some examples to illustrate the key proof ideas behind our construction of the relevant Fourier templates. 

\medskip

\paragraph{\textbf{Acknowledgements}.}
We thank Yunkun Zhou for providing the proof that the linear system
\[
\begin{pmatrix}
            1 & 1 & -1 & -1 & 0\\
            2 & -2 & 1 & 0 & -1
        \end{pmatrix}
\]
is common. We thank Dylan Toh for suggesting the greatly simplified proof of \cref{fact:cos-1}. We thank Pablo Parrilo for providing suggestions and help on symmetry-adapted semi-definite programs. We also thank Nitya Mani and Mark Saengrungkongka for helpful discussions.

\section{Preliminaries}

\subsection{Fourier analytic approach}

In this section, we discuss the Fourier analytic approach for identifying common and Sidorenko linear systems. Let $\widehat {\FF_p^n}$ denote the set of homomorphisms from $\FF_p^n\to \CC^\times$. Since every element in $\widehat {\FF_p^n}$ is of the form $x\mapsto e^{2\pi ir\cdot x/p}$,  we have an isomorphism  between $\FF_p^n$ and $ \widehat {\FF_p^n}$. For a function $f:\FF_p^n\to\RR$, the Fourier transform of $f$ is the function $\widehat f:\widehat{\FF_p^n}\to\CC$ defined by
\[
\widehat f(r)=\EE_{x\in \FF_p^n}f(x)\overline{r(x)}
\]
with inverse Fourier transform
\[
f(x)=\sum_{r\in \widehat{\FF_p^n}}\widehat f(r)r(x).
\]
Thus, we can rewrite the density function $t_L(f)$ as
\begin{align*}
    t_L(f)&=\EE_{L \mathbf x =0}\sum_{r_1,\dots,r_k\in \widehat{\FF_p^n}}\widehat f(r_1)\cdots \widehat f(r_k)r_1(x_1)\cdots r_k(x_k)
    =\sum_{r_1,\dots,r_k\in \widehat{\FF_p^n}}\widehat f(r_1)\cdots \widehat f(r_k)\EE_{L \mathbf x =0}r_1(x_1)\cdots r_k(x_k).
\end{align*}
Now observe that $\EE_{L \mathbf x =0}r_1(x_1)\cdots r_k(x_k)=\EE_{L \mathbf x =0}(e^{2\pi i/p})^{r_1\cdot x_1+\dots+r_k\cdot x_k}$ equals 1 if $(r_1,\dots,r_k)$ is a linear combination of the row vectors of $L$ \edit{}, and equals 0 otherwise. Let $\vec\ell_1,\dots,\vec\ell_m\in\ZZ^k$ be the row vectors of $L$ and $\vec a_1,\dots,\vec a_k\in\ZZ^m$ the column vectors of $L$. We know that 
\[
\text{span}(\vec\ell_1,\dots,\vec\ell_m)=\{( \vec a_1\cdot  r,\dots,  \vec a_k\cdot  r): r\in (\widehat{\FF_p^n})^m\}.
\]
Hence we have
\begin{align}\label{eqn:inversion1}
    t_L(f)&=\sum_{ r\in(\widehat{\FF_p^n})^m}\widehat f(\vec a_1\cdot   r)\cdots\widehat f(\vec a_k\cdot   r).
\end{align}

In particular, when $L$ itself is a linear equation $L\mathbf x=a_1x_1+\dots+a_kx_k$, we can rewrite the density function $t_L(f)$  as
\begin{align}\label{eqn:inversion2}
t_L(f)=\sum_{r\in\widehat{\FF_p^n}}\widehat f(a_1r)\cdots\widehat f(a_kr).
\end{align}

\subsection{Reduction to critical sets}

In this section, we define critical sets of a linear system and then introduce a reductive step in \cite{KLM21a} that will be useful in our proof. The following two definitions can be found in \cite[Section 2.2]{KLM21a}.

\begin{definition}\label{def:critical-1}
For every linear system $L$ with girth $s(L)$, define
\[
c(L)=\begin{cases}
    s(L) & s(L)\text{ even}\\
    s(L)+1 & s(L)\text{ odd}.
\end{cases}
\]
We say that $B\subseteq[k]$ is a \textit{critical set} of $L$ if 
\begin{enumerate}
    \item $|B|=c(L)$, and
    \item $L$ has a nonempty subsystem $L'$ whose variable set is $\{x_i:i\in B\}$.
\end{enumerate}
Let $\mathcal C(L)$ denote the collection of critical sets of $L$.
\end{definition}

\begin{definition}\label{def:critical-2}
    Let $L$ be a linear system and $B\subseteq [k]$ be a criticial set of $L$. 
    \begin{enumerate}
        \item Define 
        \begin{align*}
            m_B=\max\{m':&\text{ $L$ has an $m'\times c(L)$ subsystem  on variable set $\{x_i:i\in B\}$}\}.
        \end{align*}
        \item Define $L_B$ to be the (unique up to isomorphism) $m_B\times c(L)$ subsystem of $L$ on variable set $\{x_i:i\in B\}$. 
    \end{enumerate}
    The proof of the uniqueness of $L_B$ is standard linear algebra, and we refer the readers to \cite[Lemma 2.4]{KLM21a}.
\end{definition}

\begin{remark}\label{rem:critical-3}
   Since no equation of length $<s(L)$ can be a subsystem of $L$, we have the following:
    \begin{itemize}
        \item When $s(L)$ is even (so $c(L)=s(L)$), every $B\in\mathcal C(L)$ has $m_B=1$, and $L_B$ must be a length-$c(L)$ equation.
        \item When $s(L)$ is odd (so $c(L)=s(L)+1$), every $B\in\mathcal C(L)$ has $m_B\in\{1,2\}$, and $L_B$ is either a length-$c(L)$ equation or a $2\times c(L)$ linear system with girth $s(L)$.
    \end{itemize}
\end{remark}

One of the results in \cite{KLM21a} states that to conclude a linear system $L$ is uncommon, it sometimes suffices to study the collection $\{L_B:B\in\mathcal C(L)\}$.

\begin{proposition}[\text{\cite[Theorem 3.1]{KLM21a}}]
\label{prop:critical-uncommon}
    Let $L$ be an $m\times k$ linear system. If there exists $f:\FF_p^n\to\RR$ such that $\EE f=0$ and
    $\sum_{B\in \mathcal C(L)}t_{L_B}(f)<0$,
    then $L$ is uncommon over $\FF_p$.
\end{proposition}
\begin{proof}
Consider any  $h:\FF_p^n\to[0,1]$ with $\EE h=1/2$, and define  $f=h-1/2:\FF_p^n\to[-1/2,1/2]$. Expanding each term, one can check that
\begin{align*}
    &t_L(h)+t_L(1-h)\\
    &=\EE_{L\mathbf x=0}(1/2+f(x_1))\cdots (1/2+f(x_k))+ \EE_{L\mathbf x=0}(1/2-f(x_1))\cdots (1/2-f(x_k))\\
    &=2^{1-k}+2\sum_{\substack{\emptyset\neq B\subseteq[k]\\|B|\text{ even}}}2^{|B|-k}\EE_{L\mathbf x=0}\left[\prod_{i\in B}f(x_i)\right]
    =2^{1-k}+2\sum_{\substack{ B\subseteq[k]\\|B|\geq c(L)\text{ even}}}2^{|B|-k}\EE_{L\mathbf x=0}\left[\prod_{i\in B}f(x_i)\right],
\end{align*}
where the last equality follows from the fact that $\EE f=0$, and $L$ does not have nonempty subsystems on at most $ c(L)-2$ variables.
Thus, if there is some $f:\FF_p^n\to[-1/2,1/2]$ such that $\EE f=0$ and
\begin{align*} 
    \sum_{\substack{ B\subseteq[k]\\|B|\geq c(L)\text{ even}}}2^{|B|-k}\EE_{L\mathbf x=0}\left[\prod_{i\in B}f(x_i)\right]<0,
\end{align*}
then by \cref{def:common-sid}, $L$ is uncommon over $\FF_p$.

Observe that
\begin{align*} 
    &\sum_{\substack{ B\subseteq[k]\\|B|\geq c(L)\text{ even}}}2^{|B|-k}\EE_{L\mathbf x=0}\left[\prod_{i\in B}f(x_i)\right]\\
    &=\sum_{B\in\mathcal C(L)}2^{c(L)-k}\EE_{L\mathbf x=0}\left[\prod_{i\in B}f(x_i)\right]+\sum_{\substack{ B\subseteq[k]\\|B|\geq c(L)+2\text{ even}}}2^{|B|-k}\EE_{L\mathbf x=0}\left[\prod_{i\in B}f(x_i)\right].
\end{align*}
Consider any critical set $B\in \mathcal C(L)$ and its corresponding subsystem $L_B$. A linear algebra calculation (see also \cite[Lemma 2.6]{KLM21a}) shows that every solution to $L_B$ extends to the same number of solutions to $L$. That is, for every solution $\mathbf x'$ to $L_B\mathbf x'=0$, there are $(p^n)^{k-m-c(L)+m_B}$ solutions $\mathbf x$ to $L\mathbf x=0$ such that $\mathbf x$ restricted  on the variables $\{x_i:i\in B\}$ equals $\mathbf x'$.
Thus we get that 
$$\EE_{L\mathbf x=0}\left[\prod_{i\in B}f(x_i)\right]=\EE_{L_B\mathbf x'=0}\left[\prod_{i\in B}f(x_i)\right]=t_{L_B}(f).$$

Suppose there exists $\widetilde f:\FF_p^n\to\RR$ with $\EE \widetilde f=0$ and
$\sum_{B\in \mathcal C(L)}t_{L_B}(\widetilde f)<0$. Taking $f=\epsilon \widetilde f$ with $\epsilon>0$ sufficiently small, we can obtain $f:\FF_p^n\to [-1/2,1/2]$ with
\begin{align*}
    &\sum_{B\in\mathcal C(L)}2^{c(L)-k}\EE_{L\mathbf x=0}\left[\prod_{i\in B}\epsilon \widetilde f(x_i)\right]+\sum_{\substack{ B\subseteq[k]\\|B|\geq c(L)+2\text{ even}}}2^{|B|-k}\EE_{L\mathbf x=0}\left[\prod_{i\in B}\epsilon \widetilde f(x_i)\right]\\
    &=2^{c(L)-k}\epsilon^{c(L)}\sum_{B\in \mathcal C(L)}t_{L_B}(\widetilde f)+O(\epsilon^{c(L)+2}),
\end{align*}
which is negative. Hence $L$ is uncommon over $\FF_p$.
\end{proof}

\section{Fourier template}\label{sec:fourier}

As discussed in Section 2, for linear system $L$ and function $f:\FF_p\to\RR$, we can write the solution density  $t_L(f)$ in terms of the Fourier coefficients of $f$. In order to construct a function $f $ that certifies uncommonness, it is often fruitful to  view the problem in the Fourier space and construct a suitable $\widehat{f} \colon \widehat{\FF_p} \to \CC$.
We will introduce a notion of a \emph{Fourier template} which can then be turned into some $\widehat{f} \colon \widehat{\FF_p} \to \CC$ for all sufficiently large $p$.

\begin{example}\label{ex:fourier-template}
    As an example, to show that the linear equation $$L \mathbf x=x_1+2x_2+3x_3+4x_4$$ is uncommon over $\FF_p$, we can consider the function  $f:\FF_p\to \RR$ defined by
\begin{align*}
    \begin{cases}
        \widehat f(1)=\widehat f(-1)=-1\\
        \widehat f(2)=\widehat f(-2)=\widehat f(3)=\widehat f(-3)=\widehat f(4)=\widehat f(-4)=1\\
        \widehat f(r)=0\qquad \text{otherwise.}
    \end{cases}
\end{align*}
Notice that $\EE f=\widehat f(0)=0$. Moreover, for all sufficiently large $p$, we have
\begin{align*}
    t_L(f)=\sum_{r\in \widehat{\FF_p}}\widehat f(r)\widehat f(2r)\widehat f(3r)\widehat f(4r)=\widehat f(1)\widehat f(2)\widehat f(3)\widehat f(4)+\widehat f(-1)\widehat f(-2)\widehat f(-3)\widehat f(-4)=-2<0.
\end{align*}
By taking $\epsilon>0$ sufficiently small such that $\epsilon f:\FF_p\to[-1/2,1/2]$ and letting $h=1/2+ \epsilon f$, we can obtain a function $h:\FF_p\to[0,1]$ such that
\begin{align*}
    &t_L(h)+t_L(1-h)\\
    &=\EE_{L\mathbf x=0}[(1/2+\epsilon f(x_1))\cdots (1/2+\epsilon f(x_4))]+ \EE_{L\mathbf x=0}[(1/2-\epsilon f(x_4))\cdots (1/2-\epsilon f(x_4))]\\
    &=1/8+2\epsilon^4\EE_{L\mathbf x=0}[f(x_1)\cdots f(x_4)]=1/8+2\epsilon^4t_L(f)<1/8.
\end{align*}
Therefore, $L$ is uncommon over $\FF_p$ for all sufficiently large $p$.
\end{example}

The key step in \cref{ex:fourier-template} is to construct an appropriate Fourier transform $\widehat f:\widehat{\FF_p}\to\CC$ such that $t_L(f)<0$. Since $\widehat f$ has finite support, this construction  works for all sufficiently large $p$. What we  essentially did is creating a finitely supported map $g:\ZZ\to\CC$ from the \textit{integers}, defined by $\pm 1\mapsto -1$ and $\pm 2,\pm 3,\pm 4\mapsto 1$, such that the Fourier transform $\widehat f=g$ (which is well-defined for all sufficiently large $p$) gives us a function $f$ that helps certify uncommonness. This map $g:\ZZ\to \CC$ is what we call a 1-dimensional Fourier template.

We now formally introduce the definition of Fourier templates.

\begin{definition}\label{def:fourier-template}
    For every $m\times k$ linear system $L$ and $d\in\NN$, we say that $g:\ZZ^d\to\CC$ is a \textit{$d$-dimensional Fourier template} if the following holds:
    \begin{itemize}
        \item $g$ has finite support,
        \item $g(-r)=\overline{g(r)}$ for all $r\in\ZZ^d$,
        \item $g(0)=0$.
    \end{itemize}
\end{definition}

\begin{definition}\label{def:fourier-template-sum}
 Let $g:\ZZ^d\to\CC$ be a $d$-dimensional Fourier template, and
 \[
 L=\begin{pmatrix}
        \vec a_1 & \vec a_2 & \cdots & \vec a_k
    \end{pmatrix}=\begin{pmatrix}
        a_{1,1} & a_{1,2} & \dots & a_{1,k}\\
        \vdots & \vdots & \ddots & \vdots\\
        a_{m,1} & a_{m,2} & \dots & a_{m,k}\\
    \end{pmatrix}\in\ZZ^{m\times k}
 \]
be an $m\times k$ linear system. (Here $\vec a_1,\dots,\vec a_k\in\ZZ^m$ are the columns of $L$.)
\begin{enumerate}
    \item Define the associated function $g_L:\ZZ^{d\times m}\to\CC$ by
 \[
 g_L(R)=g(R\vec a_1)\cdots g(R\vec a_k)=\prod_{i=1}^kg(R\vec a_i).
 \]
 \item Define $\sigma_L(g)=\sum_{R\in\ZZ^{d\times m}}g_L(R)$.
\end{enumerate}
\end{definition}

\begin{remark}\label{rem:real-sum}
    Notice that $\sigma_L(g)$ is well-defined as $g$ has finite support. Moreover, since $$\sigma_L(g)=\sum_{R\in\ZZ^{d\times m}}g_L(R)=\sum_{R\in\ZZ^{d\times m}}g_L(-R)=\sum_{R\in\ZZ^{d\times m}}\overline{g_L(R)}=\overline{\sigma_L (g)},$$ we always have $\sigma_L(g)\in\RR$.
\end{remark}

We now state and prove the main theorem  of this section.

\begin{theorem}\label{thm:fourier-template}
Let $L$ be an $m\times k$ linear system. If there exists a $d$-dimensional Fourier template $g:\ZZ^d\to\CC$ such that $\sum_{B\in\mathcal C(L)} \sigma_{L_B}(g)<0$,
then $L$ is uncommon over $\FF_p$ for all sufficiently large $p$.
\end{theorem}

When $d=1$, the proof of \cref{thm:fourier-template} essentially goes as \cref{ex:fourier-template}. When $d>1$, the main idea is to construct an appropriate Freiman homormophism  $\gamma:\mathbb{Z}^d\to \mathbb{Z}$ that projects  $\ZZ^d$ onto a generalized arithmetic progression in $\ZZ$. We will then build the certificate $f:\FF_p\to\CC$  by setting  $\widehat{f}(r)=g(\gamma^{-1}(r))$ for all $r\in\text{im}\gamma\subseteq \widehat{\FF_p}$.

\begin{proof}
Consider any critical set $B\in \mathcal C(L)$. Suppose
\[L_B=\begin{pmatrix}
        \vec a_1  & \cdots & \vec a_{k'}
    \end{pmatrix}=
\begin{pmatrix}
    a_{1,1}x_1 & \dots & a_{1,k'}x_{k'}\\
    \vdots &\ddots &\vdots\\
    a_{m',1}x_1 & \dots & a_{m',k'}x_{k'}
\end{pmatrix}\in\ZZ^{m'\times k'},
\]
so that $\vec a_{1},\dots,\vec a_{k'}\in\ZZ^{m'}$ are the column vectors of $L_B$. Recall from \cref{rem:critical-3} that we have $m'\in\{1,2\}$ and $k'=c(L)$.

Pick matrix $Q_B\in \ZZ^{(k'-m')\times k'}$ so that the rows of $L_B$ span the kernel of $Q_B$. Then for $f:\FF_p\to\RR$, we have (recall \cref{eqn:inversion1})
\begin{align}\label{eqn:freiman}
    t_{L_B}(f)&=\sum_{ r\in (\widehat {\FF_p})^{m'} }\widehat {f}(\vec  a_1\cdot  r)\cdots \widehat {f}(\vec  a_{k'}\cdot r)
    =\sum_{\substack{ \vec v\in (\widehat{\FF_p})^{k'}  \\Q_B  \vec v=0}}\widehat {f}( v_1)\cdots \widehat {f}( v_{k'}).
\end{align}

For $M\in\NN$, let $[\![-M,M]\!]$ denote the subset $\{-M,\dots,-1,0,1,\dots,M\}\subseteq\ZZ$.
Suppose the Fourier template $g:\ZZ^d\to\CC$ vanishes outside $[\![-M,M]\!]^d$. Observe that, for  $C\in\NN$ sufficiently large, the projection $\gamma:[\![-M,M]\!]^d\to \ZZ$ defined by
\[
\gamma(w_0,\dots,w_{d-1})=w_0+w_1C+\dots+w_{d-1}C^{d-1}
\]
is injective. Moreover, the inverse $\phi:=\gamma^{-1}:\text{im}\gamma\to[\![-M,M]\!]^d$ satisfies the following property: for all $B\in\mathcal C(L)$ and $\vec v\in(\text{im}\gamma)^{k'}$, we have $Q_B  \vec v=0\in\ZZ^{k'-m'}$ if and only if  $Q_B\phi  (\vec v)=0\in\ZZ^{(k'-m')\times d}$, where
\begin{align*}
    \phi  (\vec v)
    =\begin{pmatrix}
        \phi (v_1)\\
        \vdots\\
        \phi (v_{k'})
    \end{pmatrix}\in\ZZ^{k'\times d}.
\end{align*}

Suppose $p$ is sufficiently large. Construct $f:\FF_p\to \RR$ as follows:
\begin{itemize}
    \item For all $-p/2< r<p/2$ with $r\in \text{im} \gamma$, set $\widehat {f}(r)=g(\phi r)$.
    \item Otherwise, set $\widehat {f} (r)=0$.
\end{itemize}
By \cref{eqn:freiman}, we have
\begingroup
\addtolength{\jot}{1em}
\begin{align*}
     t_{L_B}(f)&=\sum_{\substack{ v\in (\widehat{\FF_p})^{k'}  \\Q_B \vec v=0}}\widehat {f}( v_1)\cdots \widehat {f}( v_{k'})=\sum_{\substack{  v\in (\widehat{\FF_p})^{k'}\\Q_B\phi   (\vec v)=0}}\widehat {f}( v_1)\cdots \widehat {f}( v_{k'})=\sum_{\substack{U\in \ZZ^{k'\times d}\\Q_{B}U=0}}g(U_1)\cdots g( U_{k'}),
\end{align*}
\endgroup
where $U_1,\dots,U_{k'}\in\ZZ^d$ are the row vectors for each $U\in\ZZ^{k'\times d}$.
Since $Q_BU=0$ if and only if
\begin{align*}
    U=\begin{pmatrix}
        \vec a_1\cdot  r_1 &\dots & \vec a_1\cdot  r_d\\
        \vdots &\ddots&\vdots\\
        \vec a_{k'}\cdot  r_1 &\dots & \vec a_{k'}\cdot  r_d
    \end{pmatrix} \text{ for } r_1,\dots, r_d\in\ZZ^{m'},
\end{align*}
we get that
\begin{align*} 
     t_{L_B}(f)
     &=\sum_{\substack{r_1,\dots, r_d\in \ZZ^{m'}}}g(\vec  a_1\cdot r_1,\dots, \vec a_{1}\cdot  r_d)\cdots g(\vec a_{k'}\cdot  r_1,\dots, \vec a_{k'}\cdot r_d)
     \\&=\sum_{R\in\ZZ^{d\times m'}}g(R\vec a_1)\cdots g(R\vec a_{k'})=\sigma_{L_B}(g).
\end{align*}
Therefore, if $\sum_{B\in\mathcal C(L)} \sigma_{L_B}(g)<0,$ then we have
$
\sum_{B\in\mathcal C(L)}t_{L_B}(f)=\sum_{B\in\mathcal C(L)}\sigma_{L_B}(g)<0
$
for all sufficiently large $p$.
By \cref{prop:critical-uncommon}, $L$ is uncommon over $\FF_p$.
\end{proof}

Finally, we remark that we can join two Fourier templates $g,\widetilde g$ to form a new Fourier template $h$. Moreover, the joined Fourier template satisfies $\sigma_{L}(h)= \sigma_{L}(g)\sigma_{L}(\widetilde g)$.
\begin{proposition}\label{prop:join}
Let $g:\ZZ^{d_1}\to\CC$, $\widetilde{g}:\ZZ^{d_1}\to\CC$ be two Fourier templates. Define the map $h:\ZZ^{d_1+d_2}\to\CC$ by $$h(r_1,r_2)=g( r_1)\widetilde{g}( r_2)\qquad\text{for all $r_1\in\ZZ^{d_1}$, $ r_2\in\ZZ^{d_2}$.}$$ Then $h$ is a $(d_1+d_2)$-dimensional Fourier template. Moreover, for every linear system $L$, we have
\[
\sigma_{L}(h)= \sigma_{L}(g)\sigma_{L}(\widetilde{g}).
\]
\end{proposition}
\begin{proof}
    One can check directly from \cref{def:fourier-template} that $h$ is a $(d_1+d_2)$-dimensional Fourier template. 
    Let 
    \[
    L=\begin{pmatrix}
        \vec  a_1&\dots &\vec a_k
    \end{pmatrix}\in\ZZ^{m\times k} 
    \]
    be an $m\times k$ linear system.
    For  matrix
    $R\in\ZZ^{(d_1+d_2)\times m}$, write $$R=\begin{pmatrix}R_1\\R_2\end{pmatrix}$$ where $R_1\in \ZZ^{d_1\times m}$ and $R_2\in \ZZ^{d_2\times m}$. We then have $R\vec a_i=(R_1\vec a_i,R_2\vec a_i)$ and $h(R\vec a_i)=g(R_1\vec a_i)\widetilde g(R_2\vec a_i)$ for every $i\in[k]$, which gives
    \begin{align*}
        \sigma_L(h)&=\sum_{R\in\ZZ^{(d_1+d_2)\times m}} h_L(R)
        =\sum_{R\in\ZZ^{(d_1+d_2)\times m}}  h(R\vec a_1)\cdots h(R\vec a_k)\\
        &=\sum_{R_1\in\ZZ^{d_1\times m}} \sum_{R_2\in\ZZ^{d_2\times m}} g(R_1\vec a_1)\cdots g(R_1\vec a_k) \widetilde{g}(R_2\vec a_1)\cdots \widetilde{g}(R_2\vec a_k)\\
        &=\sum_{R_1\in\ZZ^{d_1\times m}}g(R_1\vec a_1)\cdots g(R_1\vec a_k)\sum_{R_2\in\ZZ^{d_2\times m}}\widetilde{g}(R_2\vec a_1)\cdots \widetilde{g}(R_2\vec a_k)
        =\sigma_L(g)\sigma_L(\widetilde{g}).
    \end{align*}
\end{proof}

\section{Proof of \cref{thm:2*k-even}}

\subsection{Introduction}
\label{app:Even-intro}

In this section, we prove \cref{thm:2*k-even} and \cref{thm:2*k-even-ext}.  We prove \cref{thm:2*k-even} by showing that, for all $2\times k$ linear systems $L$ with $s(L)=k-1$, there exists a Fourier template $g$ that satisfies the condition in \cref{thm:fourier-template}.

The existence of such Fourier template will be given by a probabilistic argument. Recall that for every $M\in\NN$, we let $[\![-M,M]\!]$ denote the subset $\{-M,\dots,-1,0,1,\dots,M\}\subseteq\ZZ$. Define the random Fourier template $f_M\colon \ZZ \to \CC$ by \[\begin{cases} f_M(r) = \xi_r &r \in [\![-M,M]\!]\setminus\{0\} \\ f_M(r) = 0 &\text{otherwise,}  \end{cases}\]
where every $\xi_r$ is an independent random unit complex number, subject to the constraint that $\xi_r = \overline{\xi_{-r}}$. 

We say that two numbers form a \textit{cancelling pair} if they sum up to zero. The key observation is that, for all $a_1,\dots,a_t\in\ZZ$, we have
\begin{align*}
    \EE[f_M(a_1)\cdots f_M(a_t)]=\begin{cases}
        1 &a_1,\dots,a_t\in [\![-M,M]\!]\setminus\{0\}\\
         &\text{and $a_1,\dots,a_t$ can be partitioned into cancelling pairs}\\
         0& \text{otherwise.}
    \end{cases}\tag{$\dagger$}
\end{align*}

A result in \cite{FPZ21} showed that  any even length equation whose coefficients cannot be partitioned into cancelling pairs  is uncommon over sufficiently large $\FF_p$. This can be proved using a first moment argument on the above $f_M$: if $L=a_1x_1+\dots+a_kx_k$
is such an equation, then for every $r\in\ZZ$, we have $\EE[f_M(a_1r)\dots f_M(a_kr)]=0$ as $a_1,\dots,a_k$ cannot be partitioned into cancelling pairs. Thus,  for sufficiently large $M$, we have $\EE[\sigma_{L}(f_M)]=\EE[\sum_{r\in\ZZ}f_M(a_1r)\dots f_M(a_kr)]= 0$ while $\PP[\sigma_{L}(f_M)>0]>0$, which means that there is a choice of $f_M$ with $\sigma_{L}(f_M)<0$.

When $L$ is a $2\times k$ linear system, it turns out that a first moment argument is not enough. For example, consider the linear system
\begin{align*}
    L=\begin{pmatrix}
        1 & 0 & -3 & -4\\
        0 & 1 & 1 & 2
    \end{pmatrix},
\end{align*}
so that $$\sigma_L(f_M)=\sum_{r_1,r_2\in\ZZ}f_M(r_1)f_M(r_2)f_M(-3r_1+r_2)f_M(-4r_1+2r_2).$$ 

By the above key observation $(\dagger)$, we get that
\begin{align*}
    \EE[\sigma_L(f_M)]&=|\{(r_1,r_2)\in\ZZ^2:r_1,r_2,-3r_1+r_2,-4r_1+2r_2\in [\![-M,M]\!]\setminus \{0\},\\
    &\hspace{3.15cm}r_1,r_2,-3r_1+r_2,-4r_1+2r_2\text{ can be partitioned into cancelling pairs}\}|\\
    &=|\{(r_1,r_2)\in\ZZ^2:r_1,r_2,-3r_1+r_2,-4r_1+2r_2\in [\![-M,M]\!]\setminus \{0\},\\
    &\hspace{3.08cm}-3r_1+2r_2=0\\
    &\hspace{3.1cm}\text{or }r_1+r_2=-7r_1+3r_2=0\\
    &\hspace{3.1cm}\text{or }-2r_1+r_2=-4r_1+3r_2=0\}|\\
    &=\Theta(M)>0.
\end{align*}
Therefore, to prove \cref{thm:2*k-even}, a first moment argument is insufficient and we need a more careful analysis on the distribution of $\sigma_L(f_M)$. 

\subsection{First attempt: second moment method}

The most natural next step is to try to extract more information about $\sigma_L(f_M)$ from higher moments. For the sake of concreteness, we continue to work with 
\[ L = \begin{pmatrix} 1 & 0 & -3 & -4 \\ 0 & 1 & 1 & 2 \end{pmatrix}.\]

First we compute the variance of $\sigma_L(f_M)$. Define
\begin{align*}
    A_M&=\{(r_1,r_2)\in\ZZ^2: r_1,r_2,-3r_1+r_2,-4r_1+r_2\in [\![-M,M]\!]\setminus \{0\},\\
    &\hspace{3.08cm}r_1,r_2,-3r_1+r_2,-4r_1+r_2\text{ cannot be partitioned into cancelling pairs}\},\\
    B_M&=\{(r_1,r_2)\in\ZZ^2: r_1,r_2,-3r_1+r_2,-4r_1+r_2\in [\![-M,M]\!]\setminus \{0\},\\
    &\hspace{3.08cm}r_1,r_2,-3r_1+r_2,-4r_1+r_2\text{ can be partitioned into cancelling pairs}\}.
\end{align*}
Then we know that
\begin{enumerate}
    \item $\sigma_L(f_M)=\sum_{(r_1,r_2)\in A_M\cup B_M}f_M(r_1)f_M(r_2)f_M(-3r_1+r_2)f_M(-4r_1+2r_2)$;
    \item If $(r_1,r_2)\in B_M$,  then $f_M(r_1)f_M(r_2)f_M(-3r_1+r_2)f_M(-4r_1+2r_2)$ is constant 1;
    \item If $(r_1,r_2)\in A_M$,  then $\EE[f_M(r_1)f_M(r_2)f_M(-3r_1+r_2)f_M(-4r_1+2r_2)]=0$.
\end{enumerate}
Thus, we have
\begin{align*}
    \text{Var}[\sigma_L(f_M)]
    &=\EE[(\sum_{(r_1,r_2)\in A_M} f_M(r_1)f_M(r_2)f_M(-3r_1+r_2)f_M(-4r_1+2r_2))^2]\\
    &= \sum_{(r_1,r_2),(r_1',r_2')\in A_M} \EE [ f_M(r_1)f_M(r_2)f_M(-3r_1+r_2)f_M(-4r_1+2r_2)\\
        &\hspace{3.5cm} \cdot f_M(r_1')f_M(r_2')f_M(-3r_1'+r_2')f_M(-4r_1'+2r_2')] \\
        &\overset{(\dagger)}{=} | \{ ((r_1,r_2),(r_1',r_2'))\in A_M^2:  r_1, r_2, -3r_1 + r_2, -4r_1 + 2r_2, r_1', r_2', -3r_1'+r_2', -4r_1' + 2r_2' \\
        &\hspace{4.93cm}\text{ can be partitioned into cancelling pairs} \}| \\
        &\geq  |\{ ((r_1,r_2),(r_1',r_2'))\in A_M^2: r_1' = -r_1, r_2' = -r_2\}|= \Omega(M^2).
\end{align*}

In particular, this implies that $\mathrm{Var}[\sigma_L(f_M)] = \Omega\left(\EE[\sigma_L(f_M)]^2 \right)$, so the second moment argument also does not provide a sufficiently strong anticoncentration statement to give the desired conclusion.

\subsection{Speculating on higher moments}

However, the observation ($\dagger$) suggests a way to compute higher moments. In particular, the $d$-th moment of $\sigma_L(f_M)$ is number of $d$-tuples $((r_1^{(1)}, r_2^{(1)})$, \dots, $(r_1^{(d)}, r_2^{(d)}))\in A_M^d$ such that the union of multisets $\bigcup_{i=1}^{d} \{\!\!\{ r_1^{(i)}, r_2^{(i)}, -3r_1^{(i)} + 2r_2^{(i)}, -4r_1^{(i)} + 2r_2^{(i)} \}\!\!\}$ can be partitioned into cancelling pairs.

When $d=2t$ is even, this number is at least
\[ \binom{2t}{2}\binom{2t-2}{2} \ldots \binom{4}{2}\binom{2}{2}A_M^t = \frac{(2t)!}{2^tt!}(cM)^{2t}. \]
Here $\binom{2t}{2}\binom{2t-2}{2} \ldots \binom{4}{2}\binom{2}{2}$ is the number of ways to pair up $[2t]$ into $t$ pairs $(i_1, i_1')$, \dots, $(i_t, i_t')$, as in each such pairing, we can pick arbitrary $(r_1^{(i_{j})},r_2^{(i_{j})})\in A_M$ and set $(r_1^{(i'_{j})},r_2^{(i'_{j})}) = -(r_1^{(i_{j})},r_2^{(i_{j})})$ to meet the requirement. It is plausible to guess that this is the dominant term in the $(2t)$-th moment, and this ansatz suggests that $\sigma_L(f_M)$ might be close  to a normal random variable with mean $\Theta(M)$ and variance $\Theta(M^2)$. 

However, if we want to prove this by the method of higher moments, then we need to show that the other ways to partition $\bigcup_{i=1}^{d} \{\!\!\{ r_1^{(i)}, r_2^{(i)}, -3r_1^{(i)} + 2r_2^{(i)}, -4r_1^{(i)} + 2r_2^{(i)} \}\!\!\}$ into cancelling pairs correspond to lower order terms $o(M^d)$. There are $\frac{(4d)!}{2^{2d}(2d)!}$ ways to partition this set into pairs; although most of them cannot be a valid partition into cancelling pairs, such  approach becomes computationally messy very quickly. This motivates us to find another way to prove that $\sigma_L(f_M)$ is close in distribution to a normal random variable, with information from low moments only.

\subsection{Second attempt: Stein's method}

As mentioned earlier, we would like to prove that $\sigma_L(f_M)$ is close in distribution to a normal random variable without appealing to high moment calculations. The method of moments fundamentally relies on the fact that a distribution is uniquely determined by its characteristic function. Another approach to distributional approximation is \emph{Stein's method} in which we replace the characteristic function with a characterizing operator.

\begin{lemma}[Stein's lemma]
    Define the functional operator $\cA$ by $\cA f(x) = f'(x) - xf(x)$. If a random variable $Z$ is such that $\EE \cA f(Z) =0$ for all absolutely continuous functions $f$ with $\norm{f'}_{\infty} < \infty$, then $Z \sim \mathcal{N}(0,1)$.
\end{lemma}

We say that $\cA$ is the \emph{characterizing operator} of the standard normal. An immediate consequence of the above is that we can get quantitative distributional approximation using the characterizing operator. 

\begin{corollary}[Another form of Stein's lemma]
    Let $W$ be any random variable and $Z\sim\mathcal{N}(0,1)$. For any family of functions $\cH$, we have that 
    \[ \sup_{h \in \cH}[|\EE h(W)- \EE h(Z)|] = \sup_{h \in \cH} |\EE[f_h'(W) - Wf_h(W)]|,\]
    where $f_h'$ is the solution to 
    \[ f_h'(x)- xf_h(x) = h(x) - \EE h(Z).\]
\end{corollary}

A working mathematician's form of the above corollary is the following:

\begin{definition}[Wasserstein distance]
     For random variables $X \sim \mu$ and $Y \sim \nu$ defined on a probability space $\Omega$,  the \emph{Wasserstein distance} between $X,Y$ is defined by 
    \[ d_{\text{Wass}}(X,Y) = \sup  \left\{ \left| \int f d \mu - \int f d \nu  \right| : f\colon \Omega \to \mathbb{R} \text{ is } 1\text{-Lipschitz}  \right\}.\]
\end{definition}

\begin{corollary}[{\cite[Theorem 3.1]{R11}}]\label{cor:Stein}
Define \[\mathcal{D} = \left\{ f\in \mathcal{C}^1(\RR): f' \text{ is absolutely continuous},\, \norm{f}_{\infty} \leq 1,\, \norm{f'}_{\infty} \leq \sqrt{2/\pi}, \,\norm{f''}_{\infty} \leq 2 \right\}.\] 
If $W$ is a random variable and $Z\sim\mathcal{N}(0,1)$, then
\[ d_\mathrm{Wass}(W,Z) \leq \sup_{f \in \mathcal{D}} \left| \EE[f'(W) - Wf(W)] \right|.\]
\end{corollary}

Note that in our previous example,  two variables $f_M(r_1)f_M(r_2)f_M(-3r_1+r_2)f_M(-4r_1+2r_2)$, $f_M(r_1')f_M(r_2')f_M(-3r_1'+r_2')f_M(-4r_1'+2r_2')$ are independent unless $r_1,r_1',r_2,r_2'$ satisfy certain linear relations. This is also the reason why we would expect \cref{cor:Stein} to be  useful for us: it allows us to quantify the intuitive idea that a sum of only locally dependent random variables  is approximately normal. More precisely, we have the following consequence of \cref{cor:Stein}, which we formalize using the notion of \textit{dependency graph}.  
\begin{definition}[Dependency graph]\label{def:dependency-graph}
    Given a graph $G = (V, E)$ and a collection of random variables $\{X_i \}_{i \in V}$, we say that $G$ is a \emph{dependency graph} for $\{X_i \}_{i \in V}$ if for any disjoint subsets $S,T \subseteq V$ with no edges between $S$ and $T$, the two random vectors $(X_{i})_{i\in S}$ and $(X_{i})_{i\in T}$ are independent. 
    
    For a graph $G=(V,E)$ and vertex $u\in V$, we use $N_u$ to denote the closed neighborhood of $u$; that is, $N_u = \{ v: uv \in E \} \cup \{u \}$. 
\end{definition}

\begin{lemma}\label{lem:dep-stein}
    Let $\{X_v\}_{v \in V}$ be a family of random variables with dependency graph $G=(V,E)$. 
    Let $\rho = \sqrt{\mathrm{Var}\left( \sum_v X_v \right)}$ and $W = \rho^{-1}\sum_{ v \in V} X_v$. Suppose $\EE X_v = 0$ for all $v \in V$. Then for $Z \sim \mathcal N(0,1)$ and any $x\in\RR$, we have 
\begin{align*}
    &|\PP[W \leq x] - \PP[Z \leq x]|^2 \\
    &\leq 2/\pi\cdot \rho^{-3} \sum_{u \in V} \left( 3 \sum_{v,w \in N_u}\left| \EE [X_uX_vX_w] \right|  + 4 \sum_{v \in N_u} | \EE [X_uX_v]|\cdot \EE \left[ \left| \sum_{w \in N_u \cup N_v} X_w  \right| \right]  \right). 
\end{align*}

\end{lemma}
A proof of \cref{lem:dep-stein} (which is effectively a modified version of \cite[Theorem 3.1]{BRS89}) is provided in \cref{app:stein}.

\subsection{Proof of \cref{thm:2*k-even}}

Suppose $k\geq 4$ is even, and $L$ is a $2\times k$ linear system with $s(L)=k-1$.
To prove \cref{thm:2*k-even}, by \cref{thm:fourier-template}, it suffices to construct a Fourier template $g:\ZZ\to\CC$ such that $\sigma_L(g)<0$.

Suppose $L$ takes the form
\begin{align*}
    L=\begin{pmatrix}
        a_1&\dots&a_k\\
        b_1&\dots&b_k
    \end{pmatrix}.
\end{align*}
Since $s(L)=k-1$, the column vectors $(a_1,b_1),\dots,(a_k,b_k)\in\ZZ^2$ are pairwise linearly independent.
For any Fourier template $g \colon \ZZ \to \CC$, we have
\begin{align*}
    \sigma_L(g) = \sum_{r_1, r_2 \in \ZZ} g(a_1r_1+b_1r_2)  \cdots g(a_kr_1 + b_kr_2).
\end{align*}

The key idea for proving \cref{thm:2*k-even} is to combine \cref{thm:fourier-template} with a probabilistic argument.
\begin{proposition}\label{prop:CLT}

    Let $L$ be a $2 \times k$ linear system with $k\geq 4$ even and  $s(L) = k-1$.
    For every $M \in \NN$, define the random Fourier template $f_M\colon \ZZ \to \CC$ by \[\begin{cases} f_M(r) = \xi_r &r \in [\![-M,M]\!]\setminus\{0\} \\ f_M(r) = 0 &\text{otherwise,}  \end{cases}\]
where every $\xi_r$ is a uniformly random unit complex number, subject to $\xi_r = \overline{\xi_{-r}}$ but independent otherwise. 
Then  $\PP[\sigma_L(f_M)<0]>0$ for all sufficiently large $M$.
\end{proposition}

\cref{thm:2*k-even} then directly follows from \cref{prop:CLT} and \cref{thm:fourier-template}.
\begin{proof}[Proof of \cref{thm:2*k-even}]
    Let $L$ be a $2 \times k$ linear system with $k\geq 4$ even and  $s(L) = k-1$. By \cref{prop:CLT}, there exists a Fourier template $g:\ZZ\to\CC$ such that $\sigma_L(g)<0$. By \cref{thm:fourier-template}, $L$ is uncommon over $\FF_p$ for all sufficiently large $p$.
\end{proof}

It remains to show \cref{prop:CLT}.

\begin{proof}[Proof of \cref{prop:CLT}]
Fix a $2 \times k$ linear system $L$ with $k\geq 4$ even and  $s(L) = k-1$. Choose $\lambda\in\NN$ sufficiently large such that, for all $M\in\NN$, if $ a_1r_1+b_1r_2$, \dots, $a_kr_1+b_kr_2\in [\![-M,M]\!]$, then $-\lambda M\leq r_1,r_2\leq \lambda M$. 

Let $N=\lambda M$. Since $f_M$ is always supported on $[\![-M,M]\!]\setminus\{0\}$, we have
\begin{align*}
    \sigma_L(f_M)&=\sum_{r_1,r_2\in\ZZ} f_M(a_1r_1+a_1r_2)\cdots f_M(a_kr_1+b_kr_2)
    =\sum_{ r_1,r_2\in[\![-N,N]\!]} f_M(a_1r_1+a_1r_2)\cdots f_M(a_kr_1+b_kr_2).
\end{align*}
For every multiset $S\subseteq \ZZ$, define the associated random variable
$$X_{S} = \prod_{s\in S} f_M(s),$$
where the product goes over all $s\in S$ counting multiplicities. Note that $X_S$ has the following properties:
\begin{enumerate}
\item[(i)] $X_SX_{S'}=X_{S\cup S'}$.
    \item[(ii)] 
    \begin{align*}
    \begin{cases}
        X_S=1 & S\subseteq  [\![-M,M]\!]\setminus\{0\}, \, S\text{ can be partitioned into cancelling pairs,}\\
        \EE X_{S}=0 &\text{otherwise.}
    \end{cases} 
\end{align*}
    \item[(iii)] Let $\mathcal S$ and $\mathcal S'$ be two collections of subsets of $\ZZ$. If  $\pm \bigcup_{S\in \mathcal S} S$ and  $\pm \bigcup_{ S'\in \mathcal S'} S'$ are disjoint, then the two random vectors $(X_{S})_{{S}\in \mathcal S}$ and $(X_{S})_{{S}\in \mathcal S'}$ are independent.
\end{enumerate}

For every $(r_1,r_2)\in[\![-N,N]\!]^2$, define the associated multiset 
\[
S_{(r_1,r_2)}=\{\!\!\{a_1r_1+b_1r_2,\dots,a_kr_1+b_kr_2\}\!\!\}.
\]
We simply write $X_{(r_1,r_2)}$ to denote $X_{S(r_1,r_2)}$. Thus, we have
\begin{align*}
    \sigma_L(f_M)&=\sum_{ r_1,r_2\in[\![-N,N]\!]} f_M(a_1r_1+a_1r_2)\cdots f_M(a_kr_1+b_kr_2)
=\sum_{ r_1,r_2\in [\![-N,N]\!]}X_{(r_1,r_2)}.
\end{align*}
Consider the partition $[\![-N,N]\!]^2=A_N\cup B_N$ defined by
\begin{align*}
    A_N&=\{(r_1,r_2)\in[\![-N,N]\!]^2: S_{(r_1,r_2)}\not\subseteq[\![-M,M]\!]\setminus\{0\}\\
    &\hspace{4.25cm}\text{or }S_{(r_1,r_2)}\text{ cannot be partitioned into cancelling pairs} \},\\ 
    B_N&=\{(r_1,r_2)\in[\![-N,N]\!]^2:  S_{(r_1,r_2)}\subseteq[\![-M,M]\!]\setminus\{0\}\\&\hspace{4.25cm}\text{and }S_{(r_1,r_2)}\text{ can be partitioned into cancelling pairs}\}.
\end{align*}

For all $(r_1,r_2)\in B_N$, we always have $X_{(r_1,r_2)}=1$. Therefore we have
$$\sigma_L(f_M)=|B_N|+\sum_{ (r_1,r_2)\in A_N}X_{(r_1,r_2)}.$$

We study the distribution of $\sum_{ (r_1,r_2)\in A_N}X_{(r_1,r_2)}$ using \cref{lem:dep-stein}. Let $G$ be a dependency graph for the family of variables $\{X_{(r_1,r_2)}:(r_1,r_2)\in A_N\}$, defined by
\begin{align*}
    V(G)&=A_N,\\
    E(G)&=\{\{(r_1,r_2),(r_1',r_2')\}: \pm S_{(r_1,r_2)}\cap \pm S_{(r_1',r_2')}\neq\emptyset\}\subseteq \binom{A_N}{2}.
\end{align*}
Recall from \cref{def:dependency-graph} that $N_{(r_1,r_2)}$ denotes the closed neighborhood of $(r_1,r_2)$ in $G$, which in this case equals 
\[
\{(r_1',r_2')\in A_N:\pm S_{(r_1,r_2)}\cap \pm S_{(r_1',r_2')}\neq\emptyset\}.
\]
We will verify the following properties of $\sum_{ (r_1,r_2)\in A_N}X_{(r_1,r_2)}$. In the asymptotic bounds below, we allow the constants in $O(\cdot)$ and $\Omega(\cdot)$  to depend on $k$, as $L$ is fixed.
\begin{claim}
\label{claim:dependency-properties}\quad
\textup{
  \begin{enumerate}
    \item $\EE[\sum_{ u\in A_N}X_{u}]=0$, and $\textup{Var}(\sum_{ u\in A_N}X_{u})=\Omega(N^2)$;
    \item For every $u\in A_N$, 
    $\sum_{v\in A_N}\EE[X_{u}X_{v}]=O(1)$;
    \item $\sum_{u,v,w\in A_N}\EE[X_{u}X_{v}X_{w}]=O(N^2);$
    \item For every $u,v\in A_N$, 
    $\sum_{w_1,w_2\in N_u\cup N_{v}}\EE[X_{w_1}\overline{X_{w_2}}]=O(N)$.
\end{enumerate}  }
\end{claim}

We first prove \cref{prop:CLT} assuming \cref{claim:dependency-properties}.  Let $Z\sim \mathcal N(0,1)$ be the normal distribution and $\rho=\sqrt{\text{Var}(\sum_{u\in A_N}X_{u})}=\Omega(N)$ (per \cref{claim:dependency-properties}(1)). We know from \cref{claim:dependency-properties}(4) that
\begin{enumerate}
    \item [(5)] For every $u,v\in A_N$, by applying the Cauchy--Schwarz inequality, we have
    \begin{align*}
    (\EE | \sum_{w\in N_{u}\cup N_{v}} X_{w}|)^2&\leq \EE[|\sum_{w\in N_{u}\cup N_{v}} X_{w}|^2] 
    =\EE[(\sum_{w\in N_{u}\cup N_{v}} X_{w})(\sum_{w\in N_{u}\cup N_{v}} \overline{X_{w}})] \\
    &= \sum_{w_1,w_2\in N_u\cup N_{v}}\EE [X_{w_1}\overline{X_{w_2}}]=O(N).
\end{align*}
\end{enumerate}
Let  $W=\rho^{-1}\sum_{u\in A_N}X_{u}$. By \cref{lem:dep-stein}, for $Z\sim \mathcal N(0,1)$ and all $x\in\RR$, we have
\begin{align*}
    &\pi/2\cdot |\PP[W\leq x]-\PP[Z\leq x]|^2\\
    &\leq \rho^{-3} \sum_{u \in A_N}( 3 \sum_{v,w \in N_u} |\EE [X_uX_vX_w]|  + 4 \sum_{v \in N_u} |\EE [X_uX_v]|\cdot   \EE|  \sum_{w \in N_u \cup N_v}  X_w |  )\\
    &\overset{(*)}{\leq}\rho^{-3}(3\sum_{u,v,w\in A_N}\EE[X_uX_vX_w]+4\sum_{u,v\in A_N}\EE[X_uX_v]\cdot \EE|  \sum_{w \in N_u \cup N_v}  X_w | ) \\
    &\leq O(N^{-3})(O(N^2)+O(N^{1/2})\sum_{u,v\in A_N}\EE [X_uX_v]) \tag*{(1)(3)(5)}\\
    &\leq O(N^{-3})(O(N^2)+O(N^{5/2}))\tag*{(2)}=O(N^{-1/2}),
\end{align*}
where $(\ast)$ is due to the fact that (recall properties (i)(ii) at the beginning of proof) $$0\leq \EE[X_uX_v], \EE[X_uX_vX_w]\leq 1\qquad \text{for all $u,v,w\in A_N$}.$$
Since $\sigma_L(f_M)=|B_N|+\sum_{v\in A_N}X_v$, we have
\[
\PP(\sigma_L(f_M)<0)=\PP(\sum_{v\in A_N}X_v<-|B_N|)=\PP(W<-|B_N|/\rho),
\]
and therefore
\begin{align*}
    |\PP(\sigma_L(f_M)<0)-\PP(Z<-|B_N|/\rho)|=O(N^{-1/4}).
\end{align*}
Since $|B_N|=O(N)$ and $\rho=\Omega(N)$, we have $|B_N|/\rho=O(1)$. Thus $\PP(Z<-|B_N|/\rho)=c>0$, and therefore $\PP(\sigma_L(f_M)<0)\geq c-O(N^{-1/4})>0$ for all $N$ sufficiently large.
\end{proof}

It remains to verify \cref{claim:dependency-properties}.

\begin{proof}[Proof of \cref{claim:dependency-properties}]\quad

(1) For every $u=(r_1,r_2)\in A_N$, since $S_{(r_1,r_2)}$ cannot be partitioned into cancelling pairs, we have $\EE X_{(r_1,r_2)}=0$.  Moreover, we have
\begin{align*}
    &\EE[(\sum_{u\in A_N}X_{u})^2]=\sum_{u,v\in A_N}\EE[X_{u}X_{v}]=\sum_{u,v\in A_N}\EE[X_{S_u\cup S_v}]\\
    &=|\{(u,v)\in A_N^2: S_{u}\cup S_{v}\subseteq [\![-M,M]\!]\setminus\{0\}, \, S_{u}\cup S_{v}\text{ can be partitioned into cancelling pairs}\}|\\
    &\geq |\{((r_1,r_2),(r_1',r_2'))\in A_N^2: S_{(r_1,r_2)}\subseteq[\![-M,M]\!]\setminus\{0\}, r_1'=-r_1, r_2'=-r_2\}|\\
    &=\Omega(N^2). 
\end{align*}
This gives $\text{Var}(\sum_{u\in A_N}X_{u})=\Omega(N^2)$.

(2) For all $u=(r_1,r_2)\in A_N$, we have
\begin{align*}
    &\sum_{v\in A_N}\EE[X_{u}X_{v}]=\sum_{v\in A_N}\EE[X_{S_u\cup S_v}]\\
&\leq |\{v\in A_N: S_{u}\cup S_{v} \text{ can be partitioned into cancelling pairs}\}|.
\end{align*}
Recall that $S_{u}$ itself cannot be partitioned into cancelling pairs. Therefore, if $S_{u}\cup S_{v}$ can be partitioned into cancelling pairs, then $v=(r_1',r_2')$ must satisfy equations of the form
\begin{align*}
    a_{i'}r_1'+b_{i'}r_2'&=-(a_{i}r_1+b_ir_2)\\
    a_{j'}r_1'+b_{j'}r_2'&=-(a_{j}r_1+b_jr_2)
\end{align*}
with $i\neq j$ and $i'\neq j'$. Since $(a_{i'},b_{i'})$ and $(a_{j'},b_{j'})$ are \textit{linearly independent} due to the girth condition $s(L)=k-1$, every such pair of equations has at most one solution. Moreover, there are at most $k^4$ such pairs. Thus, for every $u\in A_N$, we get that
$$|\{v\in A_N: S_{u}\cup S_{v} \text{ can be partitioned into cancelling pairs}\}|=O(1).$$

Proofs of (3)(4) have similar ideas as in (2).

(3) For every triple $(u,v,w)\in A_N^3$, define a graph $H_{u,v,w}$  on the multiset $S_{u}\cup S_{v}\cup S_{w}$, such that $s_1\sim s_2$ in $H_{u,v,w}$ if and only if $s_1+s_2=0$. Observe that
\begin{align*}
    \sum_{u,v,w\in A_N}\EE[X_{u}X_{v}X_{w}]&=\sum_{u,v,w\in A_N}\EE[X_{S_u\cup S_v\cup S_w}]\\
    &=|\{(u,v,w)\in A_N^3: S_u,S_v,S_w\subseteq[\![-M,M]\!]\setminus\{0\}, \\
    &\hspace{3.5cm} S_{u}\cup S_{v}\cup S_w \text{ can be partitioned into cancelling pairs}\}|\\
    &=|\{(u,v,w)\in A_N^3: S_u,S_v,S_w\subseteq[\![-M,M]\!]\setminus\{0\}, \\
    &\hspace{3.5cm} H_{u,v,w}\text{ has a perfect matching}\}|.
\end{align*}
Note that every graph $H_{u,v,w}$ has $|S_{u}\cup S_{v}\cup S_{w}|=3k$ vertices.
We prove that, for every perfect matching $\mathcal M$ on $3k$ vertices, there are $O(N^2)$  choices of $(u,v,w)\in A_N^3$ such that $S_{u}\cup S_{v}\cup S_{w}\subseteq[\![-M,M]\!]\setminus\{0\}$ and $H_{u,v,w}$ contains $\mathcal M$ as a subgraph. Since there are finitely many perfect matchings on $3k$ vertices, we get that 
\[
|\{(u,v,w)\in A_N^3:  S_u,S_v,S_w\subseteq[\![-M,M]\!]\setminus\{0\},\, H_{u,v,w}\text{ has a perfect matching}\}|=O(N^2).
\]

Fix  a perfect matching $\mathcal M$ on $3k$ vertices. Suppose $(u,v,w)\in A_N^3$ satisfies $S_{u}\cup S_{v}\cup S_{w}\subseteq[\![-M,M]\!]\setminus\{0\}$,  and $H_{u,v,w}$ contains $\mathcal M$ as a subgraph.
For $U,V\subseteq S_{u}\cup S_{v}\cup S_{w}$, let $e_{\mathcal M}(U,V)$ denote the number of edges in $\mathcal M$ whose vertices intersect both $U$  and $V$. Since none of $S_{u}$, $S_{v}$, $S_{w}$ can be partitioned into cancelling pairs, we must have
\begin{align*}
&e_{\mathcal M}(S_{u},\, S_{v}\cup S_{w})\geq 2, \qquad e_{\mathcal M}(S_{v},\, S_{u}\cup S_{w})\geq 2,\qquad e_{\mathcal M}(S_{w},\, S_{u}\cup S_{v})\geq 2.
\end{align*}
Up to permuting $u,v,w$, we have the following two cases.

\textbf{Case 1.} $e_{\mathcal M}(S_{u},S_{v})\geq 2$.
In this case, $\mathcal M$ gives $i\neq j$ and $i'\neq j'$ such that
\begin{align*}
    a_{i'}r_1'+b_{i'}r_2'=-(a_{i}r_1+b_ir_2),\qquad 
    a_{j'}r_1'+b_{j'}r_2'=-(a_{j}r_1+b_jr_2)
\end{align*}
for $u=(r_1,r_2)$ and $v=(r_1',r_2')$. Since $(a_{i'},b_{i'})$ and $(a_{j'},b_{j'})$ are linearly independent, $v=(r_1',r_2')$ is uniquely determined by $u=(r_1,r_2)$. Since $e_{\mathcal M}(S_{w}, S_{u}\cup S_{v})\geq 2$, $\mathcal M$ also gives $i''\neq j''$ and $i^\circ,i^*$ such that
\begin{align*}
    &a_{i''}r_1''+b_{i''}r_2''=-(a_{i^\circ}r_1^{\circ}+b_{i^\circ}r_2^{\circ}),\qquad a_{j''}r_1''+b_{j''}r_2''=-(a_{i^*}r_1^{*}+b_{i^*}r_2^{*})
\end{align*}
for $w=(r_1'',r_2'')$ and $(r_1^\circ,r_2^\circ),(r_1^*,r_2^*)\in\{(r_1,r_2),(r_1',r_2')\}=\{u,v\}$. Since $(a_{i''},b_{i''})$ and $(a_{j''},b_{j''})$ are linearly independent, 
$w=(r_1'',r_2'')$ is uniquely determined by $u=(r_1,r_2)$ and $v=(r_1',r_2')$. Therefore, the number of choices for $u,v,w$ is at most the number of choices for $u=(r_1,r_2)$, which is $O(N^2)$.

\textbf{Case 2.} $
    e_{\mathcal M}(S_{u},S_{v})=e_{\mathcal M}(S_{u},S_{w})=e_{\mathcal M}(S_{v},S_{w})=1.$
In this case, $\mathcal M$ gives $i\neq j$, $i'\neq j'$ and $i''\neq j''$ such that
\begin{align*}
    a_{i'}r_1'+b_{i'}r_2'&=-(a_{i}r_1+b_{i}r_2),\\
    a_{i''}r_1''+b_{i''}r_2''&=-(a_{j}r_1+b_{j}r_2),\\
    a_{j''}r_1''+b_{j''}r_2''&=-(a_{j'}r_1'+b_{j'}r_2').
\end{align*}
Since $S_{u}\cup S_{v}\cup S_{w}\subseteq[\![-M,M]\!]\setminus\{0\}$, none of the above three equations can be 0.

Now since $|S_u|=|S_v|=|S_w|=k\geq 4$,  $\mathcal M$ also has at least one edge within  each of $S_{u},S_{v},S_{w}$. Therefore, $\mathcal M$ in particular gives  $i',k',\ell'$ distinct,  $i'',k'',\ell''$ distinct, and $i\neq j$ such that
\begin{align*}
    &a_{i'}r_1'+b_{i'}r_2'=-(a_{i}r_1+b_{i}r_2),\\
    &a_{i''}r_1''+b_{i''}r_2''=-(a_{j}r_1+b_{j}r_2),\\
    &a_{k'}r_1'+b_{k'}r_2'=-(a_{\ell'}r_1'+b_{\ell'}r_2'),\\
    &a_{k''}r_1''+b_{k''}r_2''=-(a_{\ell''}r_1''+b_{\ell''}r_2'').
\end{align*}
Since
\begin{enumerate}
    \item[(i)] $(a_{k'},b_{k'})$ and $(a_{\ell'},b_{\ell'})$ are  linearly independent,
    \item[(ii)]$(a_{k''},b_{k''})$ and $(a_{\ell''},b_{\ell''})$ are  linearly independent, 
    \item[(iii)]  $a_ir_1+b_ir_2$ 
 and $a_jr_1+b_jr_2$ are nonzero,
\end{enumerate}
we get that $v=(r_1',r_2') $ and $ w=(r_1'',r_2'')$ are uniquely determined by $u=(r_1,r_2)$. Hence the number of choices for $u,v,w$ is at most the number of choices for $u=(r_1,r_2)$, which is $O(N^2)$.

(4) For every $u=(r_1,r_2),v=(r_1',r_2')\in A_N$, we have
\begin{align*}
    &\sum_{\substack{w_1,w_2\in N_u\cup N_v}}\EE[X_{w_1}\overline{X_{w_2}}]=\sum_{\substack{w_1,w_2\in N_u\cup N_v}}\EE[X_{w_1}X_{-w_2}]\\
    &\leq |\{(w_1,w_2)\in (N_{u}\cup N_{v})^2: S_{w_1}\cup S_{-w_2}\text{ can be partitioned into cancelling pairs}\}|.
\end{align*}
Recall that each of $S_{w_1},S_{w_2}$ cannot be partitioned into cancelling pairs. Therefore, if $S_{w_1}\cup S_{-w_2}$ can be partitioned into cancelling pairs, then $w_1=(r_1'',r_2'')$, $w_2=(r_1''',r_2''')$ must satisfy two equations
\begin{align*}
    a_{i''}r_1''+b_{i''}r_2''=-(a_{i'''}r_1'''+b_{i'''}r_2'''),\qquad 
    a_{j''}r_1''+b_{j''}r_2''=-(a_{j'''}r_1'''+b_{j'''}r_2''')
\end{align*}
with $i''\neq j''$ and $i'''\neq j'''$. Again, by linear independence of the columns of $L$, for fixed choices of $i'',j'',i''',j'''$, each one of $w_1,w_2$ will uniquely determine another.

We now look at the number of possible choices for $w_1$. Since $w_1\in N_{u}\cup N_{v}$, by \cref{def:dependency-graph}, $w_1=(r_1'',r_2'')$ must also satisfy
\begin{align*}
a_{\ell}r_1''+b_{\ell}r_2''&=a_{\ell^\circ}r_1^\circ+ b_{\ell^\circ}r_2^\circ \qquad\text{ or }\qquad a_{\ell}r_1''+b_{\ell}r_2''=-(a_{\ell^\circ}r_1^\circ+ b_{\ell^\circ}r_2^\circ)
\end{align*}
for some $\ell,\ell^\circ$ and
 $(r_1^\circ,r_2^\circ)\in \{(r_1,r_2),(r_1',r_2')\}$. Given $\ell,\ell^\circ$ and $(r_1^\circ,r_2^\circ)$, since we cannot have $a_\ell=b_\ell=0$, the number of possible choices for $w_1$ is  at most $O(N)$. 

Overall, since there are finitely many choices for $i'',j'',i''',j''',\ell,\ell^\circ$ and  $(r_1^\circ,r_2^\circ)\in \{(r_1,r_2),(r_1',r_2')\}$, we get that
\[
|\{(w_1,w_2)\in (N_{u}\cup N_{v})^2: S_{w_1}\cup S_{-w_2}\text{ can be partitioned into cancelling pairs}\}|=O(N).
\]
\end{proof}

\subsection{Proof of \cref{thm:2*k-even-ext}}

    Let $L$ be a linear system with $c=c(L)$. Recall from \cref{rem:critical-3} that all elements in $\{L_B:B\in \mathcal C(L)\}$ are either length-$c$ linear equations or $2\times c$ systems  with  girth $c-1$.
    
    Consider the partition $\mathcal C(L)=\mathcal C_1\cup \mathcal C_2$ defined by
    \begin{align*}
        \mathcal C_1&=\{B\in \mathcal C(L):L_B\text{ is length-$c$ linear equation}\},\\
        \mathcal C_2&=\{B\in \mathcal C(L):L_B\text{ is $2\times c$ linear system}\}.
    \end{align*}

The following lemma is essentially due to Altman \cite[Theorem 5.2]{Alt22a}, although Altman did not work in the language of Fourier templates. For the sake of completeness, we provide its proof here.
\begin{proposition}\label{prop:tensor}
    Let $L$ be a linear system with $s(L)$ odd. Suppose there exists a Fourier template $g:\ZZ^d\to\CC$ such that $\sum_{B\in\mathcal C_2}\sigma_{L_B}(g)<0$. Then $L$ is uncommon over all sufficiently large $p$.
\end{proposition}
\begin{proof}Suppose $\sum_{B\in\mathcal C_2}\sigma_{L_B}(g)<0$ for some Fourier template $g:\ZZ^d\to\CC$. By \cref{thm:fourier-template}, this gives a function $f:\FF_p\to \RR$ with $\EE f=0$ and $\sum_{B\in\mathcal C_2}t_{L_B}(f)<0$.

However, to show that $L$ is uncommon over all sufficiently large $p$, we should take into account subsystems in $\mathcal C_1$ as well. That is, we wish to show that there is some function $h:\FF_p^n\to \RR$ with $\EE h=0$ and $\sum_{B\in\mathcal C(L)}t_{L_B}(h)<0$, so that we can utilize \cref{prop:critical-uncommon}. 

To do so, we take a ``tensoring'' of $f$ with the indicator function of 0. Namely, define $h:\FF_p^2\to\RR$ by $h(y,z)=f(y)\mathbf 1_{z=0}$. Since $f$ has mean 0, so does $h$. Moreover, for any $B\in \mathcal C(L)$,  we have (with $c=c(L)$)
\begin{align*}
    t_{L_B}(h)&=\EE_{\substack{\mathbf x\in(\FF_p^2)^{c}\\L_B\mathbf x=0}}h(\mathbf x_1)\cdots h(\mathbf x_{c})
    =\EE_{\substack{\mathbf y\in\FF_p^{c}\\L_B\mathbf y=0}}f(y_1)\cdots f(y_{c})\cdot \EE_{\substack{\mathbf z\in\FF_p^{c}\\L_B\mathbf z=0}}\mathbf 1_{z_1=0}\cdots \mathbf 1_{z_{c}=0}
    =t_{L_B}(f)t_{L_B}(\mathbf 1_{z=0}).
\end{align*}
Observe that if $L_B$ is a length-$c$ equation, then it has $p^{c-1}$ solutions $\mathbf z\in \FF_p^{c}$; if $L_B$ is a $2\times c$ linear system, then it has $p^{c-2}$ solutions $\mathbf z\in \FF_p^{c}$. Thus, we have
\begin{align*}
    t_{L_B}(\mathbf 1_{z=0})=\begin{cases}
        p^{1-c} &B\in\mathcal C_1\\
        p^{2-c} &B\in\mathcal C_2.
    \end{cases}
\end{align*}
Therefore, given that $\sum_{B\in\mathcal C_2}t_{L_B}(f)<0$, we have
\begin{align*}
    \sum_{B\in\mathcal C(L)}t_{L_B}(h)&=\sum_{B\in\mathcal C_1}t_{L_B}(h)+\sum_{B\in\mathcal C_2}t_{L_B}(h)
    =p^{1-c}\sum_{B\in\mathcal C_1} t_{L_B}(f)+p^{2-c}\sum_{B\in\mathcal C_2} t_{L_B}(f),
\end{align*}
which is negative for $p$ sufficiently large. Thus, by \cref{prop:critical-uncommon}, $L$ is uncommon over all sufficiently large $p$.
\end{proof}

We are now ready to prove \cref{thm:2*k-even-ext}.
\begin{proof}By \cref{prop:tensor}, it suffices to construct a Fourier template $g:\ZZ \to \CC$ such that
    $\sum_{B\in\mathcal C_2}\sigma_{L_B}(g)<0$.
    Suppose $\mathcal C_2(L)=\{B_1,\dots,B_\ell\}$, with
    \[
    L_{B_j}=\begin{pmatrix}
        a_{j,1}&\dots & a_{j,c}\\
        b_{j,1}&\dots& b_{j,c}
    \end{pmatrix}\qquad \text{ for every $j\in[\ell]$.}
    \]

    Define the random Fourier template $f_M:\ZZ\to\CC$ as in \cref{prop:CLT}.
    Also, choose $\lambda\in\NN$ sufficiently large such that   if $a_{j,1}r_1+b_{j,1}r_2,\dots, a_{j,k}r_1+b_{j,k}r_2$ lie in $[\![-M,M]\!]\setminus\{0\}$ for all $j\in [\ell]$, then $-\lambda M\leq r_1,r_2\leq \lambda M$.

    For every $(r_1,r_2)\in[\![-N,N]\!]$, define the associated multisets
    \begin{align*}
        S_{(r_1,r_2)}^1&=\{\!\!\{a_{1,1}r_1+b_{1,1}r_2,\dots, a_{1,k}r_1+b_{1,k}r_2\}\!\!\},\\
        &\vdots\\
        S_{(r_1,r_2)}^\ell&=\{\!\!\{a_{\ell,1}r_1+b_{\ell,1}r_2,\dots, a_{\ell,k}r_1+b_{\ell,k}r_2\}\!\!\}.
    \end{align*}
    We then have 
    \[
    \sum_{B\in \mathcal C_2}\sigma_{L_B}(f_M)=\sum_{r_1,r_2\in [\![-N,N]\!]}\sum_{j=1}^\ell X_{S^j_{(r_1,r_2)}}.
    \]
Let  $\rho$ denote the standard deviation of $\sum_{r_1,r_2\in [\![-N,N]\!]}\sum_{j=1}^\ell X_{S^\ell_{(r_1,r_2)}}$, and define
\[
W=\rho^{-1}\sum_{r_1,r_2\in [\![-N,N]\!]}\sum_{j=1}^\ell X_{S^\ell_{(r_1,r_2)}}.
\]
An argument almost identical to \cref{prop:CLT} shows that 
$\PP\Big(\sum_{B\in \mathcal C_2}\sigma_{L_B}(f_M)<0\Big)>0$
for all $M$ sufficiently large. In particular, there exists Fourier template $g:\ZZ\to\CC$ such that $\sum_{B\in \mathcal C_2}\sigma_{L_B}(g)<0$. Finally, by \cref{prop:tensor}, we know that $L$ is uncommon over $\FF_p$ for all $p$ sufficiently large.
\end{proof}

\section{Proof of \cref{thm:2*5}}\label{sec:2*5}

\subsection{Introduction}\label{subsec:5-intro}
In this section, we prove \cref{thm:2*5}. Let $L$ be an irredundant $2\times 5$ linear system on variables $x_1,\dots,x_5$. The bulk of \cref{thm:2*5} considers the case $s(L)=4$.
Under the assumption $s(L)=4$, we know that the elements in $\{L_B:B\in\mathcal C(L)\}$ of $L$ are five length-4 linear equations, supported on five different 4-subsets of $\{x_1,\dots,x_5\}$.

Let $L_1,\dots,L_5$ denote these five equations.
By \cref{thm:fourier-template}, to show that $L$ is uncommon, it suffices to find a Fourier template $g \colon \ZZ^d \to \CC$ such that $\sum_{i=1}^5\sigma_{L_i}(g) < 0$. Since each $\sigma_{L_i}(g)$ is defined by both $g$ and the coefficients in $L_i$, to simplify our further discussion, we give special names to the set of coefficients in each $L_i$.

\begin{setup}\label{setup:s=4} Suppose $L$ is an irredundant $2\times 5$  linear system with $s(L)=4$. For every $i=1,\dots,5$, we admit the following notation:
\begin{itemize}
    \item Let $L_i \mathbf x := a_{i,1}x_1+\dots+a_{i,5}x_5$ be the unique (up to constant multiplication) equation in the row span of $L$, such that $a_{i,i}=0$.
    \item Let $A_i=\{\!\!\{a_{i,j}:j\in[5],\,j\neq i\}\!\!\}$ be the multiset of nonzero coefficients in $L_i$.
    \item For every Fourier template $g:\ZZ^d\to\CC$, define the associated function $g_{A_i}:\ZZ^{d}\to\CC$ by
    \[
    g_{A_i}(r) :=\prod_{a\in A_i}g(ar)
    \]
    and the associated sum $\sigma_{A_i}(g)$ by $$\sigma_{A_i}(g)=\sum_{r\in\ZZ^d}g_{A_i}(r),$$
    so that $g_{A_i}=g_{L_i}$ and $\sigma_{A_i}(g)=\sigma_{L_i}(g)$ as in \cref{def:fourier-template-sum}.
\end{itemize}

\end{setup}

Suppose we wish to show that $L$ is uncommon, and hope to find some Fourier template $g$ with $\sum_{i=1}^{5} \sigma_{A_i}(g) < 0$. 
This is easy if none of $L_1,\dots,L_5$ is common, as a random Fourier template $g:\ZZ\to \CC$ has $\sum_{i=1}^{5} \EE[\sigma_{A_i}(g)]=0$. However, if some equation $L_{i_0}$ among $L_1,\dots,L_5$ is common, with $A_{i_0}=\{\!\!\{1,-1,\lambda,-\lambda \}\!\!\}$, then $\sigma_{A_{i_0}}(g)=\sum_{r}|g(r)|^2|g(\lambda r)|^2$ must be a positive contribution to $\sum_{i=1}^{5} \sigma_{A_i}(g)$. In this case, we have to carefully pick some $g$ so that the other terms $\sigma_{A_{i}}(g)$, $i\neq i_0$ are ``negative enough'' to balance out this positive term. Now, if some $L_{i_0}$ is an additive quadruple $L_{i_0}\mathbf x=x_{j_1}+x_{j_2}-x_{j_3}-x_{j_4}$, then we have $A_{i_0}=\{\!\!\{1,1,-1,-1\}\!\!\}$, and the sum $\sigma_{A_{i_0}}(g)=\sum_{r}|g(r)|^4$ would be even harder to balance out. 

This motivates us to classify all the irredundant $2\times 5$ linear systems $L$ according to the number of   common equations and additive quadruples among $L_1,\dots,L_5$,  and prove \cref{thm:2*5} under each case.

\begin{lemma}
    Let $L$ be an irredundant $2\times 5$  linear system, with $s(L)=4$. Then there cannot be two additive quadruples among $L_1,\dots,L_5$.
\end{lemma}
\begin{proof}Without loss of generality, suppose $L_1$ and $L_2$ are both additive quadruple. Up to isomorphism, $L$ is one of 
\begin{align*}
    \begin{pmatrix}
        0 & 1 & 1 & -1 & -1\\
        1 & 0 & 1 & -1 & -1
    \end{pmatrix},\, 
    \begin{pmatrix}
        0 & 1 & 1 & -1 & -1\\
        1 & 0 & -1 & 1 & -1
    \end{pmatrix}.
\end{align*}
In either case, we have $s(L)\leq 3$, which is a 
contradiction.
\end{proof}

\begin{corollary}
\label{cor:2*5-classification}
    Let $L$ be an irredundant $2\times 5$  linear system, with $s(L)=4$. Then $L$ satisfies exactly one of the following:
    \begin{enumerate}\itemindent=20pt
        \item[Case A.] None of $L_1,\dots, L_5$ is an additive quadruple, and at least three of them are uncommon.
        \item[Case B.] Exactly one of $L_1,\dots,L_5$ is an additive quadruple, and none of the others is common.
        \item[Case C.] Exactly one of $L_1,\dots,L_5$ is an additive quadruple, and exactly one among the others is common but not an additive quadruple.
        \item[Case D.] At least three of $L_1,\dots,L_5$ are common.
    \end{enumerate}
\end{corollary}

Our strategy of finding the appropriate Fourier template $g$ with $\sum_{i=1}^{5} \sigma_{A_i}(g) < 0$  differs in each case listed in \cref{cor:2*5-classification}.
Before diving into the proof of \cref{thm:2*5}, we provide some examples illustrating our main ideas in Cases A, B and C.

In Case A, mostly, we  aim to pick a Fourier template $g$ with $\sigma_{A_{i_0}}(g)$ ``very negative'' for some uncommon $L_{i_0}$, while keeping $\sum_{i\in[5],\, L_i\text{ common}}\sigma_{A_i}(g)$ relatively small.

\begin{example}[Case A]\label{ex:non-coincidental-uncommon}
Let $L$ be the linear system 
\[
\begin{pmatrix}
    1 & 3 & -1 & -3 & 0\\
    2 & 3 & -3 & 0 & -2
\end{pmatrix}.
\]
We then have

\begin{alignat*}{3}
    L_5 \mathbf x &= x_1 + 3x_2 - x_3 - 3x_4\qquad\qquad  &&A_5=\{\!\!\{1,3,-1,-3\}\!\!\}
\\ L_4 \mathbf x &=2x_1 + 3x_2 - 3x_3 - 2x_5 &&A_4=\{\!\!\{2,3,-3,-2\}\!\!\}\\
L_3 \mathbf x &=x_1 + 6x_2 - 9x_4 + 2x_5 &&A_3=\{\!\!\{1,6,-9,2\}\!\!\}  \\ 
L_2 \mathbf x &=x_1 - 2x_3 + 3x_4 - 2x_5 &&A_2=\{\!\!\{1,-2,3,-2\}\!\!\}  \\
L_1 \mathbf x &=3x_2 + x_3 - 6x_4 + 2x_5 &&A_1=\{\!\!\{3,1,-6,2\}\!\!\}.
\end{alignat*}

We wish to find a Fourier template $g:\ZZ\to\CC$ such that
\begin{align*}
    &\sum_{r\in\ZZ}(|g(r)|^2|g(3r)|^2+|g(2r)|^2|g(3r)|^2\\
    &\qquad \quad +g(r)g(6r)g(-9r)g(2r)+g(r)g(-2r)g(3r)g(-2r)+g(3r)g(r)g(-6r)g(2r))
    =\sum_{i=1}^5\sigma_{A_i}(g)<0.
\end{align*}
Observe that the first two terms always have nonnegative contribution. This motivates us to find some $g$ that makes one of the last three terms ``very negative" without making the first two terms ``too positive".

For example, consider the Fourier template $g:\ZZ\to\CC$ defined by
\begin{align*}
    \begin{cases}
    g(2) = g(-2) =g(6) = g(-6) = g(9) = g(-9)= 1\\
    g(1)=g(-1)=-100\\
    g(r)=0 \quad \text{otherwise.}
    \end{cases}
\end{align*}
We then have
\[
\sum_{r\in\ZZ}g(r)g(6r)g(-9r)g(2r)=g(1)g(6)g(-9)g(2)+g(-1)g(-6)g(9)g(-2)=-200,
\]
while other terms
\begin{align*}
    &\sum_{r\in\ZZ}|g(r)|^2|g(3r)|^2=|g(2)|^2|g(6)|^2+|g(-2)|^2|g(-6)|^2=2,\\
    &\sum_{r\in\ZZ}|g(2r)|^2|g(3r)|^2=|g(6)|^2|g(9)|^2+|g(-6)|^2|g(-9)|^2=2,\\
    &\sum_{r\in\ZZ}g(r)g(-2r)g(3r)g(-2r)=
    \sum_{r\in\ZZ}g(3r)g(r)g(-6r)g(2r)=0
\end{align*}
have small absolute value. Altogether, we get that $\sum_{i=1}^5\sigma_{A_i}(g)=-196<0$.
\end{example}

 We will discuss Case A further in \cref{subsec:A}.

In Case B, we continue to work with the Fourier template $g \colon \ZZ \to \CC$. However, the strategy in Case A fails in the presence of an additive quadruple, because setting $|g(r)|$ to be big for any particular $r \in \ZZ$ will  make $|g(r)|^4$, the term contributed by the additive quadruple,  blow-up the fastest.

As such, we need to crucially use the fact that we have at least four uncommon equations, say $A_1, \ldots, A_4$. In Case B, our construction is roughly in two steps: 
\begin{itemize}
    \item We construct our Fourier template to have support in some $S \subseteq \ZZ$ such that there are many nonvanishing terms in $ \sum_{i=1}^{4}\sum_{r \in S} g_{A_i}(r)$. We do this by taking $S$ to be a  \emph{multiplicative grid}, illustrated in the example below.
    \item Next, we choose each $g(r)$ to be complex numbers with appropriate phases that align in a way to ensure that $\sum_{i=1}^{4} \sigma_{A_i}(g)$ is a complex vector with negative real part. 
\end{itemize}

The details of this procedure is also illustrated in the example below.  
\begin{example}[Case B]\label{ex:cos-uncommon}
    Let $L$ be the $2\times 5$ linear system 
\begin{align*}
    \begin{pmatrix}
        1 & -1 & 1 & -1 & 0\\
        -2 & 4 & 3 & 0 & -9
    \end{pmatrix}.
\end{align*}
We then have
\begin{alignat*}{3}
    L_5 \mathbf x &= x_1 - x_2 + x_3 - x_4\qquad\qquad  &&A_5=\{\!\!\{1,1,-1,-1\}\!\!\}
\\ L_4 \mathbf x &=-2x_1 + 4x_2 + 3x_3 - 9 x_5 &&A_4=\{\!\!\{-2,4,3,-9\}\!\!\}\\
L_3 \mathbf x &=5x_1 - 7x_2 - 3x_4 + 9x_5 &&A_3=\{\!\!\{5,-7,-3,9\}\!\!\}  \\ 
L_2 \mathbf x &=2x_1 + 7x_3 - 4x_4 - 9 x_5 &&A_2=\{\!\!\{2,7,-4,-9\}\!\!\}  \\
L_1 \mathbf x &=2x_2 + 5x_3 -2x_4 - 9x_5 &&A_1=\{\!\!\{2,5,-2,-9\}\!\!\}.
\end{alignat*}

Again, we wish to find a Fourier template $g:\ZZ\to\CC$ such that
\begin{align*}\label{eqn:cos-ex-expand}
&\sum_{r\in \ZZ}(|g(r)|^4+ g(-2r)g(4r) g(3r) g(-9r) +  g(5r)g(-7r) g(-3r) g(9r) \\ &\qquad\qquad\qquad  + g(2r) g(7r) g(-4r) g(-9r) + g(2r) g(5r) g(-2r) g(-9r))=\sum_{i=1}^5\sigma_{A_i}(g)<0.
    \end{align*}
Since $L$ contains an additive quadruple, we cannot hope to use a construction similar to \cref{ex:non-coincidental-uncommon}, as 
enlarging  $|g(r)|$ for any particular $r$ would make $\sum_{r\in \ZZ}|g(r)|^4$ the dominant term.

Instead, we build some Fourier template whose support is ``almost closed'' under multiplication by any coefficient in $A_1,\dots,A_5$. Consider the \emph{multiplicative grid} $G$ defined by 
    \[ G = \{ 2^{d_1}3^{d_2}5^{d_3}7^{d_4}: d_1,\dots,d_4 \in \{0,1,\dots\}\}\subseteq\NN.\]
    We can think of $G$ as a four-dimensional grid, with each dimension corresponding to a prime divisor. Here we include $2,3,5,7$  as they are precisely the prime divisors of elements in $A_1\cup\dots\cup A_5$. 
    Consequently, $G$ has the nice property that $\pm 2G, \pm 3G, \pm 4G, \pm 5G, \pm 7G, \pm 9G\subseteq \pm G$; in other words, $\pm G$ is closed under multiplication by any coefficient in $A_1\cup\dots\cup A_5$.
    
    However, a Fourier template must have finite support, and therefore cannot be nonzero on the whole of $\pm G$. Therefore,  we will first build some $g:\pm G\to \CC$, and eventually replace it by its restriction on $\pm G_D$ , where
    \[
    G_D=\{ 2^{d_1}3^{d_2}5^{d_3}7^{d_4}: d_1,\dots,d_4 \in \{0,1,\dots,D\}\}
    \]
    is a finite truncation of $G$.

    Our $g:\pm G\to\CC$ will be of the form
    \[\begin{cases}
        g(2^{d_1}3^{d_2}5^{d_3}7^{d_4})=e((\theta_1d_1+\theta_2d_2+\theta_3d_3+\theta_4d_4)t)\\
        g(-2^{d_1}3^{d_2}5^{d_3}7^{d_4})=e(-(\theta_1d_1+\theta_2d_2+\theta_3d_3+\theta_4d_4)t)
    \end{cases}
    \]
    where $\theta_1,\dots,\theta_4\in\QQ$ and $t\in\RR$.
    In other words, $g$ is the composition of $e(\cdot)$ and a ``linear map" from the multidimensional grid $G$ to $\RR$. This choice of $g$ will ensure that the phases of $\sigma_{A_i}(g)$ line up in a way to ensure that $\sum_{i=2}^{5}\sigma_{A_i}(g)$ is very negative. 

Observe that  $|g(r)|=1$ for all $r\in\pm G$. Moreover, for all $x,y\in G$, we have $g(xy)=g(x)g(y)$ and $g(-xy)=g(-x)g(-y)$. Hence for all $r\in\pm G$, we have
    \begin{align*}
    g(-2r)g(4r)g(3r) g(-9r)=
\begin{cases}
    g(-2)g(4) g(3)g(-9)\cdot g(r)^2 g(-r)^2 = e((\theta_1-\theta_2)t)& r\in G\\
    g(2)g(-4) g(-3)g(9)\cdot g(r)^2 g(-r)^2 = e((\theta_2-\theta_1)t)& r\in -G.
\end{cases}
    \end{align*}
Similarly, we have
\begin{align*}
    g(5r) g(-7r)g(-3r)g(9r) &= \begin{cases}
    e((\theta_2+\theta_3-\theta_4)t)& r\in G\\
    e((-\theta_2-\theta_3+\theta_4)t)& r\in -G
\end{cases}\\
g(2r) g(7r) g(-4r) g(-9r) &= \begin{cases}
    e((-\theta_1-2\theta_2+\theta_4)t)& r\in G\\
    e((\theta_1+2\theta_2-\theta_4)t)& r\in -G
\end{cases}\\
g(2r) g(5r) g(-2r) g(-9r) &= \begin{cases}
    e((-2\theta_2+\theta_3)t)& r\in G\\
    e((2\theta_2-\theta_3)t)& r\in -G.
\end{cases}
\end{align*}
Therefore, for all $r\in\pm G$, we have
\begin{align*}
    \sum_{i=1}^5\text{Re}(g_{A_i}(r))
    &=1+\cos((\theta_1-\theta_2)t)+\cos((\theta_2+\theta_3-\theta_4)t)+\cos((-\theta_1-2\theta_2+\theta_4)t)+\cos ((-2\theta_2+\theta_3)t).
\end{align*}
Setting $\theta_1=\theta_3=1$, $\theta_2=\theta_4=0$ and $t=\pi$, we get that
\begin{align*}
    \sum_{i=1}^5\text{Re}(g_{A_i}(r))&=1+\cos\pi+\cos \pi+\cos(-\pi)+\cos\pi=-3.
\end{align*}
Finally, we switch our attention from $g$ to its restriction on $\pm G_D$. Consider the Fourier template $g_D:\ZZ\to\CC$ defined by
\begin{align*}
    g_D(r)=\begin{cases}
        e((\theta_1d_1+\theta_2d_2+\theta_3d_3+\theta_4d_4)t)& r=2^{d_1}3^{d_2}5^{d_3}7^{d_4}\in G_D\\
        e(-(\theta_1d_1+\theta_2d_2+\theta_3d_3+\theta_4d_4)t)& r=-2^{d_1}3^{d_2}5^{d_3}7^{d_4}\in -G_D\\
        0&\text{otherwise.}
    \end{cases}
\end{align*}

Observe that if $r\notin \pm G_D$, then $(g_D)_{A_i}(r)=0$ for all $i\in [5]$. Moreover, for those $r\in \pm G_D$ such that $r,2r,4r,3r,9r,5r,7r\in \pm G_D$, we have $\sum_{i=1}^5\text{Re}((g_D)_{A_i}(r))=\sum_{i=1}^5\text{Re}(g_{A_i}(r))=-3$. Since
\[
\lim_{D\to\infty}\frac{|\{r\in \pm G_D:r,2r,4r,3r,9r,5r,7r\in G_D\}|}{|\pm G_D|}=1,
\]
we get that
\begin{align*}
    \lim_{D\to\infty}\frac{\sum_{i=1}^5\sigma_{A_i}(g_D)}{|\pm G_D|}=\lim_{D\to\infty}\frac{\sum_{r\in\pm G_D}\sum_{i=1}^5\text{Re}((g_D)_{A_i}(r))}{|\pm G_D|}=-3
\end{align*}
(recall from \cref{rem:real-sum} that $\sigma_{A_i}(g_D)$ is always real). Taking $D$ sufficiently large, we can obtain the desired Fourier template $g_D$ with $\sum_{i=1}^5\sigma_{A_i}(g_D)<0$.
\end{example}

We discuss Case B further in \cref{subsec:B}.
 
Case B covers the case where exactly one of $L_1,\dots,L_5$ is additive quadruple, and the other four are uncommon. What if one of the other four becomes common? The following examples illustrates why the construction in Case B fails and what the alternate strategy is.

\begin{example}[Case C]\label{ex:1AQ-1common}
    Let $L$ be the $2\times 5$ linear system 
\begin{align*}
    \begin{pmatrix}
        1 & 1 & -1 & -1 & 0\\
        1 & -1 & 4 & 0 & -4
    \end{pmatrix}.
\end{align*}
We then have
\begin{alignat*}{3}
    L_5 \mathbf x &= x_1 + x_2 - x_3 - x_4\qquad\qquad  &&A_5=\{\!\!\{1,1,-1,-1\}\!\!\}
\\ L_4 \mathbf x &=x_1-x_2+4x_3-4x_5 &&A_4=\{\!\!\{1,-1,4,-4\}\!\!\}\\
L_3 \mathbf x &=5x_1 + 3x_2 - 4x_4 -4x_5 &&A_3=\{\!\!\{3,5,-4,-4\}\!\!\}  \\ 
L_2 \mathbf x &=2x_1 + 3x_3 - 1x_4 - 4 x_5 &&A_2=\{\!\!\{2,3,-1,-4\}\!\!\}  \\
L_1 \mathbf x &=2x_2 - 5x_2 -1x_4 + 4x_5 &&A_1=\{\!\!\{2,-5,-1,4\}\!\!\}.
\end{alignat*}

    The difference between this example and \cref{ex:cos-uncommon} is that, in addition to the additive quadruple $L_5$, we have another common linear equation $L_4$.
    
    Let us first try the strategy in \cref{ex:cos-uncommon}. Since $2,3,5$ are the prime divisors of at least one coefficient in $L_1,\dots,L_5$, consider the multiplicative grid
    $G = \{ 2^{d_1} 3^{d_2} 5^{d_3}: d_1 ,d_2,d_3 \in \{0,1,\dots\} \}$
    and some $g:\pm G\to\CC$ of the form
    \[\begin{cases}
        g(2^{d_1}3^{d_2}5^{d_3})=e((\theta_1d_1+\theta_2d_2+\theta_3d_3)t)\\
        g(-2^{d_1}3^{d_2}5^{d_3})=e(-(\theta_1d_1+\theta_2d_2+\theta_3d_3)t).
    \end{cases}
    \]
    We wish to find some $\theta_1,\theta_2,\theta_3,t\in\RR$ such that for all $r\in\pm G$, we have
\begin{align*}
    0>\sum_{i=1}^5\text{Re}(g_{A_i}(r))&=|g(r)|^4+ |g(r)|^2|g(4r)|^2 +  g(5r)g(3r) g(-4r) g(-4r) \\ &\qquad + g(2r) g(3r) g(-r) g(-4r) + g(-2r) g(5r) g(r) g(-4r)\\
    &=2+\cos((-4\theta_1+\theta_2+\theta_3)t)+\cos((-\theta_1+\theta_2)t)+\cos((-3\theta_1+\theta_3)t).
\end{align*}

    However, since there are no $\alpha,\beta\in\RR$ with $\cos\alpha+\cos\beta+\cos(\alpha+\beta)<-2$,  we cannot find solutions $\theta_1,\theta_2,\theta_3,t \in \RR$ to the above inequality.

    This motivates us to pursue a different route. 
    Instead of letting $\sum_{i=1}^5\text{Re}(g_{A_i}(r))$ be constant for every $r\in\pm G$, we choose some $g:\pm G\to\CC$ that is ``periodic" in the multiplicative grid, so that $\sum_{i=1}^5\text{Re}(g_{A_i}(r))$ sums up to a negative number  within every period.

    We construct some $g:\pm G_L\to\CC$ so that for all $d_1,d_2,d_3\in\{0,1,\dots\}$, we have $g(2^{d_1}3^{d_2}5^{d_3})=g(2^{d_1+4}3^{d_2+2}5^{d_3+2})$. More precisely, consider
    \begin{align*}
        \begin{cases}
            g(2^{d_1}3^{d_2}5^{d_3})=h(\text{Mod}(d_1,4),\text{Mod}(d_2,2),\text{Mod}(d_3,2))\\
            g(-2^{d_1}3^{d_2}5^{d_3})=\overline{h(\text{Mod}(d_1,4),\text{Mod}(d_2,2),\text{Mod}(d_3,2))},
        \end{cases}
    \end{align*}
where $h:\ZZ_4\times \ZZ_2\times \ZZ_2\to\CC$ is defined by

\[
\begin{matrix*}[l]
    & h(0,0,0)=0.25+0.42i   & h(1,0,0)=0 & h(2,0,0)=-0.35+0.35i  & h(3,0,0)=1-i\\
    & h(0,0,1)=0.35-0.35i & h(1,0,1)=-1-i & h(2,0,1)=0.48+0.12i & h(3,0,1)=0\\
    & h(0,1,0)=-0.35+0.35i  & h(1,1,0)=-1-i &h(2,1,0)=0.48+0.12i  & h(3,1,0)=0\\
    &h(0,1,1)=-0.25-0.42i & h(1,1,1)=0 &h(2,1,1)=0.35-0.35i & h(3,1,1)=1-i.
\end{matrix*}
\]

One can check that, for all $r\in G$, we have
\begin{align*}
    &\sum_{j=0}^3\sum_{k,\ell=0}^1\left(\sum_{i=1}^5 g_{A_i}(2^{j}3^k5^\ell r)\right)\\
    &=\sum_{\substack{j\in\ZZ_4\\k,\ell\in\ZZ_2}}(|h(j,k,\ell)|^4+|h(j,k,\ell)|^2|h(j+2,k,\ell)|+h(j,k,\ell+1)h(j,k+1,\ell)\overline{h(j+2,k,\ell)}^2\\
    &\qquad\quad+h(i+1, j,k)h(i, j+1, k) \overline{h(i,j,k)}\overline{h(i+2,j,k)} + \overline{h(i+1, j,k)}h(i,j,k+1) h(i,j,k) \overline{h(i+2, j,k)})\\
    &=-0.249573+0.723675i.
\end{align*}
Similarly, for all $r\in-G$, we have 
\[\sum_{j=0}^3\sum_{k,\ell=0}^1\left(\sum_{i=1}^5 g_{A_i}(2^{j}3^k5^\ell r)\right)=-0.249573-0.723675i.\]
As before, let $g_D:\ZZ\to\CC$ be the restriction of $g$ on the finite truncation $\pm G_D$, so that $g_D$ is a Fourier template. Since
\[
\lim_{D\to\infty}\frac{|\{r\in \pm G_D:\{1,2,4,8\}\times \{1,3\}\times \{1,5\}\times r\subseteq\pm G_D \}|}{|\pm G_D|}=1,
\]
we get that
\[
\lim_{D\to\infty}\frac{\sum_{i=1}^5\sigma_{A_i}(g_D)}{|\pm G_D|/16}=-0.249573<0.
\]
Taking $D$ sufficiently large,  we can obtain the desired Fourier template $g_D$ with $\sum_{i=1}^5\sigma_{A_i}(g_D)<0$. 
 \end{example}

In general, to prove \cref{thm:2*5} under Case C, we will repeatedly use and generalize the above function $h$. At this point, it might not be intuitive why the values of $h(i,j,k)$ are chosen in this way, and how they can be generalized. We will discuss this further in \cref{subsec:C}, which hopefully will explain more to the readers.

The above three cases form the majority of the proof of \cref{thm:2*5}. Case D does not have many possibilities, as each linear system under Case D can be parameterized by one single variable; we will prove \cref{thm:2*5} under Case D in \cref{subsec:D}. In \cref{subsec:s<4}, we will show \cref{thm:2*5} for $2\times 5$ linear systems with girth $s(L)\leq 3$. Finally, in \cref{sec:common}, we show commonness of the several linear systems indicated in \cref{thm:2*5}.

Having illustrated some of our proof strategies, we now begin the proof of \cref{thm:2*5}.

\subsection{Proof of \cref{thm:2*5} under Case A}\label{subsec:A}

In this section, we prove \cref{thm:2*5} under Case A listed in \cref{cor:2*5-classification}.

\begin{definition}
Suppose $L$ is an irredundant $2\times 5$ system with $s(L)=4$. Recall \cref{setup:s=4}.
    \begin{itemize}
\item For $i\in[5]$ and $\lambda\in\QQ\setminus\{0\}$, we say that $L_i$ is \textit{$\lambda$-common} if  $A_i$ is of the form
$\{\!\!\{a,-a,\lambda a,-\lambda a\}\!\!\}$.
\item For $a,b\in \QQ$, we say that $a,b$ are \emph{$L$-coincidental} if there exist $\lambda\in\{\pm a/b,\pm b/a\}$ and $i \in [5]$ such that $L_i$ is $\lambda$-common.
    \end{itemize}
\end{definition}

Suppose $L$ is a linear system that falls under Case A listed in \cref{cor:2*5-classification}. We show that $L$ is always uncommon. This follows from combining the following two statements:
\begin{enumerate}
    \item (\cref{lem:finding-non-coincidental}) If $L$ falls under Case A in \cref{cor:2*5-classification}, then there exists $A_i$ that cannot be partitioned into two $L$-coincidental pairs. 
    \item (\cref{lem:coincidence-k13,lem:coincidence-isolated}) If for some $i\in[5]$, $A_i$ cannot be partitioned into two $L$-coincidental pairs, then $L$ is uncommon.

    Note that, if $A_i$ cannot be partitioned into two $L$-coincidental pairs, then one of the following occurs:
    \begin{enumerate}
        \item there exists $a\in A_i$ that is not $L$-coincidental with any element in $A_i\setminus\{a\}$,
        \item there exists $a\in A_i$ that is $L$-coincidental with every element in $A_i\setminus\{a\}$, while no two of the three elements in $A_i\setminus\{a\}$ are $L$-coincidental.
    \end{enumerate}

    In particular, \cref{lem:coincidence-isolated} shows that $L$ is uncommon when (a) occurs, and \cref{lem:coincidence-k13} shows that $L$ is uncommon when (b) occurs.
\end{enumerate}

The detailed statements are as follows.

\begin{lemma}\label{lem:finding-non-coincidental}
    Suppose $L$ is an irredundant $2\times 5$ linear system that falls under Case A in \cref{cor:2*5-classification}. Assume \cref{setup:s=4}. Then there exists $i\in[5]$ such that $A_i$ cannot be partitioned into two $L$-coincidental pairs.
\end{lemma}

\begin{lemma}\label{lem:coincidence-isolated}
Suppose $L$ is an irredundant $2\times 5$ linear system that falls under Case A in \cref{cor:2*5-classification}. Assume \cref{setup:s=4}. If there exists $i\in[5]$ and $a\in A_i$ such that $a$ is not $L$-coincidental with any element in $A_i\setminus\{a\}$, then there exists a Fourier template $g:\ZZ^d\to\CC$ such that $\sum_{i=1}^5 \sigma_{A_i}(g)<0$.
\end{lemma}

\begin{lemma}\label{lem:coincidence-k13}
Suppose $L$ is an irredundant $2\times 5$ linear system that falls under Case A in \cref{cor:2*5-classification}. Assume \cref{setup:s=4}. Moreover, suppose:
\begin{enumerate}
    \item for every $i\in[5]$ and every $a\in A_i$, there is some $b\in A_i\setminus\{a\}$ such that $a,b$ are $L$-coincidental.
    \item there exists $i_0\in[5]$, $A_{i_0}=\{\!\!\{a_0,b_0,c_0,d_0\}\!\!\}$ such that:
\begin{itemize}
    \item $a_0$ is $L$-coincidental with each of $b_0,c_0,d_0$,
    \item $b_0,c_0,d_0$ are pairwise not $L$-coincidental.
\end{itemize}
\end{enumerate}
Then there exists a  Fourier template $g:\ZZ\to\CC$ such that $\sum_{i=1}^5 \sigma_{A_i}(g)<0$.
\end{lemma}

The proof of \cref{lem:coincidence-k13} is effectively a generalization of \cref{ex:non-coincidental-uncommon}.

\begin{proof}[Proof of \cref{lem:coincidence-k13}]
Consider the Fourier template $g:\ZZ\to\CC$ defined by
\begin{align*}
    \begin{cases}
    g(a_0)=g(-a_0)=1\\
    g(b_0)=g(-b_0)=g(c_0)=g(-c_0)=g(d_0)=g(-d_0)=-C\\
    g(r)=0 \quad \text{otherwise.}
    \end{cases}
\end{align*}
We then have $\sigma_{A_{i_0}}(g)=-2C^3$.
Moreover, because of condition (1) in the lemma statement, there is no $i\in[5]$ such that $A_i$ has all elements lying in $\{\pm b_0,\pm c_0,\pm d_0\}$. Therefore, for all $i\in[5]$, we either have  $\sigma_{A_i}(g)=-2C^3$ or $|\sigma_{A_i}(g)|=O(C^2)$. This gives $\sum_{i=1}^5 \sigma_{A_i}(g)<0$ when $C>0$ is sufficiently large.
\end{proof}

The proof of \cref{lem:coincidence-isolated} has more case discussions than  \cref{lem:coincidence-k13}, although the underlying principle is identical. Similarly, the proof of \cref{lem:finding-non-coincidental} is routine. As such, we defer the proofs of these lemmas to \cref{app:A}.

\subsection{Proof of \cref{thm:2*5} under Case B}\label{subsec:B}

In this section, we prove \cref{thm:2*5} under Case B of \cref{cor:2*5-classification}. In this case, we always assume the following setup:

\begin{setup}\label{setup:b}
    Suppose
    \begin{alignat*}{3}
    L_5 \mathbf x &=x_1+x_2-x_3-x_4 && A_5=\{\!\!\{1,1,-1,-1\}\!\!\}\\
    L_4 \mathbf x &=ax_1+bx_2+cx_3+dx_5 && A_4=\{\!\!\{a,b,c,d\}\!\!\}\\
    L_3 \mathbf x &=(a+c)x_1+(b+c)x_2-cx_4+dx_5\qquad && A_3=\{\!\!\{a+c,b+c,-c,d\}\!\!\}\\
    L_2 \mathbf x &=(a-b)x_1+(b+c)x_3+bx_4+dx_5 && A_2=\{\!\!\{a-b,b+c,b,d\}\!\!\}\\
    L_1 \mathbf x &=(b-a)x_3+(a+c)x_3+ax_4+dx_5 && A_1=\{\!\!\{b-a,a+c,a,d\}\!\!\}
\end{alignat*}
with $a,b,c,d\in\ZZ\setminus\{0\}$, and none of $L_1,\dots,L_4$ is common. By replacing $a,b,c,d$ by their negations if necessary, we may suppose that at least two of $a,b,c,d$ are positive.
\end{setup}

We aim to show that $L$ is always uncommon. As in \cref{ex:cos-uncommon}, the proof strategy is in two steps:
\begin{enumerate}
    \item First, we construct an appropriate ``multiplicative grid'' $G_L$ which is the support of the Fourier template $g$.
    \item Next, we set the values of $g(r)$ for $r \in G_L$ to be some complex numbers so that the phases align in a way for $\sum_{i=1}^{5} \sigma_{A_i}(g)$ to be negative. We will in fact choose $g $ to be the finite truncation of some $h:\pm G_L\to\CC$ where each $h_{A_i}$ is periodic (in some suitable sense that we will define) on $G_L$. Most of the times $h$ will look like the functions in \cref{ex:cos-uncommon}, given by the composition of $e(\cdot)$ with a linear map from $G_L \to \RR$. 
\end{enumerate}

As such, we begin by defining the ``multiplicative grid'' $G_L$ associated to a $2 \times 5$ linear system $L$ as follows. 

\begin{definition}\label{def:mult-grid} 
Let $L$ be a $2\times 5$ linear system with $s(L)=4$.

\begin{itemize}
    \item  Let $P_L=\{p_1,\dots,p_\ell\}=\{p:p\text{ is a prime divisor of some element in }A_1\cup\dots \cup A_5\}$.
    \item  Let $G_L=\{p_1^{d_1}\cdots p_{\ell}^{d_\ell}:d_1,\dots,d_\ell\in\{0,1,\dots\}\}$.
\end{itemize}

For a function $ h: \pm G_L\to\CC$, $i\in[5]$ and $\vec u=(u_1,\dots,u_\ell)\in\NN^{\ell}$, we say that $h_{A_i}$ is \emph{$\vec u$-periodic} if for all $r\in G_L$, we have
$$h_{A_i}(r)=h_{A_i}(p_1^{u_1}\cdots p_{\ell}^{u_\ell}\cdot r).$$
\end{definition}

\begin{remark}
    We note that while a more natural definition for $\vec{u}$-periodicity  is  $h(r) = h(p_1^{u_1}\cdots p_{\ell}^{u_\ell}\cdot r)$ -- and some of our constructions satisfy this stronger periodicity condition -- we do need the more general notion of periodicity in \cref{def:mult-grid} for the construction given in \cref{ex:cos-uncommon}. 
\end{remark}

Now, we construct a function $h \colon \pm G_L \to \CC$ such that $h_{A_1},\dots,h_{A_5}$ are $\vec{u}$-periodic on $G_L$ and has the ``phase cancellation'' property on a period. That is, 
\begin{equation}\label{eq:grid}
    \mathrm{Re} \sum_{\substack{0 \leq d_1 \leq u_1-1 \\ \vdots \\ 0 \leq d_{\ell} \leq u_{\ell} -1}} \sum_{i = 1}^{5} h_{A_i}(p_1^{d_1} \ldots p_{\ell}^{d_\ell}\cdot r) < 0.
\end{equation}
It turns out that the above is enough to guarantee the existence of a Fourier template $g \colon \ZZ \to \CC$ such that $\mathrm{Re}(\sigma_{A_i}(g)) < 0$. We have seen some version of this idea in \cref{ex:cos-uncommon}. 

\begin{lemma}\label{lem:periodic}
    Let $L$ be a $2 \times 5$ system with $s(L) = 4$ and suppose $G_L$ is the associated multiplicative grid as in \cref{def:mult-grid}. Condier $h \colon \pm G_L \to \CC$ such that $h(-r)=\overline{h(r)}$ for all $r\in G_L$. Suppose $h_{A_1},\dots,h_{A_5}$ are $\vec{u}$-periodic for some $\vec{u} \in \NN^{\ell}$,  such that \cref{eq:grid} holds. Then there exists a Fourier template $g\colon \ZZ \to \CC$ such that $\sum_{i=1}^{5} \sigma_{A_i}(g) < 0$. 
\end{lemma}

We defer the formal proof of \cref{lem:periodic} to \cref{app:B}. However, the main idea is just the last part of \cref{ex:cos-uncommon}.
Roughly speaking, we ``tessellate'' enough copies of $h$ on a sufficiently large truncation $\pm G_D$, where $G_D=\{p_1^{d_1}\cdots p_\ell^{d_\ell}:d_1,\dots,d_\ell\in\{0,1,\dots,D\}\}\subseteq G_L$. Then, we will define the Fourier template $g:\pm G_D\to\CC$ by
\[
g(r)=h(r) \qquad \text{for every }r\in\pm G_D
\]
(so the support of $g$ is $\pm G_D\subseteq \ZZ$, which is finite). By choosing $D$ large, we can ensure that the terms involved in the ``wrap around''/boundary terms in $\sum_{i=1}^{5} \sigma_{A_i}(g)$ is much smaller than the number of main terms coming from terms in \cref{eq:grid}, and so the overall sum stays negative.

As such, in what follows, we restrict our attention to constructing periodic functions on $\pm G_L$ that sum to a negative value in its period. Depending on the properties of $A_1, \ldots, A_5$, we construct $h$ as a suitable product formed from four special ``primitive functions'': $h_1, h_2, h_3, h_4$.

\begin{definition}\label{def:cancelling}
    Let $A$ be a multiset of size 4. We say that $A$ is \textit{cancelling} if $A$ is of the form $\{\!\!\{a_1,a_2,b_1,b_2\}\!\!\}$ such that $a_1,a_2,-b_1,-b_2\in\NN$, with $a_1a_2=b_1b_2$.
\end{definition}

We already know that $A_5=\{\!\!\{1,1,-1,-1\}\!\!\}$. For each $A_i$ among $A_1,A_2,A_3,A_4$, we know that $A_i$ satisfies one of the following:
\begin{enumerate}
    \item[(i)] $A_i$ has four positive elements,
    \item[(ii)] $A_i$ has three positive and one negative elements,
    \item[(iii)] $A_i$ has two positive and two negative elements, but is not cancelling,
    \item[(iv)] $A_i$ is cancelling (but not an additive quadruple).
\end{enumerate}

Observe that if $A_i$ meets (i) or (ii), then we can pick very simple functions $h:\pm G_L\to\CC$ to make $h_{A_i}$ negative.

\begin{observation}\label{obs:b-1-1}
    There exist functions $h_1,h_2:\pm G_L\to\CC$, defined by
    $$\begin{cases}
        h_1(r)=e(1/8) & r\in G_L\\
        h_1(r)=e(-1/8) & r\in -G_L
    \end{cases}\qquad
    \begin{cases}
        h_2(r)=i & r\in G_L\\
        h_2(r)=-i & r\in -G_L
    \end{cases}$$
    such that, for all 4-multisets $A$, $(h_1)_{A}$ and $(h_2)_{A}$ are both constant on $G_L$ with the following values:
    \begin{align*}
        (h_1)_A&=\begin{cases}
            -1 & \text{$A$ has four positive elements}\\
            i\text{ or }-i & \text{$A$ has three positive and one negative elements}\\
            1 & \text{$A$ has two positive and two negative elements,}
        \end{cases}\\
        (h_2)_A&=\begin{cases}
            -1 & \text{$A$ has three positive and one negative elements}\\
            1 & \text{$A$ has two positive and two negative elements.}
        \end{cases}
    \end{align*}
\end{observation}

However, if some $A_i$, $i\in[4]$  has an equal number of positive and negative elements, then we have $(h_1)_{A_i}= (h_2)_{A_i} = 1$, so new constructions are needed when this occurs.

We first discuss the case when all of $A_1,\dots,A_4$ meets (iii). The construction for this case generalizes that in \cref{ex:cos-uncommon}. 

\begin{proposition}\label{prop:b-1-2}
    Let $L$ be a $2 \times 5$ system that falls under Case B in \cref{cor:2*5-classification}. 
    Let $G_L$ be the associated multiplicative grid as given by \cref{def:mult-grid}. Suppose  $A_5$ is an additive quadruple, and each of $A_1, \ldots, A_4$ contains an equal number of positive and negative elements, but is not cancelling.
    
    Then there exists a function $ h_3 \colon \pm G_L \to \CC$ with $h(-r)=\overline{h(r)}$ such that each $(h_3)_{A_i}$, $i \in [5]$  is constant on $G_L$.  Furthermore, $\mathrm{Re}\sum_{i=1}^{5} (h_3)_{A_i}(r) < 0$ for every $r \in \pm G_L$.
\end{proposition}

In the proof of \cref{prop:b-1-2}, we will need the following fact, whose proof we put in \cref{app:B}. \cref{fact:cos-1} is the driving force behind the idea of getting the phases of $h_{A_i}(r)$ to align and produce a negative real value. 
\begin{fact}\label{fact:cos-1}
    For all $\gamma_1,\gamma_2,\gamma_3,\gamma_4\in\QQ\setminus\{0\}$, there exists $t\in\RR$ such that $1+\cos(\gamma_1t)+\cos(\gamma_2t)+\cos(\gamma_3t)+\cos(\gamma_4t)<0$.
\end{fact}

The proof of \cref{prop:b-1-2} generalizes what we saw in \cref{ex:cos-uncommon}. 

\begin{proof}[Proof of \cref{prop:b-1-2}]
    Recall that $$P_L=\{p_1,\dots,p_\ell\}=\{p:p\text{ is a prime divisor of some element in }A_1\cup\dots \cup A_5\}.$$
For some $\theta_1,\dots,\theta_\ell\in\QQ$ and $t\in\RR$ to be chosen later, define $h_3:\pm  G_L\to\CC$ by
$$\begin{cases}
    h_3(p_1^{d_1}\cdots p_\ell^{d_\ell})=  e((\theta_1d_1+\dots+\theta_\ell d_\ell)t)\\
     h_3(-p_1^{d_1}\cdots p_\ell^{d_\ell})=  e(-(\theta_1d_1+\dots+\theta_\ell d_\ell)t).
\end{cases}$$

Observe that, since each of $A_1,\dots,A_5$ contains two positive and two negative elements, $(h_3)_{A_{1}}$, \dots, $(h_3)_{A_{5}}$ are constant on $G_L$. In particular. we have 
$(h_3)_{A_{5}}=1$. Because none of $A_1, \ldots, A_4$ are cancelling, by some elementary linear algebra, we can choose $\theta_1,\dots,\theta_\ell\in\QQ$ so that
\[
\text{Re}\sum_{i=1}^5(h_3)_{A_{i}}(r)=1+\cos(\gamma_1t)+\cos(\gamma_2t)+\cos(\gamma_3t)+\cos(\gamma_4t)\qquad \text{ with }\gamma_1,\gamma_2,\gamma_3,\gamma_4\in \QQ\setminus\{0\}.
\]
By \cref{fact:cos-1}, we can choose $t\in \RR$ so that $1+\cos(\gamma_1t)+\cos(\gamma_2t)+\cos(\gamma_3t)+\cos(\gamma_4t)<0$. This ensures that \cref{eq:grid} holds.
\end{proof}

It is clear that the construction in \cref{prop:b-1-2} fails when some $A_i$, $i\in[4]$ is cancelling, because then $h_{A_i}$ becomes identically 1, and we can no longer leverage \cref{fact:cos-1} to do the phase cancellation. To handle this final case, we can leverage the discrepancy between the largest exponents of primes that divide various elements of $A_i$; that is, $h$ is a function of these largest exponents, for which we use the following notation.
\begin{definition}\label{def:prime-exponent}\quad
\begin{enumerate}
    \item For prime number $p$ and $r\in\ZZ\setminus\{0\}$, define $v_p(r)=\max\{d\in\{0,1,\dots\}:p^d\mid |r|\}$. 
    \item For prime number $p$ and multiset $A$, define $V_{p}(A)=\{\!\!\{v_p(a):a\in A\}\!\!\}$.
\end{enumerate}
\end{definition}

Next, we demonstrate our construction through two examples.

\begin{example}\label{ex:cancelling-1}
    Let $L$ be the $2 \times 5$ linear system 
    \[ \begin{pmatrix}  1 & 1 & -1 & -1 & 0 \\ 1 & - 3 & -4 & 0 & 12 \end{pmatrix}.\]
    We then have 
    \begin{alignat*}{3}
    L_5 \mathbf x &= x_1 + x_2 - x_3 - x_4\qquad\qquad  &&A_5=\{\!\!\{1,1,-1,-1\}\!\!\}
\\ L_4 \mathbf x &= x_1 -3x_2 -4x_3 + 12 x_5 &&A_4=\{\!\!\{1,-3,-4,12\}\!\!\}\\
L_3 \mathbf x &=-3x_1 - 7x_2 +4x_4 + 12x_5 &&A_3=\{\!\!\{-3,-7,4,12\}\!\!\}  \\ 
L_2 \mathbf x &= 4x_1 -7x_3 - 3x_4 + 12 x_5 &&A_2=\{\!\!\{4,-7,-3,12\}\!\!\}  \\
L_1 \mathbf x &= -4x_2 -3x_3 +x_4 + 12x_5 &&A_1=\{\!\!\{-4,-3,1,12\}\!\!\}.
\end{alignat*}
Note that in this context we have that $G_L = \{2^{d_1}3^{d_2}7^{d_3}: d_1,d_2,d_3 \in \{0,1,\dots\} \}$. By \cref{lem:periodic}, to show that $L$ is uncommon, it suffices to find a $\vec{u}$-periodic function $h\colon \pm G_L\to\CC$ with $h(-r)=\overline{h(r)}$ such that \cref{eq:grid} holds.

Note that $v_2(3) = 0, v_2(4) = 2, v_2(12) = 2$ and $v_3(3) = 1, v_3(4) = 0, v_3(12) = 1$. It turns out to be advantageous to consider the largest exponent of 2 and 3, which are the primes that ``witness'' the cancelling structure in $A_4$. In particular, sending $A_i\mapsto \{\!\!\{(v_2(r), v_3(r)):r\in A_i\}\!\!\}$ for each $i$, we have:
    \begin{align*}
        &A_5=\{\!\!\{1,1,-1,-1\}\!\!\}\mapsto \{\!\!\{(0,0),(0,0),(0,0), (0,0)\}\!\!\} \\ 
        &A_4=\{\!\!\{1,-3,-4,12\}\!\!\}\mapsto\{\!\!\{(0,0),(0,1),(2,0), (2,1)\}\!\!\} \\
        &A_3=\{\!\!\{-3,-7,4,12\}\!\!\}\mapsto\{\!\!\{(0,1),(0,0),(2,0), (2,1)\}\!\!\} \\
        &A_2=\{\!\!\{4,-7,-3,12\}\!\!\}\mapsto \{\!\!\{(2,0),(0,0),(1,0), (2,1)\}\!\!\} \\
        &A_1=\{\!\!\{-4,-3,1,12\}\!\!\}\mapsto\{\!\!\{(2,0),(0,1),(0,0), (2,1)\}\!\!\}. 
    \end{align*}
    Consider the following function $h\colon \pm G_L \to \CC$ given by
    \[ h(r) = \begin{cases} -1 &\text{if } \left \lfloor \frac{v_2(r)}{2} \right \rfloor \equiv v_3(r)\equiv 1 \pmod{2} \\ 1 &\text{otherwise.} \end{cases}\]
    Then, since $h(2^4\cdot 3^2\cdot r)=h(r)$ for every $r$, we know that every $h_{A_i}$ is $\vec u$-periodic with $\vec u=(4,2,1)$. Moreover, we have
\begin{align*}
&\sum_{\substack{0 \leq d_1 \leq 3\\ 0 \leq d_{2} \leq 1 \\ r = 2^{d_1} 3^{d_{2}}}}(|h(r)|^4+ h(r)h(-3r) h(-4r) h(12r) +  h(-3r)h(-7r) h(4r) h(12r) \\ &\qquad\qquad\qquad  + h(4r) h(-7r) h(-3r) h(12r) + h(-4r) h(-3r) h(r) h(12r)) \\
&=8-8-8-8-8=-24<0.
    \end{align*}
    Therefore, $h$ satisfies the condition in \cref{lem:periodic}.

\end{example}

\begin{example}
\label{ex:cancelling-2}
    Let $L$ be the $2 \times 5$ linear system 
    \[ \begin{pmatrix}  1 & 1 & -1 & -1 & 0 \\ 6 & 24 & -144 & 0 & -1 \end{pmatrix}.\]
    We then have 
    \begin{alignat*}{3}
    L_5 \mathbf x &= x_1 + x_2 - x_3 - x_4\qquad\qquad\qquad  &&A_5=\{\!\!\{1,1,-1,-1\}\!\!\}
\\ L_4 \mathbf x &= 6x_1 + 24x_2 -144x_3 - x_5 &&A_4=\{\!\!\{6,24,-144, -1\}\!\!\}\\
L_3 \mathbf x &=-138x_1 - 120x_2 + 144x_4 - x_5 &&A_3=\{\!\!\{-138,-120,144,-1\}\!\!\}  \\ 
L_2 \mathbf x &= -18x_1 -120x_3 + 24x_4 - x_5 &&A_2=\{\!\!\{-18,-120,24,-1\}\!\!\}  \\
L_1 \mathbf x &= 18x_2 -138x_3 + 6x_4  - x_5 &&A_1=\{\!\!\{18,-138,6,-1\}\!\!\}
\end{alignat*}
Note that in this context we have $G_L = \{2^{d_1}3^{d_2}5^{d_3}23^{d_4}: d_1, d_2, d_3, d_4 \in \NN \}$. By \cref{lem:periodic}, to show that $L$ is uncommon, it suffices to find a $\vec{u}$-periodic function $h\colon \pm G_L\to\CC$ such that \cref{eq:grid} holds.

    In this case, note that $v_2(6), v_2(24), v_2(144) \geq 1$ and $v_2(6) + v_2(24) = v_2(144)$, and so the cancelling structure in $A_4$ is ``witnessed'' by the prime 2. This stands in contrast to \cref{ex:cancelling-1}, where two primes were needed to ``witness'' this cancelling structure. At present it turns out to be advantageous to isolate the largest exponent of 2 that divide each of the elements of $A_i$. In particular, sending $A_i\mapsto \{\!\!\{v_2(r):r\in A_i\}\!\!\}$ for each $i$, we have:
    \begin{align*}
        &A_5=\{\!\!\{1,1,-1,-1\}\!\!\}\mapsto \{\!\!\{0, 0, 0, 0\}\!\!\} \\
        &A_4=\{\!\!\{6,24,-144, -1\}\!\!\}\mapsto  \{\!\!\{1, 3, 4, 0\}\!\!\}\\
        &A_3=\{\!\!\{-138,-120,144,-1\}\!\!\}\mapsto \{\!\!\{1, 3, 4, 0\}\!\!\}\\
        &A_2=\{\!\!\{-18,-120,24,-1\}\!\!\} \mapsto  \{\!\!\{1,3,3,0\}\!\!\}\\
        &A_1=\{\!\!\{18,-138,6,-1\}\!\!\}\mapsto  \{\!\!\{1,1,1,0\}\!\!\} 
    \end{align*}
     We will choose an appropriate $\phi \colon \NN \to \{ \pm 1 \}$ and show that $h(r) = \phi(v_2(r))$ has the desired properties. We will choose $\phi$ to be periodic, and then it will follow that $h$ is periodic on $G_L$ as well. Consider 
    \[ \phi(v) = \begin{cases} -1 &\text{if } v \equiv 0,1 \pmod{4}, \\ 1 &\text{if } v\equiv 2,3 \pmod{4}. \end{cases} \]
Then, since $h(2^4\cdot r)=h(r)$ for every $r$, we know that every $h_{A_i}$ is $\vec u$-periodic with $\vec u=(4,1,1,1)$. Moreover, for $r=2^{d_1}$ with $d_1=0,1,2,3$, the values of each $h_{A_i}(r)$ varies as follows:
    \begin{itemize}
        \item  $h_{A_5}= 1$, as $$\phi(0)^4=\dots=\phi(3)^4=1.$$
        \item  $h_{A_4}, h_{A_3} =-1$, as $$\phi(1)\phi(3)\phi(0)^2=\phi(2)\phi(0)\phi(1)^2=\phi(3)\phi(1)\phi(2)^2=\phi(0)\phi(2)\phi(3)^2=-1.$$
        \item  $h_{A_1} =h_{A_2} = 1$ when $d_1\in\{0,2\}$ and $h_{A_1} =h_{A_2} = -1$ when $d_1\in\{1,3\}$, as
        $$\phi(1)\phi(0)=\phi(3)\phi(2)=1,\qquad \phi(2)\phi(1)=\phi(0)\phi(3)=-1.$$
    \end{itemize}
Altogether, we have
\begin{align*}
&\sum_{\substack{0 \leq d_1 \leq 3 \\ r = 2^{d_1}}}(|h(r)|^4+ h(6r)h(24r) h(-144r) h(r) +  h(-138r)h(-120r) h(144r) h(-r) \\ &\qquad\qquad\qquad  + h(-18r) h(-120r) h(24r) h(-r) + h(18r) h(-138r) h(6r) h(-r))\\
&=4-4-4+0+0=-4<0.
\end{align*}
    Therefore, $h$ satisfies the condition in \cref{lem:periodic}.
\end{example}

The general result, whose proof we put in \cref{app:B}, is the following.

\begin{proposition}\label{prop:cancelling}
    Suppose $L$ is an irredundant $2\times 5$ linear system that falls under Case B in \cref{cor:2*5-classification} and suppose that for some $i\in[4]$, $A_i$ is a cancelling pair. Then there exists a Fourier template $g\colon \ZZ \to\CC$ such that $\sum_{i=1}^5\sigma_{A_i}(g)<0$.
\end{proposition}

The proof of \cref{prop:cancelling} proceeds in two steps. First, in \cref{cor:classification-case-b} we make precise the notion of ``witnessing the cancelling structure'' in the above examples, by classifying the types of cancelling sets that can occur. Second, we handle each case in the series of \cref{prop:b-2}--\cref{prop:b-5}. The main ideas on how to handle each case generalize those in \cref{ex:cancelling-1} and \cref{ex:cancelling-2}.

We summarize some of the properties of the above constructions in the following table. 
\begin{table}[!htp]
\caption {Periodic sums of $(h_j)_A$ on $G_L$}
\begin{center}
  \begin{tabular}{|l|*{4}{c|}}\hline
 &$h_1$&$h_2$&$h_3$&$h_4$\\\hline
$A$ has four positive elements &$-1$&&&\\\hline
$A$ has three positive and one negative elements &$\pm i$&$-1$&&\\\hline
\makecell[l]{$A$ has two positive and two negative elements \\ but is not cancelling} &1&1&$e^{i\gamma t}$&\\\hline
$A$ is cancelling but not additive quadruple &1&1&1&$-1$\\\hline
Relevant results & \ref{obs:b-1-1} &\ref{obs:b-1-1} &\ref{prop:b-1-2} & \ref{prop:cancelling}\\\hline
\end{tabular}  
\label{table:case-b}
\end{center}
\end{table}

Finally, we put together \cref{obs:b-1-1}, \cref{prop:b-1-2} and \cref{prop:cancelling} to complete the proof of case B. 
\begin{proposition}\label{prop:case-B}
    Suppose $L$ is an irredundant $2\times 5$ linear system that falls under Case B in \cref{cor:2*5-classification}. Then there exists a Fourier template $g:\ZZ^d\to\CC$ such that $\sum_{i=1}^5\sigma_{A_i}(g)<0$.
\end{proposition}

Roughly speaking, depending on the structures of $A_1,\dots,A_5$ that appear in the linear system $L$ -- in particular, which rows in Table~\ref{table:case-b} these $A_i$ belong to -- we choose some subset of $h_j$ as described in Table~\ref{table:case-b} such that the product of each row  restricted to chosen $h_j$ sums up to negative. Then we  take the  Fourier template $g_j$ to be the finite truncation of $h_j$, and finally take $g$ to be the joined Fourier template from the $g_j$.

In the proof of \cref{prop:case-B}, we will need the following analogue of \cref{fact:cos-1} (also proved in \cref{app:B}). 

\begin{fact}\label{fact:cos-2}
    For all $\gamma_1,\gamma_2\in\QQ\setminus\{0\}$, there exists $t\in\RR$ such that $1+\cos(\gamma_1t)+\cos(\gamma_2t)<0$.
\end{fact}

\begin{proof}[Proof of \cref{prop:case-B}]
    If at least one of $A_1,\dots,A_4$ is cancelling then we can  conclude by  \cref{prop:cancelling}. It remains to handle the case when none of $A_1, A_2, A_3, A_4$ is cancelling. 
    \begin{enumerate}
        \item Suppose  two of $A_1, A_2, A_3, A_4$ contains four positive elements and the other two contain three positive and one negative elements. Then we can apply \cref{lem:periodic} and \cref{obs:b-1-1} with  $ h_1:\pm G_L\to\CC$. 
        \item Suppose all four of $A_1, A_2, A_3, A_4$ contain three positive and one negative element. Then we can apply \cref{lem:periodic} and \cref{obs:b-1-1} with  $ h_2:\pm G_L\to\CC$. 
        \item Suppose each of $A_1, \ldots, A_4$ contain two positive and two negative elements.  Then we can apply \cref{lem:periodic} and \cref{prop:b-1-2} with  $ h_3:\pm G_L\to\CC$. 
        \item Suppose that two of $A_1, \ldots A_4$ (say $A_{i_1}, A_{i_2}$) contain three positive and one negative elements, and the other two contain two positive and two negative elements (say $A_{i_3}, A_{i_4}$). We construct two functions $h_3$ and $h_3h_2$, and it will be clear in a moment why we do so. 

            We set $h_3$ to be the function that we constructed in \cref{prop:b-1-2}, but with $\theta_i$ and $t$ chosen according to \cref{fact:cos-2}. More precisely, recall that $$P_L=\{p_1,\dots,p_\ell\}=\{p:p\text{ is a prime divisor of some element in }A_1\cup\dots \cup A_5\}.$$ For some $\theta_1,\dots,\theta_\ell\in\QQ$ and $t\in\RR$ to be chosen later, define $h_3\colon \pm G_L\to\CC$ by $$\begin{cases}
                h_3(p_1^{d_1}\cdots p_\ell^{d_\ell})=  e((\theta_1d_1+\dots+\theta_\ell d_\ell)t)\\
                h_3(-p_1^{d_1}\cdots p_\ell^{d_\ell})=  e(-(\theta_1d_1+\dots+\theta_\ell d_\ell)t)
            \end{cases}$$ Observe that $(h_3)_{A_{i_1}}, (h_3)_{A_{i_2}},(h_3)_{A_{5}}$ are constant on $G_L$, and $(h_3)_{A_{5}}=1$.  Since $A_{i_1}$ and $A_{i_2}$ are not cancelling, we are able to choose $\theta_1,\dots,\theta_\ell\in\QQ$ so that
            \[
(h_3)_{A_{5}}+\text{Re}((h_3)_{A_{i_1}})+\text{Re}((h_3)_{A_{i_2}})=1+\cos(\gamma_1t)+\cos(\gamma_2t),\quad \gamma_1,\gamma_2\in\QQ\setminus\{ 0\}.
\]
        Using \cref{fact:cos-2}, we can choose $t \in \RR$ so that $1+ \cos(\gamma_1 t) + \cos(\gamma_2 t) < 0$. 

        Now consider $h_3h_2$. Since $(h_2)_{A_{i_1}}, (h_2)_{A_{i_2}}, (h_2)_{A_{i_5}} = 1$ and $(h_2)_{A_{i_3}} = (h_2)_{A_{i_4}} = -1$ (as summarized in Table~\ref{table:case-b}),  we have either 
        \[ \mathrm{Re} \sum_{i=1}^{5} h_3(r) < 0 \quad \text{or} \quad \mathrm{Re}\sum_{i=1}^{5}  h_3(r)h_2(r) < 0\]
        for every $r\in \pm G_L$. Applying \cref{lem:periodic} gives the result. 
        \end{enumerate}
\end{proof}

\subsection{Proof of \cref{thm:2*5} under Case C}\label{subsec:C}

In this section, we prove \cref{thm:2*5} under Case C listed in \cref{cor:2*5-classification}. 

Suppose $L$ falls under Case C  in \cref{cor:2*5-classification}. We show that, if $L$ is not isomoprhic to
    \begin{align*}
    &\begin{pmatrix}
            1 & 1 & -1 & -1 & 0\\
            1 & -1 & 3 & 0 & -3
        \end{pmatrix}\quad
        \begin{pmatrix}
            1 & 1 & -1 & -1 & 0\\
            2 & -2 & 3 & 0 & -3
        \end{pmatrix}
    \end{align*}
then $L$ is common if it is isomorphic to
\[
\begin{pmatrix}
            1 & 1 & -1 & -1 & 0\\
            2 & -2 & 1 & 0 & -1
        \end{pmatrix}
\]
and uncommon otherwise.

Suppose $L$ is an irredundant $2\times 5$ linear system under Case C listed in \cref{cor:2*5-classification}. In this case, we always assume the following setup:
\begin{setup}\label{setup:c}
    Suppose
\begin{alignat*}{2}
    L_5 \mathbf x &=x_1+x_2-x_3-x_4 && A_5=\{\!\!\{1,1,-1,-1\}\!\!\}\\
    L_4 \mathbf x &=bx_1-bx_2+ax_3-ax_5 && A_4=\{\!\!\{1,-1,\frac{a}{b},-\frac{a}{b}\}\!\!\}\\
    L_3 \mathbf x &=(a+b)x_1+(a-b)x_2-ax_4-ax_5 \qquad  && A_3=\{\!\!\{\frac{a+b}{b},\frac{a-b}{b},-\frac{a}{b},-\frac{a}{b}\}\!\!\}\\
    L_2 \mathbf x &=2bx_1+(a-b)x_3-bx_4-ax_5 && A_2=\{\!\!\{2,\frac{a-b}{b},-1,-\frac{a}{b}\}\!\!\}\\
    L_1 \mathbf x &=2bx_2-(a+b)x_3-bx_4+ax_5 && A_1=\{\!\!\{2,-\frac{a+b}{b},-1,\frac{a}{b}\}\!\!\}
\end{alignat*}
with $a,b\in\NN$ coprime, and none of $L_1,L_2,L_3$ is common.
\end{setup}

In order to build some intuition for this case, we revisit \cref{ex:1AQ-1common} and reframe the construction in a manner which generalizes easily. 

\begin{example}[Take two of \cref{ex:1AQ-1common}]
    We recall the setup.  Let $L$ be the linear system 
    \[ \begin{pmatrix}
        1 & 1 & -1 & -1 & 0\\
        1 & -1 & 4 & 0 & -4
    \end{pmatrix}.\]
We then have
\begin{alignat*}{3}
    L_5 \mathbf x &= x_1 + x_2 - x_3 - x_4\qquad\qquad  &&A_5=\{\!\!\{1,1,-1,-1\}\!\!\}
\\ L_4 \mathbf x &=x_1-x_2+4x_3-4x_5 &&A_4=\{\!\!\{1,-1,4,-4\}\!\!\}\\
L_3 \mathbf x &=5x_1 + 3x_2 - 4x_4 -4x_5 &&A_3=\{\!\!\{3,5,-4,-4\}\!\!\}  \\ 
L_2 \mathbf x &=2x_1 + 3x_3 - 1x_4 - 4 x_5 &&A_2=\{\!\!\{2,3,-1,-4\}\!\!\}  \\
L_1 \mathbf x &=2x_2 - 5x_2 -1x_4 + 4x_5 &&A_1=\{\!\!\{2,-5,-1,4\}\!\!\}.
\end{alignat*}

We have the corresponding multiplicative grid
\[ G_L = \{ 2^{d_1} 3^{d_2} 5^{d_3}: d_1 ,d_2,d_3 \in \{0,1,\dots\} \},\]
and the goal is to construct a $\vec{u}$-periodic $h\colon \pm G_L \to\CC$ with $h(-r)=\overline{h(r)}$ and the property \cref{eq:grid}:
\begin{equation*}
    \mathrm{Re} \sum_{\substack{0 \leq d_1 \leq u_1-1 \\  0 \leq d_2 \leq u_2-1\\ 0 \leq d_3 \leq u_3-1}} \sum_{i = 1}^{5} h_{A_i}(2^{d_1}3^{d_2}5^{d_3}\cdot r) < 0.
\end{equation*}

\begin{figure}[!htp]
    \centering
    \begin{tikzpicture}[scale=0.9]
\draw[draw=gray] (0,0) rectangle ++(6,4);
\draw[draw=gray] (0,2) to (6,2);
\draw[draw=gray] (0,1) to (6,1);
\draw[draw=gray] (0,3) to (6,3);
\draw[draw=gray] (3,0) to (3,4);
\node [label={$1$}] at (1.5,3) {};
\node [label={$3$}] at (1.5,2) {};
\node [label={$5$}] at (1.5,1) {};
\node [label={$3\cdot 5$}] at (1.5,0) {};
\node [label={$0.25+0.42i$}] at (4.5,3) {};
\node [label={$0.35-0.35i$}] at (4.5,2) {};
\node [label={$-0.35-0.35i$}] at (4.5,1) {};
\node [label={$-0.25-0.42i$}] at (4.5,0) {};

\draw[draw=gray] (0,-5) rectangle ++(6,4);
\draw[draw=gray] (0,-3) to (6,-3);
\draw[draw=gray] (0,-4) to (6,-4);
\draw[draw=gray] (0,-2) to (6,-2);
\draw[draw=gray] (3,-5) to (3,-1);
\node [label={$4$}] at (1.5,-2) {};
\node [label={$4\cdot 3$}] at (1.5,-3) {};
\node [label={$4\cdot 5$}] at (1.5,-4) {};
\node [label={$4\cdot 3\cdot 5$}] at (1.5,-5) {};
\node [label={$-0.35-0.35i$}] at (4.5,-2) {};
\node [label={$0.48+0.12i$}] at (4.5,-3) {};
\node [label={$0.48+0.12i$}] at (4.5,-4) {};
\node [label={$0.35-0.35i$}] at (4.5,-5) {};

\draw[draw=gray,fill=yellow!50] (8,0) rectangle ++(6,4);
\draw[draw=gray] (8,2) to (14,2);
\draw[draw=gray] (8,1) to (14,1);
\draw[draw=gray] (8,3) to (14,3);
\draw[draw=gray] (11,0) to (11,4);
\node [label={$2$}] at (9.5,3) {};
\node [label={$2\cdot 3$}] at (9.5,2) {};
\node [label={$2\cdot 5$}] at (9.5,1) {};
\node [label={$2\cdot 3\cdot 5$}] at (9.5,0) {};
\node [label={$0$}] at (12.5,3) {};
\node [label={$-1-i$}] at (12.5,2) {};
\node [label={$-1-i$}] at (12.5,1) {};
\node [label={$0$}] at (12.5,0) {};

\draw[draw=gray,fill=yellow!50] (8,-5) rectangle ++(6,4);
\draw[draw=gray] (8,-3) to (14,-3);
\draw[draw=gray] (8,-4) to (14,-4);
\draw[draw=gray] (8,-2) to (14,-2);
\draw[draw=gray] (11,-5) to (11,-1);
\node [label={$8$}] at (9.5,-2) {};
\node [label={$8\cdot 3$}] at (9.5,-3) {};
\node [label={$8\cdot 5$}] at (9.5,-4) {};
\node [label={$8\cdot 3\cdot 5$}] at (9.5,-5) {};
\node [label={$1-i$}] at (12.5,-2) {};
\node [label={$0$}] at (12.5,-3) {};
\node [label={$0$}] at (12.5,-4) {};
\node [label={$1-i$}] at (12.5,-5) {};

\draw [<->, very thick, green!50!black!80] (2,2.7) -- (2,3.3);
\draw [<->, very thick, green!50!black!80] (2,0.7) -- (2,1.3);
\draw [<->, very thick, green!50!black!80] (2,-2.3) -- (2,-1.7);
\draw [<->, very thick, green!50!black!80] (10,-2.3) -- (10,-1.7);
\draw [<->, very thick, green!50!black!80] (2,-4.3) -- (2,-3.7);
\draw [<->, very thick, green!50!black!80] (10,2.7) -- (10,3.3);
\draw [<->, very thick, green!50!black!80] (10,0.7) -- (10,1.3);
\draw [<->, very thick, green!50!black!80] (10,-4.3) -- (10,-3.7);
\draw[<->, very thick, dashed, blue!70] (0,3.5) to [in=130,out=-130] (0,1.5);
\draw[<->, very thick, dashed, blue!70] (0,2.5) to [in=130,out=-130] (0,0.5);
\draw[<->, very thick, dashed, blue!70] (0,-1.5) to [in=130,out=-130] (0,-3.5);
\draw[<->, very thick, dashed, blue!70] (0,-2.5) to [in=130,out=-130] (0,-4.5);
\draw[<->, very thick, dashed, blue!70] (14,3.5) to [in=50,out=-50] (14,1.5);
\draw[<->, very thick, dashed, blue!70] (14,2.5) to [in=50,out=-50] (14,0.5);
\draw[<->, very thick, dashed, blue!70] (14,-1.5) to [in=50,out=-50] (14,-3.5);
\draw[<->, very thick, dashed, blue!70] (14,-2.5) to [in=50,out=-50] (14,-4.5);
\draw[<->, very thick, dotted, red!100!black!50] (6,3.5) to [in=65,out=-65] (6,-1.5);
\draw[<->, very thick, dotted, red!100!black!50] (6,2.5) to [in=65,out=-65] (6,-2.5);
\draw[<->, very thick, dotted, red!100!black!50] (6,1.5) to [in=65,out=-65] (6,-3.5);
\draw[<->, very thick, dotted, red!100!black!50] (6,0.5) to [in=65,out=-65] (6,-4.5);

\draw[<->, very thick, dotted, red!100!black!50] (8,3.5) to [in=115,out=-115] (8,-1.5);
\draw[<->, very thick, dotted, red!100!black!50] (8,2.5) to [in=115,out=-115] (8,-2.5);
\draw[<->, very thick, dotted, red!100!black!50] (8,1.5) to [in=115,out=-115] (8,-3.5);
\draw[<->, very thick, dotted, red!100!black!50] (8,0.5) to [in=115,out=-115] (8,-4.5);
\end{tikzpicture}
    \caption{An illustration of the ``orbits'' of $h$. The pink/dotted arrow represents the action of multiplying by 4, the green/solid arrow represents the action of multiply by 3 and the blue/dashed arrow represents the action of multiplying by 5.}
    \label{fig:caseC}
\end{figure}
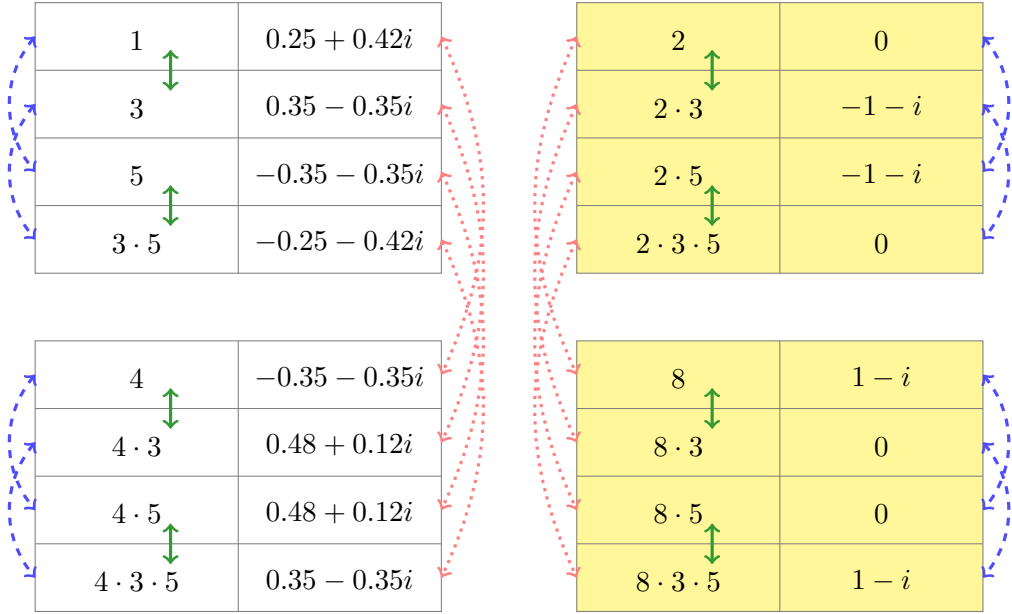

Figure~\ref{fig:caseC} illustrates our construction.
In particular, we chose $\vec{u} = (4,2,2)$, which means that $h(2^4\cdot 3^2\cdot 5^2\cdot r)=h(r)$ for every $r\in G_L$. The left column denotes the value of $h(r)$ when $r=2^i3^j5^k$ with $i\in\{0,2\}$ and $j,k\in\ZZ_2$; right left column denotes the value of $h(r)$ when $r=2^i3^j5^k$ with $i\in\{1,3\}$ and $j,k\in\ZZ_2$. And as always, we take $h(-r)=\overline{h(r)}$ for every $r\in G_L$.

The key observation is that, since $h_{A_1}(r),h_{A_2}(r), h_{A_4}(r)$ are nonvanishing only if $h(r)$, $h(4r)$ are both nonzero, we have $h_{A_1}(r)= h_{A_2}(r)= h_{A_4}(r) = 0$ if $r$ is any of the yellow/shaded elements. This gives
\begin{align*}
    \sum_{r \text{ yellow}} \sum_{i=1}^{5}h_{A_i}(r) &= \sum_{r \text{ yellow}} \left( h_{A_5}(r) + h_{A_3}(r) \right) \\
    &= 2|1-i|^4+2|-1-i|^4+2(-1-i)^2(1+i)^2+2(1-i)^2(-1+i)^2 = 0.
\end{align*}

We draw attention to the fact that multiplication by 3 and 5 create orbits of length two. This fact, coupled with the choice of values of $h$ on the yellow/shaded elements, is essential in ensuring that the contributions from $h_{A_3}$ cancels out the contributions of $h_{A_5}$ on the yellow/shaded elements. In effect the only contribution to the sum $ \sum_{\substack{0 \leq d_1 \leq 3 \\  0 \leq d_2 \leq 1\\ 0 \leq d_3 \leq 1}} \sum_{i = 1}^{5} h_{A_i}(2^{d_1} 3^{d_2} 5^{d_3})$ comes from $\sum_{i=1}^{5} h_{A_i}(r)$ when $r$ are the white elements on the left. This final part is a computation:   
\[ \text{Re}\sum_{r \text{ white}} \sum_{i=1}^{5}h_{A_i}(r) = 0.474101+0.474026-0.2401-0.4788-0.4788 = -0.249573. \]
\end{example}

To summarize, we can think of our construction as a combination of two components:
\begin{itemize}
    \item A function $f \colon \ZZ_{4} \times \ZZ_2 \times \ZZ_2 \to \CC$ represented by the table of values in \cref{fig:caseC}.
    \item A function $\phi \colon  G_L \to \ZZ_{4} \times \ZZ_2 \times \ZZ_2$ with the ``orbit structure'' given by: 
    \begin{itemize}
        \item[$\circ$] $\phi(4r) = \phi(r) + (2,0,0)$,
        \item[$\circ$] $\phi(3r) = \phi(r) + (0,1,0)$,
        \item[$\circ$] $\phi(5r) = \phi(r) + (0,0,1)$,
        \item[$\circ$] $\phi(2r) = \phi(r) + (1,0,0)$ when $\phi(r)_1 \in \{0, 2 \}$.
    \end{itemize}
\end{itemize}
The Fourier template in \cref{ex:1AQ-1common} can therefore be thought of as $h = f \circ \phi$. The combination of these two elements ensure that the yellow/shaded cells cancel in contribution to give a net zero contribution to the sum 
\begin{equation*}
    \mathrm{Re} \sum_{\substack{0 \leq d_1 \leq u_1-1 \\ \vdots \\ 0 \leq d_{\ell} \leq u_{\ell} -1}} \sum_{i = 1}^{5} h_{A_i}(p_1^{d_1} \ldots p_{\ell}^{d_\ell}\cdot r),
\end{equation*}
while the white cells contribute a negative value. Using such a framework, we can generalize our construction to handle other systems $L$ in Case C. In the first step, we show how to extend \cref{fig:caseC} to a larger table, which gives us more flexibility with constructing the function $\phi$ in the second step. This is the content of the following \cref{prop:c-reduction}.

\begin{proposition}\label{prop:c-reduction}
    Suppose $L$ is an irredundant $2\times 5$ linear system that falls under Case C in \cref{cor:2*5-classification}. Assume \cref{setup:c}. If there exists $m\geq 2$ and $\phi:G_L\to\ZZ_{2m}\times \ZZ_2\times \ZZ_2$ such that
   \begin{align*}
       \phi(\frac{a}{b}\cdot r)&=\phi(r)+(m,0,0)\\
       \phi(\frac{|a-b|}{b}\cdot r)&=\phi(r)+(0,1,0)\\
       \phi(\frac{a+b}{b}\cdot r)&=\phi(r)+(0,0,1)\\
       \phi(2 r)&=\phi(r)+(1,0,0)\text{ when $\phi(r)_1\in\{0,m\}$,}
   \end{align*}
   Then there exists a Fourier template $g:\ZZ\to\CC$ such that $\sum_{i=1}^5\sigma_{A_i}(g)<0$.
\end{proposition}

Vaguely speaking, we can think of \cref{prop:c-reduction} as ``extending'' the table in \cref{table:case-b} by duplicating more columns of yellow/shaded cells in \cref{fig:caseC} and leaving the white cells intact. 

\begin{proof}
    Suppose such $\phi$ exists.  We first consider the case $a-b>0$. In this case, define $h=\phi\circ f:G_L\to\CC$, where $f:\ZZ_{2m}\times \ZZ_2\times \ZZ_2\to\CC$ is given by
    \begin{alignat*}{2}
    & (0,0,0)\mapsto 0.25+0.42i\qquad  && (1,0,0), \dots,(m-1,0 ,0) \mapsto 0\\
    &(0,0,1) \mapsto 0.35-0.35i\qquad  && (1,0,1),\dots, (m-1,0,1) \mapsto-1-i\\
    &(0,1,0) \mapsto-0.35+0.35i\qquad  && (1,1,0),\dots,(m-1,1,0) \mapsto -1-i\\
    &(0,1,1) \mapsto -0.25-0.42i\qquad  && (1,1,1), \dots, (m-1,1,1) \mapsto0 \\
    &(m,0,0)\mapsto -0.35+0.35i\qquad  && (m+1,0,0),\dots, (2m-1,0,0)\mapsto 1-i\\
    &(m,0,1) \mapsto 0.48+0.12i\qquad  && (m+1,0,1),\dots, (2m-1,0,1)\mapsto 0\\
    &(m,1,0) \mapsto 0.48+0.12i\qquad  && (m+1,1,0),\dots, (2m-1,1,0)\mapsto 0\\
    &(m,1,1) \mapsto 0.35-0.35i \qquad  && (m+1,1,1),\dots, (2m-1,1,1)\mapsto 1-i.
\end{alignat*}
Take $\vec u=(u_1,\dots,u_\ell)\in \NN^{\ell}$ such that $\phi(p_1^{u_1}\dots p_\ell^{u_\ell}r)=\phi(r)$. Then $h_{A_1},\dots,h_{A_5}$ are $\vec u$-periodic. 
Moreover, we have
\begin{align*}
    \sum_{\substack{0 \leq d_1 \leq u_1-1 \\ \cdots \\ 0 \leq d_{\ell} \leq u_{\ell} -1}} \sum_{i = 1}^{5} h_{A_i}(p_1^{d_1} \ldots p_{\ell}^{d_\ell}\cdot r)
    =\sum_{i,j,k\in\ZZ_{2m}\times \ZZ_2\times \ZZ_2}&|f(i,j,k)|^4+|f(i,j,k)|^2|f(i+m,j,k)|^2\\
    &+f(i,j,k+1)f(i,j+1,k)\overline{f(i+m,j,k)}^2\\
    &+f(i+1,j,k)f(i,j+1,k)\overline{f(i,j,k)f(i+m,j,k)}\\
    &+f(i+1,j,k)\overline{f(i,j,k+1)f(i,j,k)}f(i+m,j,k).
\end{align*}
A quick observation is that, if $i\notin\{0,m\}$, then one of $f(i,j,k)$ and $f(i+m,j,k)$ equals 0, so the second, fourth and fifth terms above vanish. Moreover, the first and third terms sum up to
\begin{align*}
    2|1-i|^4+2|-1-i|^4+2(-1-i)^2(1+i)^2+2(1-i)^2(-1+i)^2=0.
\end{align*}
Therefore, the above sum reduces to
\begin{align*}
     \sum_{i,j,k\in\{0,m\}\times \ZZ_2\times \ZZ_2}&|f(i,j,k)|^4+|f(i,j,k)|^2|f(i+m,j,k)|^2+f(i,j,k+1)f(i,j+1,k)\overline{f(i+m,j,k)}^2\\
    &+f(i+1,j,k)f(i,j+1,k)\overline{f(i,j,k)f(i+m,j,k)}\\
    &+f(i+1,j,k)\overline{f(i,j,k+1)f(i,j,k)}f(i+m,j,k)
\end{align*}
which equals $-0.249573 + 0.723675 i$. Hence $h$ has property \cref{eq:grid}. By \cref{lem:periodic}, this gives Fourier template $g:\ZZ\to\CC$ such that $\sum_{i=1}^5\sigma_{A_i}(g)<0$.

Similarly, when $a-b<0$ (so $h(\frac{a-b}{b}\cdot r)=\overline{h(\frac{|a-b|}{b}\cdot r)}$ for $r\in G_L$), we define $h=\phi\circ f':G_L\to\CC$ where $f':\ZZ_{2m}\times \ZZ_2\times \ZZ_2\to\CC$ is given by
\begin{align*}
    f'(i,j,k)=\begin{cases}
        f(i,j,k) & j=0\\
        \overline{f(i,j,k)} & j=1.
    \end{cases}
\end{align*}
A similar argument shows that  $h_{A_1},\dots,h_{A_5}$ are $\vec u$-periodic, with property  \cref{eq:grid}. By \cref{lem:periodic}, this gives Fourier template $g:\ZZ\to\CC$ such that $\sum_{i=1}^5\sigma_{A_i}(g)<0$.
\end{proof}

In \cref{lem:phi}, we show how to construct the corresponding function $\phi \colon \ZZ_{2m} \times \ZZ_2 \times \ZZ_2$ for all but five specific linear systems under Case C in \cref{cor:2*5-classification}. Combining this with \cref{prop:c-reduction} proves that all systems that fall under Case C, barring these five exceptions, are uncommon. 

To do so, we first extend the function $v_p(\cdot):\ZZ\setminus\{0\}\to\ZZ$ introduced in \cref{def:prime-exponent} to one whose domain is the set of all nonzero rational numbers.

\begin{definition}\label{def:prime-exponent-rational}
    For every prime number $p$, define the function $v_p(\cdot):\QQ\setminus\{0\}\to\ZZ$ as follows:
    \begin{enumerate}
        \item for all $r\in \ZZ\setminus\{0\}$, define $v_p(r)=\max\{d\in\{0,1,\dots\}:p^d\mid |r|\}$;
        \item for all $r\in \QQ\setminus\ZZ$, define $v_p(r)=v_p(x)-v_p(y)$, where $r=x/y$ and $x,y\in\ZZ\setminus\{0\}$.
    \end{enumerate}
\end{definition}

\begin{lemma}\label{lem:phi}
    Suppose $L$ is a linear system that falls under Case C in \cref{cor:2*5-classification}. Then, unless $L$ is isomorphic to one of the following five linear systems
    \begin{align*}
        &\begin{pmatrix}
            1 & 1 & -1 & -1 & 0\\
            1 & -1 & 3 & 0 & -3
        \end{pmatrix}\quad
        \begin{pmatrix}
            1 & 1 & -1 & -1 & 0\\
            2 & -2 & 3 & 0 & -3
        \end{pmatrix}\\ 
        &\begin{pmatrix}
            1 & 1 & -1 & -1 & 0\\
            2 & -2 & 1 & 0 & -1
        \end{pmatrix}
        \quad\begin{pmatrix}
            1 & 1 & -1 & -1 & 0\\
            3 & -3 & 1 & 0 & -1
        \end{pmatrix}\quad
        \begin{pmatrix}
            1 & 1 & -1 & -1 & 0\\
            3 & -3 & 2 & 0 & -2
        \end{pmatrix},
    \end{align*}
    there exists $m\geq 2$ and $\phi:G_L\to\ZZ_{2m}\times \ZZ_2\times\ZZ_2$ that meets the condition in \cref{prop:c-reduction}.
\end{lemma}
\begin{proof}We work under \cref{setup:c}.
    One can check that, if $L$ is not isomorphic to one of the above five systems, then one of the following occurs:
    \begin{itemize}
        \item There exist odd primes $p_1,p_2,p_3$ such that
        \begin{align*}
            &v_{p_1}(a/b)\neq 0,&& v_{p_1}(|a-b|/b)=v_{p_1}((a+b)/b)=0,\\
            &v_{p_2}(|a-b|/b)\neq 0,&& v_{p_2}(a/b)=v_{p_2}((a+b)/b)=0,\\
            &v_{p_3}((a+b)/b)\neq 0,&& v_{p_3}(a/b)=v_{p_3}(|a-b|/b)=0.
        \end{align*}
        In this case, we define $\phi:G_L\to\ZZ_{4}\times \ZZ_2\times \ZZ_2$ by 
        $$\phi(r)=(2\floor{\frac{v_{p_1}(r)}{v_{p_1}(a/b)}}+v_2(r), \floor{\frac{v_{p_1}(r)}{v_{p_2}(|a-b|/b)}}, \floor{\frac{v_{p_3}(r)}{v_{p_3}((a+b)/b)}}),$$
        so that $\phi$ meets the condition in \cref{prop:c-reduction}.

        \noindent E.g., when $a=10$ and $b=1$, we define $$\phi(r)=(2v_{5}(r)+v_2(r),\floor{v_3(r)/2}, v_{11}(r)).$$

        \item There exist odd primes $p_2,p_3$ such that
        \begin{align*}
            &v_{p_2}(|a-b|/b)\neq 0,&& v_{p_2}(a/b)=v_{p_2}((a+b)/b)=0,\\
            &v_{p_3}((a+b)/b)\neq 0,&& v_{p_3}(a/b)=v_{p_3}(|a-b|/b)=0.
        \end{align*}
        Moreover, we have
        $m:=|v_2(a/b)|\geq 2$.
        
        In this case, we define $\phi:G_L\to\ZZ_{2m}\times \ZZ_2\times \ZZ_2$ by $$\phi(r)=(v_2(r), \floor{\frac{v_{p_1}(r)}{v_{p_2}(|a-b|/b)}}, \floor{\frac{v_{p_3}(r)}{v_{p_3}((a+b)/b)}}),$$
        so that $\phi$ meets the condition in \cref{prop:c-reduction}.

        \noindent E.g., in \cref{ex:1AQ-1common}, when $a=4$ and $b=1$, we define $$\phi(r)=(v_2(r),v_3(r),v_5(r)).$$
        
        \item There exist odd primes $p_1,p_2$ such that
        \begin{align*}
            &v_{p_1}(a/b)\neq 0,&& v_{p_1}(|a-b|/b)=v_{p_1}((a+b)/b)=0,\\
            &v_{p_2}(|a-b|/b)\neq 0,&& v_{p_2}(a/b)=v_{p_2}((a+b)/b)=0.
        \end{align*}
        Moreover, we have
        $m:=|v_2((a+b)/b)|\geq 2$.
        
        In this case, we define $\phi:G_L\to\ZZ_{2m}\times \ZZ_2\times \ZZ_2$ by $$\phi(r)=(mv_{p_1}(r)+\text{Mod}(v_2(r),m), \floor{\frac{v_{p_2}(r)}{v_{p_2}(|a-b|/b)}},\floor{\frac{v_{2}(r)}{m}}),$$
        so that $\phi$ meets the condition in \cref{prop:c-reduction}.

        \noindent E.g., when $a=7$ and $b=1$, we define $$\phi(r)=(3v_{7}(r)+\text{Mod}(v_2(r),3),v_3(r), \floor{v_2(r)/3}).$$

        \item There exist odd primes $p_1,p_3$ such that
        \begin{align*}
            &v_{p_1}(a/b)\neq 0,&& v_{p_1}(|a-b|/b)=v_{p_1}((a+b)/b)=0,\\
            &v_{p_3}((a+b)/b)\neq 0,&& v_{p_1}(a/b)=v_{p_2}(|a-b|/b)=0.
        \end{align*}
        Moreover, we have
        $m:=|v_2(|a-b|/b)|\geq 2$.
        
        In this case, we define $\phi:G_L\to\ZZ_{2m}\times \ZZ_2\times \ZZ_2$ by $$\phi(r)=(mv_{p_1}(r)+\text{Mod}(v_2(r),m), \floor{\frac{v_{2}(r)}{m}}, \floor{\frac{v_{p_3}(r)}{v_{p_3}((a+b)/b)}}),$$
        so that $\phi$ meets the condition in \cref{prop:c-reduction}.

        \noindent E.g.,  when $a=5$ and $b=1$, we define $$\phi(r)=(2v_{5}(r)+\text{Mod}(v_2(r),2), \floor{v_2(r)/2},v_3(r)).$$
        
        \item  $|a-b|=2$ and $a+b=2^m$, with $m\geq 3$. Moreover, there exist odd primes $p_1,p_2$ such that
        \begin{align*}
            &v_{p_1}(a)\neq 0,&& v_{p_1}(b)=0,\\
            &v_{p_2}(b)\neq 0,&& v_{p_2}(a)=0.
        \end{align*}
        In this case, we define $\psi:G_L\to\ZZ_{2m-2}\times \ZZ_2\times \ZZ_2$ by
        \begin{align*}
            \psi(r)=(u(r)(m-1)+\text{Mod}(w(r)+u(r),m-1),\, \floor{\frac{v_{p_2}(r)}{v_{p_2}(b)}}-u(r), \, \floor{\frac{w(r)+u(r)}{m-1}})
        \end{align*}
    where
    $$u(r)=\floor{\frac{v_{p_1}(r)}{v_{p_1}(a)}},\quad  w(r)=v_2(r)-\floor{\frac{v_{p_2}(r)}{v_{p_2}(b)}}.$$
    Then, we take $\phi(r)=(\psi(r)_1, \psi(r)_2-\psi(r)_3, \psi(r)_3)$.
    
    \noindent E.g., when $a=5$ and $b=3$, we define 
    \[
    \psi(r)=(2v_5(r)+\text{Mod}(v_2(r)+v_5(r)-v_3(r),2),\, v_3(r)-v_5(r),\,\floor{\frac{v_2(r)+v_5(r)-v_3(r)}{2}}).
    \]
    \end{itemize}
\end{proof}

Finally, for the five specific linear systems not covered by \cref{lem:phi}, we will show in \cref{sec:common} that
    \[
    \begin{pmatrix}
            1 & 1 & -1 & -1 & 0\\
            2 & -2 & 1 & 0 & -1
        \end{pmatrix}\]
is common, and show in the following lemma that two of the other four are uncommon (via simple periodic constructions).  This finishes the proof of \cref{thm:2*5} under Case C.

\begin{lemma}
    The two $2\times 5$ linear systems
    \[
    \begin{pmatrix}
            1 & 1 & -1 & -1 & 0\\
            3 & -3 & 1 & 0 & -1
        \end{pmatrix}\qquad \begin{pmatrix}
            1 & 1 & -1 & -1 & 0\\
            3 & -3 & 2 & 0 & -2
        \end{pmatrix}\]
    are uncommon over $\FF_p$ for sufficiently large $p$.
\end{lemma}
\begin{proof}
    The first linear system has
\begin{alignat*}{3}
    L_5 \mathbf x &= x_1 + x_2 - x_3 - x_4\qquad\qquad  &&A_5=\{\!\!\{1,1,-1,-1\}\!\!\}
\\ L_4 \mathbf x &=3x_1-3x_2+x_3-x_5 &&A_4=\{\!\!\{3,-3,1,-1\}\!\!\}\\
L_3 \mathbf x &=4x_1 - 2x_2 - x_4 -x_5 &&A_3=\{\!\!\{4,-2,-1,-1\}\!\!\}  \\ 
L_2 \mathbf x &=6x_1 -2 x_3 - 3x_4 - 1 x_5 &&A_2=\{\!\!\{6,-2,-3,-1\}\!\!\}  \\
L_1 \mathbf x &=6x_2 - 4x_3 -3x_4 + 1x_5 &&A_1=\{\!\!\{6,-4,-3,1\}\!\!\}
\end{alignat*}
which gives $P_L=\{2,3\}$ and $G_L=\{2^{d_1}3^{d_2}:d_1,d_2\in\{0,1,\dots\}\}$.  Define $h:\pm G_L\to\CC$ by
\[
\begin{cases}
    h(r)=\phi(v_2(r),v_3(r)) & r\in G_L\\
    h(r)=\overline{\phi(v_2(r),v_3(r))} & r\in -G_L
\end{cases}
\]
where $\phi:\ZZ_2\times \ZZ_4\to\CC$ is given by
\[
\begin{matrix*}[l]
    & \phi(0,0)=1   & \phi(0,1)=0 & \phi(0,2)=0.25i  & \phi(0,3)=0.5i\\
    & \phi(1,0)=-1   & \phi(1,1)=0 & \phi(1,2)=-0.25i  & \phi(1,3)=-0.5i.\\
\end{matrix*}
\]
One can then check that
$$\text{Re}\sum_{\substack{0\leq d_1\leq1\\0\leq d_2\leq 3}}\sum_{i=1}^5h_{A_i}(2^{d_1}3^{d_2})=-0.265625<0.$$



    The second linear system has
\begin{alignat*}{3}
    L_5 \mathbf x &= x_1 + x_2 - x_3 - x_4\qquad\qquad  &&A_5=\{\!\!\{1,1,-1,-1\}\!\!\}
\\ L_4 \mathbf x &=3x_1-3x_2+2x_3-2x_5 &&A_4=\{\!\!\{3,-3,2,-2\}\!\!\}\\
L_3 \mathbf x &=5x_1 - x_2 - 2x_4 -2x_5 &&A_3=\{\!\!\{5,-1,-2,-2\}\!\!\}  \\ 
L_2 \mathbf x &=6x_1 - x_3 - 3x_4 - 2 x_5 &&A_2=\{\!\!\{6,-1,-3,-2\}\!\!\}  \\
L_1 \mathbf x &=6x_2 - 5x_3 -3x_4 + 2x_5 &&A_1=\{\!\!\{6,-5,-3,2\}\!\!\}
\end{alignat*}
which gives $P_L=\{2,3,5\}$ and $G_L=\{2^{d_1}3^{d_2}5^{d_3}:d_1,d_2,d_3\in\{0,1,\dots\}\}$.  Define $h:\pm G_L\to\CC$ by
\[
\begin{cases}
    h(r)=\phi(v_5(r),v_3(r)) & r\in G_L\\
    h(r)=\overline{\phi(v_5(r),v_3(r))} & r\in -G_L
\end{cases}
\]
where $\phi:\ZZ_2\times \ZZ_4\to\CC$ is given by
\[
\begin{matrix*}[l]
    & \phi(0,0)=1   & \phi(0,1)=0 & \phi(0,2)=0.25i  & \phi(0,3)=0.5i\\
    & \phi(1,0)=-1   & \phi(1,1)=0 & \phi(1,2)=-0.25i  & \phi(1,3)=-0.5i.\\
\end{matrix*}
\]
One can then check that
$$\text{Re}\sum_{\substack{0\leq d_1\leq1\\0\leq d_2\leq 3}}\sum_{i=1}^5h_{A_i}(2^{d_1}3^{d_2})=-0.265625<0.$$
The result follows by \cref{lem:periodic}.
\end{proof}

\subsection{Proof of \cref{thm:2*5} under Case D}
\label{subsec:D}
In this section, we prove \cref{thm:2*5} under Case D listed in \cref{cor:2*5-classification}. We show that $L$ is common if and only if it is isomorphic to
\[
        \begin{pmatrix}
            1 & -1 & 1 & -1 & 0\\
            1 & 2 & -1 & 0 & -2
        \end{pmatrix}\text{ or } 
        \begin{pmatrix}
            1 & -1 & 1 & -1 & 0\\
        0 & 1 & 2 & -1 & -2
        \end{pmatrix}.
\]

Without loss of generality, suppose 
\begin{align*}
        &L_5=x_1-x_2+\alpha x_3-\alpha x_4\\
        &L_4=x_1+\beta x_2- x_3-\beta x_5\\
        &L_3=(\alpha+1)x_1+(\alpha\beta -1)x_2-\alpha x_4-\alpha\beta x_5\\
        &L_2=(\beta+1)x_1+(\alpha\beta-1)x_3- \alpha\beta x_4-\beta x_5\\
        &L_1=(\beta+1) x_2+(-1-\alpha)x_3+ \alpha x_4-\beta x_5
    \end{align*}
with $\alpha,\beta\in\QQ$ and $\alpha,\beta>0$, and at least one of $L_1,L_2,L_3$ is common. Some elementary calculations (which we defer to \cref{app:D}) show that we can reduce to the following system: 
\begin{align}\label{eq:reduced-D}
        &L_5=x_1-x_2+\alpha x_3-\alpha x_4 \notag \\
        &L_4=x_1+\alpha x_2- x_3-\alpha x_5 \notag \\
        &L_3=(\alpha+1)x_1+(\alpha^2 -1)x_2-\alpha x_4-\alpha^2 x_5 \notag \\
        &L_2=(\alpha+1)x_1+(\alpha^2-1)x_3- \alpha^2 x_4-\alpha x_5 \notag \\
        &L_1=(\alpha+1) x_2+(-1-\alpha)x_3+ \alpha x_4-\alpha x_5
\end{align}

In this case, first observe that $L_2$ and $L_3$ are uncommon, as none of $\alpha=\alpha+1$ or $\alpha=\alpha^2-1$ has rational solutions. Next, observe that $\alpha\notin\{-1,1\}$, as either $\alpha=1$ or $\alpha=-1$ would give $s(L)<4$. This implies that $L_4$ and $L_5$ cannot be additive quadruples.

Suppose $L_1$ is an additive quadruple, i.e, $\alpha=-1/2$. In this case, $L$ is isomorphic to 
\[
        \begin{pmatrix}
            1 & -1 & 1 & -1 & 0\\
            1 & 2 & -1 & 0 & -2
        \end{pmatrix}
\]
which is common (see \cref{sec:common} for proof of its commonness).

Suppose none of $L_1,L_4,L_5$ is an additive quadruple. Then the main idea is to utilize the constructions in \cref{subsec:A}, where we set the value of the Fourier template to be large on some well-chosen values. Recall that \cref{lem:coincidence-isolated} and \cref{lem:coincidence-k13} give a criterion for when such a construction works. Applying this criterion, it follows that $L$ is common only if we can partition $\{\!\!\{\alpha+1,\alpha^2-1,-\alpha,-\alpha^2 \}\!\!\}$ into two pairs, such that the ratio between the two numbers in each pair lies in $\{\alpha,-\alpha, \frac{\alpha}{1+\alpha},-\frac{\alpha}{1+\alpha}\}$. By a finite case check, the only $\alpha\in \QQ\setminus\{-1,0,1\}$ with this property is $\alpha=1/2$ \footnote{Verifying this by hand is  straightforward. Alternatively, readers can find a Mathematica code checking this at \texttt{`case-d.nb'} in the supplemental file.}.  When $\alpha=1/2$, $L$ is isomorphic to
\[
        \begin{pmatrix}
            1 & 0 & -1 & 2 & -2\\
            0 & 1 & 2 & -1 & -2
        \end{pmatrix}
\]
which is common (see \cref{sec:common} for proof of its commonness).


\subsection{Proof of \cref{thm:2*5} for $s(L)<4$}
\label{subsec:s<4}
In this section, we prove \cref{thm:2*5} for $s(L)<4$.

Suppose $L$ is an irredundant $2\times 5$ linear system with $s(L)<4$. We show that $L$ is common over all $\FF_p$ if and only if it is isomorphic to a linear system of the form
\begin{align}\label{eqn:common-length-3}
            &\begin{pmatrix}
            a & b & 0 & 0 & c\\
            0 & 0 & a & b & c
        \end{pmatrix}\\
        \label{eqn:common-length-3-2}
        &\text{ or }\begin{pmatrix}
        1 & -1 & \lambda & -\lambda & 0\\
        a & b & 0 & 0 & -a-b
    \end{pmatrix},\quad \text{$\{|a/b|,|a/(a+b)|,|b/(a+b)|\}\cap \{|\lambda|,|\lambda|^{-1}\}\neq\emptyset$.}
\end{align}

First, we recall the following lemma from \cite{KLM21a}.

\begin{lemma}[{\cite[Theorem 1.3]{KLM21a}}]\label{lem:all-uncommon}
    Let $L$ be an irredundant $2 \times 5$ linear system with $s(L)$ even. If every $1\times s(L)$ subsystem of $L$ (i.e., every length-$s(L)$ equation that is a subsystem of $L$) is uncommon over $\FF_p$, then $L$ is uncommon over $\FF_p$.  
\end{lemma}

Applying Lemma~\ref{lem:all-uncommon}, we first handle a few simple cases.

If $s(L) = 2$, then since $L$ is irredundant (i.e., not having subsystems of the form $x_i=x_j$), we know that every $1\times 2$ subsystem of $L$ is uncommon . By \cref{lem:all-uncommon}, $L$ is uncommon .

If $s(L) = 3$ and $L$ has a $2\times 4$ subsystem, then by Theorem 1.1 of \cite{KLM21a}, $L$ is uncommon.

Suppose from now that $s(L)=3$ and  $L$ does not have a $2\times 4$ subsystem. If $L$ has two linearly independent $1\times 3$ subsystems, then it is isomorphic to
\[
\begin{pmatrix}
            a_1 & b_1 & 0 & 0 & c\\
            0 & 0 & a_2 & b_2 & c
        \end{pmatrix}.
\]
In this case, we have $c(L)=4$ and $\mathcal C(L)=\{\{1,2,3,4\}\}$, i.e., the only element in $\{L_B:B\in\mathcal C(L)\}$ is the linear equation
\[
L_5\mathbf x=a_1x_1+b_1x_2-a_2x_3-b_2x_4.
\]
If $L_5$
is uncommon, then we can find a Fourier template $g:\ZZ\to\CC$ with $\sigma_{L_5}(g)<0$, so $L$ is uncommon over sufficiently large $p$ by \cref{thm:fourier-template}. We will show in \cref{sec:common} that, when $L_5$ is common, $L$ is common if and only if it is of the form listed in \cref{eqn:common-length-3,eqn:common-length-3-2}.

Now suppose that $L$ has only one $1\times 3$ subsystem. Again, by \cref{thm:fourier-template}, if all $1\times 4$ subsystems $L$ are uncommon, then $L$ is uncommon as well. Hence it suffices to consider the case where $L$ has a common $1\times 4$ subsystem, i.e.,
$L$ is isomorphic to either of the following:
\begin{align}\label{eq:reduction-E}
    \begin{pmatrix}
        1 & -1 & \lambda & -\lambda & 0\\
        a & b & 0 & 0 & c
    \end{pmatrix}\quad
    \begin{pmatrix}
        1 & -1 & \lambda & -\lambda & 0\\
        a & 0 & b & 0 & c
    \end{pmatrix}
\end{align}
with $\lambda\in\QQ$, $\lambda>0$ and $a,b,c\in\ZZ\setminus\{0\}$.

In the first case, the elements in $\{L_B:B\in\mathcal C(L)\}$ are
\begin{alignat*}{2}
    &L_5\mathbf x=x_1-x_2+\lambda x_3-\lambda x_4 &&A_5=\{\!\!\{1,-1,\lambda,-\lambda\}\!\!\}\\
    &L_1\mathbf x =(a+b)x_2-a\lambda x_3+a\lambda x_4+cx_5\qquad  &&A_1=\{\!\!\{a+b,-a\lambda,a\lambda,c\}\!\!\}\\
    &L_2\mathbf x =(a+b)x_1+b\lambda x_3-b\lambda x_4+cx_5 &&A_2=\{\!\!\{a+b,b\lambda,-b\lambda,c\}\!\!\}.
\end{alignat*} 
We first show that, if $a+b+c \neq 0$, then $L$ is uncommon. We prove this by finding Fourier template $g:\ZZ\to\CC$ such that $\sigma_{L_1}(g)+\sigma_{L_2}(g)+\sigma_{L_5}(g)<0$ and applying \cref{thm:fourier-template}. The construction of suitable Fourier templates uses nearly identical ideas as that for \cref{obs:b-1-1} and \cref{prop:b-1-2}. Indeed, we have the following cases: 
\begin{itemize}
    \item Each of $A_1,A_2$ has three positive and one negative elements. Consider $h:\pm G_L\to\CC$
    with
    \begin{align*}
        \begin{cases}
            h(r)=i & r>0\\
            h(r)=-i & r<0.
        \end{cases}
    \end{align*}
    We know that each of $h_{A_1},h_{A_2},h_{A_5}$ is constant on $\pm G_L$, with $h_{A_5}=1$ and $h_{A_1}=h_{A_2}=-1$. Hence \cref{lem:periodic} applies.
    
    \item Each of $A_1,A_2$ contains two positive and two negative elements, with $a+b+c\neq 0$. Then since $A_1,A_2$ are not cancelling, we can find some $h:\pm G_L\to\CC$ as in \cref{prop:b-1-2}, so that each of $\text{Re}(h_{A_1}),\text{Re}(h_{A_2}),h_{A_5}$ is constant on $\pm G_L$, and $\text{Re}(h_{A_1})+\text{Re}(h_{A_2})+h_{A_5}<0$. Hence \cref{lem:periodic} applies.
\end{itemize}
Finally, for $a+b+c =0$, we show in \cref{sec:common} that $L$ is common if and only if is it of the form listed in \cref{eqn:common-length-3,eqn:common-length-3-2}.

Lastly, we have the second case, where the elements in $\{L_B:B\in\mathcal C(L)\}$ are given by
\begin{alignat}{2}\label{eq:E-case 2}
    &L_5\mathbf x =x_1-x_2+\lambda x_3-\lambda x_4 &&A_5=\{\!\!\{1,-1,\lambda,-\lambda\}\!\!\} \notag \\
    &L_1\mathbf x =ax_2+(b-a\lambda) x_3+a\lambda x_4+cx_5\qquad  &&A_1=\{\!\!\{b-a\lambda,a,a\lambda,c\}\!\!\} \notag \\
    &L_3\mathbf x =(a\lambda-b)x_1+b x_2+b\lambda x_4+c\lambda x_5 &&A_3=\{\!\!\{a\lambda-b,b,b\lambda,c\lambda\}\!\!\}.
\end{alignat}

Notice that we overlap with the previous case if $\lambda=-1$. We handle the typical case  $\lambda\neq 1$ and defer the case  $\lambda =  1$ to the appendix; the construction for the case when $\lambda =  1$ uses the same idea as that in \cref{ex:cancelling-1}, \cref{ex:cancelling-2} and the proof of \cref{prop:cancelling}.

When $\lambda \not \in \{ 1,-1 \}$, the key idea is to reduce to the setup of \cref{lem:coincidence-isolated} and \cref{lem:coincidence-k13}. Namely, \cref{lem:coincidence-isolated} and \cref{lem:coincidence-k13} together show that if the $A_i$ corresponding to $L$ cannot be partitioned into two $L$-coincidental pairs, then $L$ is uncommon. For the system of $L$ above, we will show that this is the case because of parity reasons. 

Suppose $\lambda\notin\{1,-1\}$. In this case, note that at least one of $|\frac{b-a\lambda}{c}|$ and $|\frac{a\lambda-b}{c\lambda}|$ is not an odd power of $\lambda$. Without loss of generality, suppose $|\frac{a\lambda-b}{c\lambda}|$ is not an odd power of $\lambda$. Note that $L_5$ is $\lambda$-common. Moreover, if $L_1$ is common, then it must be $\lambda$-common as well. This implies that $A_3=\{\!\!\{a\lambda-b,b,b\lambda,c\lambda\}\!\!\}$ cannot be partitioned into $L$-coincidental pairs, so $L$ is uncommon.

\subsection{Proofs of commonality}\label{sec:common}

In this section, we show the positive side of \cref{thm:2*5}. That is, $L$ is common if it is one of the forms listed in \cref{thm:2*5}(1).

\subsubsection{Common $2\times 5$ linear systems with $s(L)=4$}

In this section, we show that the three $2\times 5$ linear systems 
\[
        \begin{pmatrix}
            1 & -1 & 1 & -1 & 0\\
            1 & 2 & -1 & 0 & -2
        \end{pmatrix}\quad
        \begin{pmatrix}
            1 & 0 & -1 & 2 & -2\\
            0 & 1 & 2 & -1 & -2
        \end{pmatrix} \quad
        \begin{pmatrix}
            1 & 1 & -1 & -1 & 0\\
            2 & -2 & 1 & 0 & -1
        \end{pmatrix}
\]
are common over all $\FF_p$. Proofs that the first two linear systems are common can also be found at \cite[Examples 4.5 and 4.6]{KLM21b}.

Let $L$ be an irredundant $2\times 5$ linear system with $s(L)=4$.
For $f:\FF_p^n\to[0,1]$ with $\EE f=\alpha$,  we have
\begin{align*}
    t_L(f)+t_L(1-f)&=\EE_{L \mathbf x =0}[f(x_1)\cdots f(x_5)+(1-f(x_1))\cdots (1-f(x_5))]\\
    &=1-5\alpha+\binom{5}{2}\alpha^2-\binom{5}{3}\alpha^3+\sum_{i=1}^5 \EE_{L \mathbf x =0}[\prod_{j\in[5]\setminus\{i\}}f(x_j)]\\
    &=(1-\alpha)^5+\alpha^5-5\alpha^4+\sum_{i=1}^5 t_{L_i}(f)
    \overset{\cref{eqn:inversion2}}{=}(1-\alpha)^5+\alpha^5+\sum_{i=1}^5\sum_{r\in\widehat{\FF_p^n}\setminus\{0\}}\prod_{a\in A_i}\widehat f(ar).
\end{align*}
Thus, if
\begin{align}\label{eqn:common}
    \sum_{i=1}^5\sum_{r\in\widehat{\FF_p^n}\setminus\{0\}}\prod_{a\in A_i}\widehat f(ar)\geq 0\qquad \text{ for all $f:\FF_p^n\to[0,1]$},
\end{align}
we would have 
$t_L(f)+t_L(1-f)\geq 2^{-4}$
for all $f$, which implies that $L$ is common over all $\FF_p$.

We now verify that \cref{eqn:common} holds for the three linear systems mentioned above. The first one has
\begin{alignat*}{3}
    L_5 \mathbf x &= x_1-x_2+x_3-x_4\qquad\qquad  &&A_5=\{\!\!\{1,1,-1,-1\}\!\!\}
\\ L_4 \mathbf x &=x_1+2x_2-x_3-2x_5 &&A_4=\{\!\!\{1,-1,2,-2\}\!\!\}\\
L_3 \mathbf x &=2x_1 +x_2 - x_4 -2x_5 &&A_3=\{\!\!\{1,-1,2,-2\}\!\!\}  \\ 
L_2 \mathbf x &=3x_1+x_3-2x_4-2x_5 &&A_2=\{\!\!\{1,3,-2,-2\}\!\!\}  \\
L_1 \mathbf x &=3x_2-2x_3+x_4-2x_5 &&A_1=\{\!\!\{1,3,-2,-2\}\!\!\}.
\end{alignat*}
For all $f:\FF_p^n\to[0,1]$, we have
\begin{align*}
    \sum_{i=1}^5\sum_{r\in\widehat{\FF_p^n}\setminus\{0\}}\prod_{a\in A_i}\widehat f(ar)&=\sum_{r\in\widehat{\FF_p^n}\setminus\{0\}}2|\widehat f(r)|^2|\widehat f(2r)|^2+|\widehat f(r)|^2|\widehat f(3r)|^2+2\widehat f(r)\widehat f(2r)\widehat f(3r)\widehat f(-6r).
\end{align*}
Using the Cauchy--Schwarz inequality, we get that
\begin{align*}
    \left| \sum_r\widehat f(r)\widehat f(2r)\widehat f(3r)\widehat f(-6r) \right|
    &\leq \sqrt{\left(\sum_r|\widehat f(r)|^2|\widehat f(2r)|^2\right)\left(\sum_r|\widehat f(3r)|^2|\widehat f(-6r)|^2\right)}
    =\sum_r|\widehat f(r)|^2|\widehat f(2r)|^2,
\end{align*}
which implies that $\sum_{i=1}^5\sum_{r\in\widehat{\FF_p^n}\setminus\{0\}}\prod_{a\in A_i}\widehat f(ar)\geq 0$ always.

The second linear system has
\begin{alignat*}{3}
    L_5 \mathbf x &= x_1-x_2-3x_3+3x_4\qquad\qquad  &&A_5=\{\!\!\{1,-1,3,-3\}\!\!\}
\\ L_4 \mathbf x &=x_1+2x_2+3x_3-6x_5 &&A_4=\{\!\!\{1,2,3,-6\}\!\!\}\\
L_3 \mathbf x &=2x_1 +x_2 +3 x_4 -6x_5 &&A_3=\{\!\!\{1,2,3,-6\}\!\!\}  \\ 
L_2 \mathbf x &=x_1-x_3+2x_4-2x_5 &&A_2=\{\!\!\{1,-1,2,-2\}\!\!\}  \\
L_1 \mathbf x &=x_2+2x_3-x_4-2x_5 &&A_1=\{\!\!\{1,-1,2,-2\}\!\!\}.
\end{alignat*}

For all $f:\FF_p^n\to[0,1]$, we have
\begin{align*}
    \sum_{i=1}^5\sum_{r\in\widehat{\FF_p^n}\setminus\{0\}}\prod_{a\in A_i}\widehat f(ar)&=\sum_{r\in\widehat{\FF_p^n}\setminus\{0\}}|\widehat f(r)|^4+2|\widehat f(r)|^2|\widehat f(2r)|^2+2\widehat f(r)\widehat f(3r)\widehat f(-2r)^2.
\end{align*}
Using the Cauchy--Schwarz inequality twice, we get that
\begin{align*}
    \left|\sum_r \widehat{f}(r) \widehat{f}(3r) \widehat{f}(-2r)^2 \right| &\leq \left( \sum_r |\widehat{f}(r)|^2| \widehat{f}(-2r)|^2\right)^{1/2} \left( \sum_r |\widehat{f}(-2r)|^2|\widehat{f}(3r)|^2\right)^{1/2} \\
    &\leq \left( \sum_r |\widehat{f}(r)|^2 |\widehat{f}(-2r)|^2 \right)^{1/2} \left( \frac{1}{2} \left(\sum_r |\widehat{f}(-2r)|^4 + \sum_r |\widehat{f}(3r)|^4   \right)\right)^{1/2} \\
    &\leq \frac{1}{2} \sum_{r} |\widehat{f}(r)|^2 |\widehat{f}(-2r)|^2 + \frac{1}{4} \left( \sum_r|\widehat{f}(-2r)|^4 + \sum_r|\widehat{f}(3r)|^4 \right)\\
    &=\frac{1}{2} \sum_{r} |\widehat{f}(r)|^2 |\widehat{f}(2r)|^2 + \frac{1}{2}  \sum_r|\widehat{f}(r)|^4,
\end{align*}
which implies that $\sum_{i=1}^5\sum_{r\in\widehat{\FF_p^n}\setminus\{0\}}\prod_{a\in A_i}\widehat f(ar)\geq 0$ always.

The third linear system has
\begin{alignat*}{3}
    L_5 \mathbf x &= x_1-x_2+x_3-x_4\qquad\qquad  &&A_5=\{\!\!\{1,1,-1,-1\}\!\!\}
\\ L_4 \mathbf x &=2x_1-2x_2+x_3-x_5 &&A_4=\{\!\!\{2,-2,1,-1\}\!\!\}\\
L_3 \mathbf x &=3x_1 -x_2 - x_4 -x_5 &&A_3=\{\!\!\{3,-1,-1,-1\}\!\!\}  \\ 
L_2 \mathbf x &=4x_1-x_3-2x_4-x_5 &&A_2=\{\!\!\{4,-1,-2,-1\}\!\!\}  \\
L_1 \mathbf x &=4x_2-3x_3-2x_4+x_5 &&A_1=\{\!\!\{4,-3,-2,1\}\!\!\}.
\end{alignat*}
For all $f:\FF_p^n\to[0,1]$, we have
\begin{align*}
    &\sum_{i=1}^5\sum_{r\in\widehat{\FF_p^n}\setminus\{0\}}\prod_{a\in A_i}\widehat f(ar)\\
    &=\sum_{r\in\widehat{\FF_p^n}\setminus\{0\}}|\widehat f(r)|^4+|\widehat f(r)|^2|\widehat f(2r)|^2+\widehat f(3r)\widehat f(-r)^3  +\widehat f(4r)\widehat f(-2r)\widehat f(-r)^2+\widehat f(4r)\widehat f(-3r)\widehat f(-2r)\widehat f(r)\\
    &=\sum_{r\in\widehat{\FF_p^n}\setminus\{0\}}\frac{3}{4}|\widehat f(r)|^4+\frac{1}{4}|\widehat f(3r)|^4+|\widehat f(2r)|^2|\widehat f(4r)|^2\\
     &\qquad\qquad \quad  +\text{Re}(\widehat f(3r)\widehat f(-r)^3++\widehat f(4r)\widehat f(-2r)\widehat f(-r)^2+\widehat f(4r)\widehat f(-3r)\widehat f(-2r)\widehat f(r)).
\end{align*}
Thus, to show that $L$ is common, it suffices to show that
\begin{align*}
    &\frac{3}{4}|\widehat f(r)|^4+\frac{1}{4}|\widehat f(3r)|^4+|\widehat f(2r)|^2|\widehat f(4r)|^2\\
     &\quad +\text{Re}(\widehat f(3r)\widehat f(-r)^3+\widehat f(4r)\widehat f(-2r)\widehat f(-r)^2+\widehat f(4r)\widehat f(-3r)\widehat f(-2r)\widehat f(r))\geq 0\qquad \text{for all }r\in\widehat{\FF_p^n}.
\end{align*}
Consider any $r\in\widehat{\FF_p^n}$. Let
\begin{align*}
    &a_1=|\widehat f(r)|, \quad a_2=|\widehat f(2r)|, \quad a_3=|\widehat f(3r)|, \quad a_4=|\widehat f(4r)|.
\end{align*}
Since 
\[
\widehat f(3r)\widehat f(-r)^3\cdot \widehat f(4r)\widehat f(-3r)\widehat f(-2r)\widehat f(r)=\widehat f(4r)\widehat f(-2r)\widehat f(-r)^2\cdot |\widehat f(r)|^2|\widehat f(3r)|^2,
\]
letting $\alpha,\beta$ be the arguments of $\widehat f(3r)\widehat f(-r)^3$, $\widehat f(4r)\widehat f(-3r)\widehat f(-2r)\widehat f(r)$ respectively, our goal becomes showing that
\begin{align}\label{eq:cosine}
    I:=\frac{3}{4}a_1^4+\frac{1}{4}a_3^4+a_2^2a_4^2+a_1^3a_3\cos(\alpha)+a_1^2a_2a_4\cos(\alpha+\beta)+a_1a_2a_3a_4\cos(\beta)\geq 0.
\end{align}
Setting $x=a_1^2$, $y=a_1a_3$ and $z=a_2a_4$, we have
\begin{align*}
    I&\geq \frac{1}{2}a_1^4+\frac{1}{2}a_1^2a_3^2+a_2^2a_4^2+a_1^3a_3\cos(\alpha)+a_1^2a_2a_4\cos(\alpha+\beta)+a_1a_2a_3a_4\cos(\beta)\\
    &=\frac{1}{2}(x^2+y^2+2z^2+2xy\cos(\alpha)+2xz\cos(\alpha+\beta)+2yz\cos(\beta)).
\end{align*}
Since
\begin{align*}
    0&\leq (y+z\cos(\beta)+x\cos(\alpha))^2+(z\sin(\beta)-x\sin(\alpha))^2\\
    &=y^2+z^2+x^2+2yz\cos(\beta)+2xy\cos(\alpha)+2xz\cos(\alpha+\beta),
\end{align*}
we conclude that $I\geq 0$. This gives \cref{eq:cosine}.

\subsubsection{Common $2\times 5$ linear systems with $s(L)=3$}

We first show that linear systems of the form
\[
L=\begin{pmatrix}
    a & b & 0 & 0 & c\\
    0 & 0 & a & b & c
\end{pmatrix}\qquad a,b,c\in\ZZ\setminus\{0\}
\]
are common. 
Define
\begin{align*}
    L_1 \mathbf x &=ax_1+bx_2+cx_5 \\
    L_2 \mathbf x &=ax_1+bx_2-ax_3-bx_4.
\end{align*}

Consider any $f:\FF_p^n\to[0,1]$ with $\EE f=\alpha\in[0,1]$. For any $\{i,j,k\}\in\binom{[5]}{3}$, we have
\[
\EE_{L \mathbf x =0}f(x_i)f(x_j)f(x_k)=\begin{cases}
    t_{L_1}(f) & \{i,j,k\}=\{1,2,5\}\text{ or }\{3,4,5\}\\
    \alpha^3 & \text{otherwise.}
\end{cases}
\]
For any $\{i,j,k,\ell\}\in\binom{[5]}{4}$, we have
\[
\EE_{L \mathbf x =0}f(x_i)f(x_j)f(x_k)f(x_\ell)=\begin{cases}
    t_{L_2}(f) & \{i,j,k,\ell\}=\{1,2,3,4\}\\
    \alpha t_{L_1}(f) & \text{otherwise.}
\end{cases}
\]
Thus, we have
\begin{align*}
    t_L(f)+t_L(1-f)&=\EE_{L \mathbf x =0}f(x_1)\cdots f(x_5)+(1-f(x_1))\cdots (1-f(x_5))\\
    &=1-5\alpha+10\alpha^2-(8\alpha^3+2t_{L_1}(f))+(t_{L_2}(f)+4\alpha t_{L_1}(f)).
\end{align*}
Since 
\begin{align*}
    t_{L_1}(f)&=\sum_{r\in\widehat{\FF_p^n}}\widehat f(ar)\widehat f(br)\widehat f(cr)=\alpha^3+\sum_{r\in\widehat{\FF_p^n}\setminus\{0\}}\widehat f(ar)\widehat f(br)\widehat f(cr),\\
    t_{L_2}(f)&=\sum_{r\in\widehat{\FF_p^n}}|\widehat f(ar)|^2|\widehat f(br)|^2=\alpha^4+\sum_{r\in\widehat{\FF_p^n}\setminus\{0\}}|\widehat f(ar)|^2|\widehat f(br)|^2,
\end{align*}
we get that
\begin{align*}
    t_L(f)+t_L(1-f)&=\alpha^5+(1-\alpha)^5+(4\alpha-2)\sum_{r\in\widehat{\FF_p^n}\setminus\{0\}}\widehat f(ar)\widehat f(br)\widehat f(cr)+ \sum_{r\in\widehat{\FF_p^n}\setminus\{0\}}|\widehat f(ar)|^2|\widehat f(br)|^2.
\end{align*}
Letting 
\[
X=\sum_{r\in\widehat{\FF_p^n}\setminus\{0\}}\widehat f(ar)\widehat f(br)\widehat f(cr),\qquad Y=\sum_{r\in\widehat{\FF_p^n}\setminus\{0\}}|\widehat f(ar)|^2|\widehat f(br)|^2,
\]
we know from Cauchy--Schwarz and Parseval that
\begin{align*}
    X^2&\leq Y\sum_{r\in\widehat{\FF_p^n}\setminus\{0\}}|\widehat f(cr)|^2=Y\sum_{r\in\widehat{\FF_p^n}\setminus\{0\}}|\widehat f(r)|^2
    =Y(\EE[|f|^2]-\alpha^2)\leq Y(||f||_\infty\EE[f]-\alpha^2)
    \leq Y(\alpha-\alpha^2).
\end{align*}
Since $0\leq Y=t_{L_2}(f)-\alpha^4\leq 1$, we get that
\[
t_L(f)+t_L(1-f)\geq \min_{0\leq\alpha,Y\leq 1}\alpha^5+(1-\alpha^5)-|4\alpha-2|\sqrt{Y(\alpha-\alpha^2)}+Y\geq \frac{1}{16}.
\]
Hence $L$ is common.

Next, we show that linear systems of the form
$$
L=\begin{pmatrix}
    1 & -1 & \lambda & -\lambda & 0\\
    a & b & 0 & 0 & -a-b
\end{pmatrix}
,\quad \{ |a/b|,|(a+b)/b|,|(a+b)/a|\}\cap\{|\lambda|,|\lambda|^{-1}\}\neq\emptyset$$
are common. Define
\begin{align*}
    L_1 \mathbf x &=ax_1+bx_2-(a+b)x_5\\
    L_2 \mathbf x &=x_1-x_2+\lambda x_3-\lambda x_4\\
    L_3 \mathbf x &=(a+b)x_1+b\lambda x_3-b\lambda x_4-(a+b)x_4\\
    L_4 \mathbf x &=(a+b)x_2-a\lambda x_2+a\lambda x_3-(a+b)x_4.
\end{align*}
For all $f:\FF_p^n\to[0,1]$ with $\EE f=\alpha\in[0,1]$, again we can rewrite
\begin{align*}
    t_L(f)+t_L(1-f)&=\alpha^5+(1-\alpha)^5
     +(2\alpha-1)\sum_{r\in\widehat{\FF_p^n}\setminus\{0\}}\widehat f(ar)\widehat f(br)\overline{\widehat f((a+b)r)}
     + \sum_{r\in\widehat{\FF_p^n}\setminus\{0\}}|\widehat f(r)|^2|\widehat f(\lambda r)|^2\\
     &\quad+\sum_{r\in\widehat{\FF_p^n}\setminus\{0\}}|\widehat f(b\lambda r)|^2|\widehat f((a+b)r)|^2+\sum_{r\in\widehat{\FF_p^n}\setminus\{0\}}|\widehat f(a\lambda r)|^2|\widehat f((a+b)r)|^2.
\end{align*}
Let
\[
X=\sum_{r\in\widehat{\FF_p^n}\setminus\{0\}}\widehat f(ar)\widehat f(br)\overline{\widehat f((a+b)r)},\qquad Y=\sum_{r\in\widehat{\FF_p^n}\setminus\{0\}}|\widehat f(r)|^2|\widehat f(\lambda r)|^2.
\]
Since one of $|a/b|,|a/b|^{-1},|(a+b)/b|,|(a+b)/b|^{-1},|(a+b)/a|,|(a+b)/a|^{-1}$ equals $|\lambda|$,
we know from Cauchy--Schwarz and Parseval that
\begin{align*}
    X^2&\leq Y\sum_{r\in\widehat{\FF_p^n}\setminus\{0\}}|\widehat f(r)|^2
    =Y(\EE[|f|^2]-\alpha^2)\leq Y(||f||_\infty\EE[f]-\alpha^2)
    \leq Y(\alpha-\alpha^2).
\end{align*}
Since $0\leq Y\leq 1$, we get that
\[
t_L(f)+t_L(1-f)\geq \min_{0\leq\alpha,Y\leq 1}\alpha^5+(1-\alpha^5)-|2\alpha-1|\sqrt{Y(\alpha-\alpha^2)}+Y\geq \frac{1}{16}.
\]
Hence $L$ is common.

Finally, we show that linear systems
\begin{align*}
    \begin{pmatrix}
        a & -a & 0 & 0 & c\\
        0 & 0 & b & -b & c
    \end{pmatrix}\qquad
    \begin{pmatrix}
        1 & -1 & \lambda & -\lambda & 0\\
        a & b & 0 & 0 & -a-b
    \end{pmatrix}
\end{align*}
not isomorphic to \cref{eqn:common-length-3,eqn:common-length-3-2} are uncommon over sufficiently large $\FF_p$. For the first linear system, define
\begin{align*}
    L_1 \mathbf x &=ax_1-ax_2+cx_5\\
    L_2 \mathbf x &=bx_3-bx_4+cx_5\\
    L_3 \mathbf x &=ax_1-ax_2+bx_3-bx_4.
\end{align*}
Since the linear system is not isomorphic to \cref{eqn:common-length-3}, we have $|a|\neq|b|$. Since it is not isomorphic to \cref{eqn:common-length-3-2}, we also have $|a/b|\notin\{|a/c|,|a/c|^{-1}\}$. 
For all $f:\FF_p^n\to[0,1]$ with $\EE f=\alpha\in[0,1]$, we have
\begin{align*}
    t_{L}(f)+t_{L}(1-f)&=\alpha^5+(1-\alpha)^5+(2\alpha-1)\sum_{r\in\widehat{\FF_p}\setminus\{0\}}|\widehat f(ar)|^2\widehat f(cr)\\
    &\quad +(2\alpha-1)\sum_{r\in\widehat{\FF_p}\setminus\{0\}}|\widehat f(br)|^2\widehat f(cr)+\sum_{r\in\widehat{\FF_p}\setminus\{0\}}|\widehat f(ar)|^2|\widehat f(br)|^2.
\end{align*}
Suppose $p$ is sufficiently large. Define $f:\FF_p\to[0,1]$ by
\begin{align*}
    \widehat f(0)=0.4995:=\alpha,\qquad \widehat f(a)= \widehat f(-a)= \widehat f(c)= \widehat f(-c)=0.0834:=\beta.
\end{align*}
We then have 
\begin{align*}
    t_{L}(f)+t_{L}(1-f)&\leq  \alpha^5+(1-\alpha)^5+ (2\alpha-1)\sum_{r\in\widehat{\FF_p}\setminus\{0\}}|\widehat f(ar)|^2\widehat f(cr)\\
    &= \alpha^5+(1-\alpha)^5+ (2\alpha-1)2\beta^3=0.0624995<1/16.
\end{align*}

For the second linear system,  define
\begin{align*}
    L_1 \mathbf x &=ax_1+bx_2-(a+b)x_5\\
    L_2 \mathbf x &=x_1-x_2+\lambda x_3-\lambda x_4\\
    L_3 \mathbf x &=(a+b)x_1+b\lambda x_3-b\lambda x_4-(a+b)x_4\\
    L_4 \mathbf x &=(a+b)x_2-a\lambda x_2+a\lambda x_3-(a+b)x_4.
\end{align*}
Since the linear system is not isomorphic to 
\cref{eqn:common-length-3-2}, we know that $\lambda\notin\{1,-1\}$. Moreover, we have
\begin{align*}
    &\{ |a/b|,|(a+b)/b|,|(a+b)/a|\} \cap\{|\lambda|,|\lambda|^{-1},|(a+b)/(b\lambda)|,|b\lambda/(a+b)|,|(a+b)/(a\lambda)|,|a\lambda/(a+b)|\}=\emptyset.
\end{align*}
For all $f:\FF_p^n\to[0,1]$ with $\EE f=\alpha\in[0,1]$, we have
\begin{align*}
    t_{L}(f)+t_{L}(1-f)&=\alpha^5+(1-\alpha)^5+(2\alpha-1)\sum_{r\in\widehat{\FF_p}} \widehat f(ar)\widehat f(br)\overline{\widehat f((a+b)r)}\\
    &\quad +\sum_{r\in\widehat{\FF_p}\setminus\{0\}}|\widehat f(r)|^2|\widehat f(\lambda r)|^2+|\widehat f((a+b)r)|^2|\widehat f(b\lambda r)|^2+|\widehat f((a+b)r)|^2|\widehat f(a\lambda r)|^2.
\end{align*}
Suppose $p$ is sufficiently large.  Define $f:\FF_p\to[0,1]$ by
\begin{align*}
    \widehat f(0)&=0.4995:=\alpha,\qquad \widehat f(a)= \widehat f(-a)= \widehat f(b)= \widehat f(-b)=\widehat f(a+b)= \widehat f(-a-b)=0.0834:=\beta.
\end{align*}
We then have
\begin{align*}
    t_{L}(f)&=  \alpha^5+(1-\alpha)^5+ (2\alpha-1)\sum_{r\in\widehat{\FF_p}\setminus\{0\}}\widehat f(ar)\widehat f(br)\overline{\widehat f((a+b)r)}\\
    &= \alpha^5+(1-\alpha)^5+ (2\alpha-1)2\beta^3=0.0624995<1/16.
\end{align*}

\bibliographystyle{amsplain0}

\begin{thebibliography}{10}

\bibitem{Alt22a}
Daniel Altman, \emph{Local aspects of the {Sidorenko} property for linear equations},  (2022), arXiv:2210.17493.

\bibitem{Alt22b}
Daniel Altman, \emph{On a common-extendable, non-{S}idorenko linear system}, Comb. Theory \textbf{3} (2023), Paper No. 5, 12.

\bibitem{AL24}
Daniel Altman and Anita Liebenau, \emph{On the uncommonness of minimal rank-2 systems of linear equations},  (2024), arXiv:2404.18908.

\bibitem{BRS89}
P.~Baldi, Y.~Rinott, and C.~Stein, \emph{A normal approximation for the number of local maxima of a random function on a graph}, Probability, statistics, and mathematics, Academic Press, Boston, MA, 1989, pp.~59--81.

\bibitem{BC05}
A.~D. Barbour and Louis H.~Y. Chen, \emph{An introduction to {S}tein's method}, World Scientific Publishing Co. Pte. Ltd., Hackensack, NJ, 2005.

\bibitem{BR65}
G.~R. Blakley and Prabir Roy, \emph{A {H}\"{o}lder type inequality for symmetric matrices with nonnegative entries}, Proc. Amer. Math. Soc. \textbf{16} (1965), 1244--1245.

\bibitem{BR80}
Stefan~A. Burr and Vera Rosta, \emph{On the {R}amsey multiplicities of graphs---problems and recent results}, J. Graph Theory \textbf{4} (1980), 347--361.

\bibitem{CFS10}
David Conlon, Jacob Fox, and Benny Sudakov, \emph{An approximate version of {S}idorenko's conjecture}, Geom. Funct. Anal. \textbf{20} (2010), 1354--1366.

\bibitem{CKLL18}
David Conlon, Jeong~Han Kim, Choongbum Lee, and Joonkyung Lee, \emph{Some advances on {S}idorenko's conjecture}, J. Lond. Math. Soc. (2) \textbf{98} (2018), 593--608.

\bibitem{CL17}
David Conlon and Joonkyung Lee, \emph{Finite reflection groups and graph norms}, Adv. Math. \textbf{315} (2017), 130--165.

\bibitem{CL21}
David Conlon and Joonkyung Lee, \emph{Sidorenko's conjecture for blow-ups}, Discrete Anal. (2021), Paper No. 2, 13.

\bibitem{CLS23}
David Conlon, Joonkyung Lee, and Alexander Sidorenko, \emph{Extremal numbers and sidorenko's conjecture},  (2023), arXiv:2307.04588.

\bibitem{FPZ21}
Jacob Fox, Huy~Tuan Pham, and Yufei Zhao, \emph{Common and {S}idorenko linear equations}, Q. J. Math. \textbf{72} (2021), 1223--1234.

\bibitem{Goo59}
A.~W. Goodman, \emph{On sets of acquaintances and strangers at any party}, Amer. Math. Monthly \textbf{66} (1959), 778--783.

\bibitem{H10}
Hamed Hatami, \emph{Graph norms and {S}idorenko's conjecture}, Israel J. Math. \textbf{175} (2010), 125--150.

\bibitem{CST96}
Chris Jagger, Pavel \v{S}\v{t}ov\'{\i}\v{c}ek, and Andrew Thomason, \emph{Multiplicities of subgraphs}, Combinatorica \textbf{16} (1996), 123--141.

\bibitem{KLM21a}
Nina Kam{\v{c}}ev, Anita Liebenau, and Natasha Morrison, \emph{On uncommon systems of equations},  (2021), arXiv:2106.08986.

\bibitem{KLM21b}
Nina Kam\v{c}ev, Anita Liebenau, and Natasha Morrison, \emph{Towards a characterization of {S}idorenko systems}, Q. J. Math. \textbf{74} (2023), 957--974.

\bibitem{KLL16}
Jeong~Han Kim, Choongbum Lee, and Joonkyung Lee, \emph{Two approaches to {S}idorenko's conjecture}, Trans. Amer. Math. Soc. \textbf{368} (2016), 5057--5074.

\bibitem{L11}
L\'{a}szl\'{o} Lov\'{a}sz, \emph{Subgraph densities in signed graphons and the local {S}imonovits-{S}idorenko conjecture}, Electron. J. Combin. \textbf{18} (2011), Paper 127, 21.

\bibitem{R11}
Nathan Ross, \emph{Fundamentals of {S}tein's method}, Probab. Surv. \textbf{8} (2011), 210--293.

\bibitem{SW17}
A.~Saad and J.~Wolf, \emph{Ramsey multiplicity of linear patterns in certain finite abelian groups}, Q. J. Math. \textbf{68} (2017), 125--140.

\bibitem{S91}
A.~F. Sidorenko, \emph{Inequalities for functionals generated by bipartite graphs}, Diskret. Mat. \textbf{3} (1991), 50--65.

\bibitem{S93}
Alexander Sidorenko, \emph{A correlation inequality for bipartite graphs}, Graphs Combin. \textbf{9} (1993), 201--204.

\bibitem{S14}
Balazs Szegedy, \emph{An information theoretic approach to {S}idorenko's conjecture},  (2014), arXiv:1406.6738.

\bibitem{Tho89}
Andrew Thomason, \emph{A disproof of a conjecture of {E}rdős in {R}amsey theory}, J. London Math. Soc. (2) \textbf{39} (1989), 246--255.

\bibitem{Ver21}
Leo Versteegen, \emph{Linear configurations containing 4-term arithmetic progressions are uncommon}, J. Combin. Theory Ser. A \textbf{200} (2023), Paper No. 105792, 38.

\end{thebibliography}

\appendix
\section{Proof of \cref{lem:dep-stein}}
\label{app:stein}

In this section, we prove \cref{lem:dep-stein}.
\begin{proof}[Proof of \cref{lem:dep-stein}]

We will prove that \[ d_{\text{Wass}}(W,Z) \leq \rho^{-3} \sum_{u \in V} \left( 3 \sum_{v,w \in N_u}\left| \EE [X_uX_vX_w] \right|  + 4 \sum_{v \in N_u} | \EE [X_uX_v]|\cdot \EE \left[ \left| \sum_{w \in N_u \cup N_v} X_w  \right| \right]  \right). \]
This will  imply the desired conclusion because we have the following (see for instance \cite[Theorem 3.1]{BC05}) for any $x \in \RR$:
\[ |\PP[W \leq x] - \PP[Z \leq x]| \leq \sqrt{2 d_\textup{Wass}(W,Z)/\pi}. \]

Define the set of test functions  \[\mathcal{D} = \left\{ f\in \mathcal{C}^1(\RR): f' \text{ absolutely continuous},\, \norm{f}_{\infty} \leq 1,\, \norm{f'}_{\infty} \leq \sqrt{2/\pi}, \,\norm{f''}_{\infty} \leq 2 \right\}.\] 

By \cref{cor:Stein}, it follows that
\[ d_\text{Wass}(W,Z) \leq \sup_{f \in \mathcal{D}} \left| \EE[f'(W) - Wf(W)] \right|.\]

Define
\begin{align*}
    &Y_v = \rho^{-1} \sum_{u \in N_v} X_u, \quad J_v = \rho^{-1}\sum_{u \not \in N_v} X_u,\quad  K_{vw} = \rho^{-1} \sum_{u \not \in N_w \cup N_v}X_u,  
     \quad T_{vw} = \rho^{-1}\sum_{u \in N_w \setminus N_v} X_u. 
\end{align*}
In particular, observe that for any choice of $v,w \in V$, we have $Y_v+K_{vw}+T_{vw}=W$.

Write $\EE [Wf(W) -  f'(W)] = (\textrm{I}) + (\textrm{II}) + (\textrm{III})$ as a telescoping sum, where
\begin{align*}
       (\textrm{I})  &= \EE Wf(W) - \rho^{-1} \sum_{v \in V} \EE X_v Y_v f'(J_v), \\
   (\textrm{II}) &= \rho^{-1} \sum_{v \in V} \EE X_v Y_v f'(J_v) - \rho^{-2} \sum_{v \in V} \sum_{u \in N_v} \EE[X_uX_v] \EE[f'(K_{uv})], \\
    (\textrm{III}) &= \rho^{-2} \sum_{v \in V} \sum_{u \in N_v} \EE[X_uX_v] \left( \EE f'(K_{uv}) - \EE f'(W) \right).
\end{align*}
By the triangle inequality, we have 
\begin{equation}\label{eq:tri-ineq}
    |\EE Wf(W) - \EE f'(W)| \leq |(\textrm{I})| + |(\textrm{II})| + |(\textrm{III})|.
\end{equation}
In what remains, we upper bound each of the individual pieces $|(\textrm{I})|$, $|(\textrm{II})|$ and $|(\textrm{III})|$. 

For $(\textrm{I})$, we write $W = Y_v + J_v$ and by Taylor expansion about $J_v$ to obtain 
\[ Wf(W) = \rho^{-1} \sum_{v \in V} \left(X_v f(J_v) + X_v Y_v f'(J_v) + \frac{1}{2} X_v Y_v^2 f''(J_1) \right),\]
where $J_1$ is a random variable in the interval $[J_v, W]$. It follows by the assumptions on $f$ that 
\begin{align*}
    |(\textrm{I})| &\leq \rho^{-1} \sum_{v \in V} |\EE (X_v Y_v^2)| \leq \rho^{-3} \sum_{v \in V} \sum_{u,w \in N_v} |\EE X_u X_v X_w|. 
\end{align*}
For $(\textrm{II})$, fix $v \in V$ and by writing $J_v = K_{vw} + T_{vw}$, we know from Taylor expansion about $K_{vw}$ that 
\[ X_vY_vf'(J_v) = \rho^{-1} \sum_{w \in N_v} X_v X_w(f'(K_{vw}) + T_{vw}f''(J_2))\]
where $J_2$ is a random variable in the interval $[K_{vw}, J_v]$. By the assumptions on $f$, we have 
\begin{align*}
|(\textrm{II})| &\leq 2 \rho^{-2} \sum_{v \in V} \left |\EE \left[X_v \sum_{w \in N_v}X_w T_{vw} \right]\right| \leq 2 \rho^{-3} \sum_{v \in V} \sum_{u, w \in N_v} |\EE X_vX_wX_u|. 
\end{align*}
Finally, for $(\textrm{III})$, fix $u, v \in V$. Since $K_{uv} = W - (Y_u + T_{uv})$, by the mean value theorem,  we have 
\[ f'(K_{uv}) = f'(W) - (Y_u + T_{uv})f''(J_3)\]
for some $J_3 \in [K_{uv}, W]$. Consequently, we have 
\begin{align*}
    |\EE f'(K_{uv}) - \EE f'(W)| &\leq 2 \EE|Y_u + T_{uv}| \leq 2 \rho^{-1} \EE \left|\sum_{w \in N_u \cup N_v}X_w \right|.
\end{align*}
Putting all the various pieces into \cref{eq:tri-ineq} gives the desired bound.
\end{proof}

\section{Deferred proofs from \cref{sec:2*5}}

\subsection{Case A} \label{app:A}

\begin{proof}[Proof of \cref{lem:finding-non-coincidental}] Suppose first that exactly one of $L_1,\dots,L_5$ is common. Without loss of generality, let $$L_5 \mathbf x =x_1-x_2+\lambda x_3-\lambda x_4,\quad \lambda\in\QQ\setminus\{ -1,0,1\}.$$
We use $a\sim b$ to denote that $a,b$ are $L$-coincidental, i.e., $\frac{a}{b}\in\{\lambda,-\lambda, 1/\lambda,-1/\lambda\}$. Up to replacing $\lambda$ by one of $-\lambda,1/\lambda,-1/\lambda$ and permuting the variables, we have the following two cases:
\begin{enumerate}
    \item For every $i=1,\dots,4$, letting $L_i \mathbf x =a_{i,1}x_1+\dots+a_{i,5}x_5$, we have $a_{i,1}\sim a_{i,2}$ or $a_{i,3}\sim a_{i,4}$.
    \item Letting $L_4 \mathbf x =a_{4,1}x_1+a_{4,2}x_2+a_{4,3}x_3+a_{4,5}x_5$, we have $a_{4,1}\sim a_{4,3}$ and $a_{4,2}\sim a_{4,5}$.
\end{enumerate}

Write
\begin{alignat*}{3}
    L_5 \mathbf x &=x_1-x_2+\lambda x_3-\lambda x_4 && A_5=\{\!\!\{1,-1,\lambda,-\lambda\}\!\!\}\\
    L_4 \mathbf x &=ax_1+bx_2+c x_3+d x_5 && A_4=\{\!\!\{a,b,c,d\}\!\!\}\\
    L_3 \mathbf x &=(\lambda a-c)x_1+(\lambda b+c)x_2+\lambda cx_4+\lambda dx_5\qquad &&A_3=\{\!\!\{\lambda a-c,\lambda b+c,\lambda c,\lambda d\}\!\!\}\\
    L_2 \mathbf x &=(a+b)x_1+(\lambda b+c)x_3-\lambda bx_4+dx_5 &&A_2=\{\!\!\{a+b,\lambda b+c,-\lambda b,d\}\!\!\}\\
    L_1 \mathbf x &=(a+b)x_2+(c-\lambda a)x_3+\lambda ax_4+dx_5 &&A_1=\{\!\!\{a+b,c-\lambda a,\lambda a,d\}\!\!\}.
\end{alignat*}
Suppose (1) occurs. if each of $A_1,\dots,A_4$ can be partitioned into two $L$-coincidental pairs, then we must have
\begin{align*}
    a\sim b, \quad c\sim d,\quad (\lambda a-c)\sim (\lambda b+c), \quad(\lambda b+c)\sim \lambda b,\quad (c-\lambda a)\sim \lambda a,\quad a+b\sim d.
\end{align*}
Recall that $a\sim b$ corresponds to one of the four identities $\lambda a=b$, $\lambda a=-b$, $\lambda b=a$, $\lambda b=-a$. Thus, for the above relation to hold, $a,b,c,d,\lambda$ must satisfy a system of six equations on five variables, and there are $4^6$ possible such systems of equations. A computer search  \footnote{The code can be found at \texttt{`case-a-1.nb'} in the supplemental file.} yields that none of these $4^6$ systems of equations has a solution with $a,b,c,d\in\QQ\setminus\{0\}$ and $\lambda\setminus\{-1,0,1\}$.

Suppose (2) occurs. Similarly, if each of $A_1,\dots,A_4$ can be partitioned into two $L$-coincidental pairs, then we must have
\begin{align*}
    &a\sim c, \quad b\sim d,\\
    &((\lambda a-c)\sim (\lambda b+c),\quad \lambda c\sim\lambda d)\quad  \text{ or }\quad ((\lambda a-c)\sim\lambda d,\quad (\lambda b+c)\sim\lambda c),\\
    &(-\lambda b\sim (\lambda b+c),\quad(a+b)\sim d)\quad  \text{ or }\quad ((\lambda b+c)\sim  d,\quad (a+b)\sim-\lambda b),\\
    &((c-\lambda a) \sim \lambda a,\quad(a+b)\sim d)\quad  \text{ or }\quad ((c-\lambda a)\sim  (a+b),\quad \lambda a\sim d).
\end{align*}
This is because $b\sim d$, so  we cannot have $-\lambda b\sim d$ in $A_2$. Then we either have $(\lambda b+c)\sim d$ or $(\lambda b+c)\sim -\lambda b$ in $A_2$, which means that $|(\lambda b+c)/ d|$ is an odd power of $|\lambda|$, so we cannot have $(\lambda b+c)\sim \lambda d$ in $A_3$. Then we either have $(\lambda a-c)\sim (\lambda b+c)$ or $(\lambda a-c)\sim \lambda d$ in $A_3$, which means that $|(\lambda a-c)/ d|$ is an even power of $|\lambda|$, so we cannot have $(c-\lambda a)\sim d$ in $A_1$. Therefore, we obtain  $8\cdot 4^6$ possible systems of eight equations on five variables. Again, using a computer search \footnote{The code can be found at \texttt{`case-a-2.nb'} in the supplemental file.}, one can verify that none of these $8\cdot 4^6$  systems  of equations has a solution with $a,b,c,d\in\QQ\setminus\{0\}$ and $\lambda\setminus\{-1,0,1\}$.

Now suppose two of $L_1,\dots,L_5$ are common. Without loss of generality, suppose
\begin{alignat*}{3}
    L_5 \mathbf x &=x_1-x_2+\lambda x_3-\lambda x_4 && A_5=\{\!\!\{1,-1,\lambda,-\lambda\}\!\!\}\\
    L_4 \mathbf x &=x_1+\mu x_2- x_3-\mu x_5  && A_4=\{\!\!\{1,-1,\mu,-\mu\}\!\!\}\\
    L_3 \mathbf x &=(\lambda +1)x_1+(\lambda \mu-1)x_2-\lambda x_4-\lambda \mu x_5 \qquad && A_3=\{\!\!\{\lambda+1,\lambda\mu-1,-\lambda,-\lambda\mu\}\!\!\}\\
    L_2 \mathbf x &=(\mu+1)x_1+(\lambda \mu-1)x_3-\lambda \mu x_4-\mu x_5  && A_2=\{\!\!\{\mu+1,\lambda \mu-1,-\mu,-\lambda\mu\}\!\!\}\\
    L_1 \mathbf x &=(\mu+1)x_2+(-\lambda -1)x_3+\lambda x_4-\mu x_5  && A_1=\{\!\!\{\mu+1,-\lambda-1,\lambda,-\mu\}\!\!\}
\end{alignat*}
with $\lambda,\mu\in\QQ\setminus\{-1,0,1\}$. Since $L_4$ and $L_5$ are both common, $a\sim b$ corresponds to one of the eight  identities $\lambda a=b$, $\lambda a=-b$, $\lambda b=a$, $\lambda b=-a$, $\mu a=b$, $\mu a=-b$, $\mu b=a$, $\mu b=-a$. Since $A_1$ can be partitioned into two cancelling pairs, we must have
$$(\mu+1)\sim(-\lambda-1), \quad  \lambda\sim-\mu$$
or
$$(\mu+1)\sim\lambda, \quad (-\lambda-1)\sim-\mu$$
or
$$(\mu+1)\sim-\mu,\quad \lambda\sim(-\lambda-1).$$
Thus, $\lambda,\mu\in\QQ\setminus\{-1,0,1\}$ must satisfy one of the $3\cdot 8^2$ possible systems of two equations on two variables. A computer search \footnote{The code can be found at \texttt{`case-a-3.nb'} in the supplemental file.} gives that this is only possible when $(\lambda,\mu)$ have the following values:
\begin{alignat*}{3}
    &(1/2,-1/4),\, (-1/4,1/2),\,(-2,-1/3), \, (-1/3,-2), \\
    &(-2,-2/3),\, (-2/3,-2),\, (1/2,-3/4),\, (-3/4,1/2),\\
    &(-2,3),\, (3,-2),\, (-2,-5),\, (-5,2),\,
    (-2,-2), \, (\lambda,-\lambda/(1+\lambda)).
\end{alignat*}
This is what we obtain just by assuming that $A_1$ can be partitioned into two cancelling pairs. One can then verify that $A_2,A_3$ cannot be both partitioned into $L$-coincidental pairs in any of these solutions.
\end{proof}

\begin{proof}[Proof of \cref{lem:coincidence-isolated}]
Let $E$ denote the collection of those $i\in[5]$ that satisfies the condition in \cref{lem:coincidence-isolated}. That is, for every $i\in E$, there exists some $a\in A_{i}$ such that  $a,b$ are not $L$-coincidental for any $b\in A_i$.

Suppose there exists $i_0\in E$ with $A_{i_0}=\{\!\!\{1,1,1,1\}\!\!\}$ or $\{\!\!\{-1,-1,-1,-1\}\!\!\}$. In this case,  the Fourier template $g:\ZZ\to\CC$ defined by
\begin{align*}
    \begin{cases}
    g(1)=e(1/8)\\
     g(-1)=e(-1/8)\\
     g(r)=0 &\text{otherwise}
    \end{cases}
\end{align*}
has $\sum_{i=1}^5\sigma_{A_i}(g)\leq -2<0$.

Otherwise, suppose there exists $i_0\in E$ with $A_{i_0}=\{\!\!\{1,1,1,-1\}\!\!\}$ or $\{\!\!\{1,-1,-1,-1\}\!\!\}$. In this case, the Fourier template $g:\ZZ\to\CC$  defined by
\begin{align*}
    \begin{cases}
    g(1)=i\\
    g(-1)=-i\\
    g(r)=0 &\text{otherwise}
    \end{cases}
\end{align*}
has $\sum_{i=1}^5\sigma_{A_{i}}(g)\leq -2<0$.

Otherwise, suppose there exists $i_0\in E$ with $A_{i_0}=\{\!\!\{a, a, a,b\}\!\!\}$ or $\{\!\!\{a,a,-a,b\}\!\!\}$ with $|a|\neq |b|$. In this case,  the Fourier template $g:\ZZ\to\CC$ defined by
\begin{align*}
    \begin{cases}
    g(a)= g(-a)=C\\
    g(b)= g(-b)=-1\\
    g(r)=0 \quad \text{otherwise}
    \end{cases}
\end{align*}
has $\sum_{i=1}^5\sigma_{A_i}(g)=-2C^3+O(C^2)$, which is negative for $C>0$ sufficiently large.

Otherwise, suppose there exists $i_0\in  E$ with  $A_{i_0}=\{\!\!\{a,a,b,b\}\!\!\}$ with $|a|\neq|b|$. In this case, consider the joined Fourier template $h:\ZZ^2\to\CC$ defined by $h(r_1,r_2)=g(r_1)\widetilde{g}(r_2)$, where   $g,\widetilde{g}:\ZZ\to\CC$ are 1-dimensional Fourier templates given by
\begin{align*}
    \begin{cases}
        g(a)=g(b)=e(1/8)\\
        g(-a)=g(-b)=e(-1/8)\\
        g(r)=0 \quad \text{otherwise}
    \end{cases}\qquad 
    \begin{cases}
        \widetilde g(a)=\widetilde g(-b)=e(1/8)\\
        \widetilde g(-a)=\widetilde g(b)=e(-1/8)\\
        \widetilde g(r)=0 \quad \text{otherwise.}
    \end{cases}
\end{align*}
Observe that, since $a,b$ are not $L$-coincidental, there is no $i\in[5]$ with $A_i=\{\!\!\{a,-a,b,-b\}\!\!\}$. Thus, $\sigma_{A_i}(h)=\sigma_{A_i}(g)\sigma_{A_i}(\widetilde{g})\neq 0$ if and only if $A_i$ equals one of $\pm S_1,\pm S_2,\pm S_3,\pm S_4$, defined by
\begin{align*}
    S_1=\{\!\!\{a,a,b,b\}\!\!\},\quad S_2=\{\!\!\{a,-a,b,b\}\!\!\},\quad S_3=\{\!\!\{a,a,b,-b\}\!\!\},\quad S_4=\{\!\!\{a,a,-b,-b\}\!\!\}.
\end{align*}
One can check that
\begin{align*}
     \sigma_{S_1}(g)&=g(a)^2g(b)^2+g(-a)^2g(-b)^2=-2,\qquad 
      \sigma_{S_1}(\widetilde{g})=\widetilde{g}(a)^2\widetilde{g}(b)^2+\widetilde{g}(-a)^2\widetilde{g}(-b)^2=2,
\end{align*}
and similarly
\begin{alignat*}{3}
    \sigma_{S_2}(g)=0\qquad  & \sigma_{S_3}(g)=0 \qquad &&\sigma_{S_4}(g)=2\\
  \sigma_{S_2}(\widetilde g)=0 \qquad & \sigma_{S_3}(\widetilde g)=0 &&\sigma_{S_4}(\widetilde g)=-2.
\end{alignat*}
Therefore, we have $\sigma_{A_i}(h)\leq 0$ for all $i\in[5]$. Since $A_{i_0}=S_1$ with $\sigma_{A_{i_0}}(h)=\sigma_{A_{i_0}}(g)\sigma_{A_{i_0}}(\widetilde g)=-4$, we get that $\sum_{i=1}^5\sigma_{A_i}(h)<0$.

Otherwise, suppose there exists $i_0\in  E$ with $A_{i_0}=\{\!\!\{a,a,b,-b\}\!\!\}$ with $|a|\neq|b|$. In this case, consider the joined Fourier template $h:\ZZ^2\to\CC$ defined by $h(r_1,r_2)=g(r_1)\widetilde{g}(r_2)$, where   $g,\widetilde{g}:\ZZ\to\CC$ are 1-dimensional Fourier templates given by
\begin{align*}
    \begin{cases}
        g(a)=g(-a)=1\\
        g(b)=i,\,\, g(-b)=-i\\
        g(r)=0 \quad \text{otherwise}
    \end{cases}\qquad 
    \begin{cases}
        \widetilde g(b)=\widetilde g(-b)=1\\
        \widetilde g(a)=i,\,\,
        \widetilde g(-a)=-i\\
        \widetilde g(r)=0 \quad \text{otherwise.}
    \end{cases}
\end{align*}
Again,  since $a,b$ are not $L$-coincidental, there is no $i\in[5]$ with $A_i=\{\!\!\{a,-a,b,-b\}\!\!\}$. Thus, $\sigma_{A_i}(h)=\sigma_{A_i}(g)\sigma_{A_j}(\widetilde{g})\neq 0$ if and only if $A_i$ is one of $\pm S_2,\pm S_3$ defined above.
One can check that
\begin{alignat*}{3}
      & \sigma_{S_2}(g)=2 \qquad \sigma_{S_3}(g)=-2 \qquad  \sigma_{S_2}(\widetilde g)=-2 \qquad \sigma_{S_3}(\widetilde g)=2.
\end{alignat*}
Therefore, we have $\sigma_{A_i}(h)\leq 0$ for all $i\in[5]$. Since $A_{i_0}=S_3$ with $\sigma_{A_{i_0}}(h)=\sigma_{A_{i_0}}(g)\sigma_{A_{i_0}}(\widetilde g)=-4$, we get that $\sum_{i=1}^5\sigma_{A_i}(h)<0$.

Otherwise, suppose there exists $i_0\in E$ with $A_{i_0}=\{\!\!\{a, a, b,c\}\!\!\}$ or $\{\!\!\{a,-a,b,c\}\!\!\}$, $|a|, |b|,|c|$ distinct, and $a$ is not $L$-coincidental with $b $ or $c$. In this case, 
the Fourier template $g:\ZZ\to\CC$ defined by
\begin{align*}
    \begin{cases}
        g(a)=g(-a)=C\\
        g(b)=g(-b)=1\\
        g(c)=g(-c)=-1\\
        g(r)=0 \quad \text{otherwise}
    \end{cases}
\end{align*}
has $\sum_{i=1}^5\sigma_{A_i}(g)=-2C^2+O(C)$, which is negative for $C>0$ sufficiently large.

Finally, suppose for all  $i\in E$ and all $a\in A_i$ not $L$-coincidental with any element in $ A_i$, we have $|A_i\cap\{a,-a\}|=1$.
In this case, pick any  $i_0\in E$ and suppose $A_{i_0}=\{\!\!\{a_0,b_0,c_0,d_0 \}\!\!\}$, such that $a_0$ is not $L$-coincidental with any of $b_0,c_0,d_0$. Consider the Fourier template $g:\ZZ\to\CC$ defined by
\begin{align*}
    \begin{cases}
    g(a_0)=g(-a_0)=-C\\
    g(b_0)=g(-b_0)=
    g(c_0)=g(-c_0)=
    g(d_0)=g(-d_0)=1\\
   g(r)=0\quad  \text{otherwise.}
    \end{cases}
\end{align*}
Then we have $\sigma_{A_{i_0}}(g)=-2C$. Moreover, for all $i\in[5]\setminus\{i_0\}$, we either have $\sigma_{A_i}(g)=-2C$ or have $|\sigma_{A_i}(g)|=O(1)$. Hence for  sufficiently large $C>0$, we have $\sum_{i=1}^5 \sigma_{A_i}(g)<0$.
\end{proof}

\subsection{Case B}
\label{app:B}

We first prove \cref{lem:periodic}. To do so, we introduce the notion of ``periodic sum'' for $h_{A_i}$ on the multiplicative grid $G_L$.

\begin{definition}\label{def:periodic-sum}
    Suppose $h:\pm G_L\to\CC$ is  function such that $h_{A_1},\dots,h_{A_5}$ are $\vec u$-periodic (recall \cref{def:mult-grid}). Define
    \[
    \sigma(h_{A_i},\vec u)=\sum_{\substack{0 \leq d_1 \leq u_1-1 \\ \vdots \\ 0 \leq d_{\ell} \leq u_{\ell} -1}}  h_{A_i}(p_1^{d_1} \ldots p_{\ell}^{d_\ell}\cdot r).
    \]
\end{definition}

\begin{proof}[Proof of \cref{lem:periodic}]
    For $D>0$, define 
$G_D=\{p_1^{d_1}\cdots p_\ell^{d_\ell}:d_1,\dots,d_\ell\in\{0,1,\dots,D\}\}$ to be a finite truncation of $G_L$, and $g \colon \pm G_D \to \CC$  by $g(r) = h(r)$. Observe that, if $r\in \pm G_D$ satisfies
\[
\{ap_1^{d_1}\cdots p_\ell^{d_\ell}r: a\in A_1\cup\dots\cup A_5,\, 0\leq d_1<u_1,\,\dots,\,0\leq d_\ell<u_\ell\}\subseteq \pm G_D,  \tag{*}
\]
then we will have
$$\sum_{0\leq d_1<u_1}\cdots \sum_{0\leq d_\ell <u_\ell} g_{A_i}(p_1^{d_1}\cdots p_{\ell}^{d_\ell}r)=\sigma_{A_i}(h,\vec u).$$
As $D\to\infty$, we have
\[
\frac{|\{r\in\ZZ:g_{A_i}\neq 0\text{ for some }i\in[5]\}|}{|\{r\in\pm G_D:\text{$r$ satisfies (*)}\}|}\to 1.
\]
This captures the idea that the fraction of ``boundary terms'' become negligible and the main contribution comes from the internal portion of the grid. That is, we have
\begin{align*}
    1=\lim_{D\to\infty}\frac{\sum_{i=1}^5\sigma_{A_i}(g)}{\sum_{i=1}^5\sum_{\substack{r\in \pm G_D\text{ satisfies (*)}}}g_{A_i}(r)}=\lim_{D\to\infty}\frac{\sum_{i=1}^5\sigma_{A_i}(g)}{\sum_{i=1}^52\text{Re}(\sigma_{A_i}(h,\vec u))\cdot D^\ell / u_1\cdots u_\ell}.
\end{align*}
Therefore, given $\sum_{i=1}^5 \text{Re}(\sigma_{A_i}(h,\vec u))<0$, we know that $\sum_{i=1}^5\sigma_{A_i}(g)<0$ for $D$ sufficiently large.
\end{proof}

\begin{proof}[Proof of \cref{fact:cos-2}]
    For $\gamma_1,\gamma_2\in\QQ\setminus\{0\}$, we wish to find $t\in\RR$ such that $1+\cos(\gamma_1t)+\cos(\gamma_2t)<0$.  Without loss of generality, suppose $\gamma_1,\gamma_2>0$. The statement clearly holds for $\gamma_1=\gamma_2$, so suppose $\gamma_1<\gamma_2$. 
    
    If  $\gamma_1>\gamma_2/2$, we can take $t=\pi/\gamma_2$, so that $\gamma_2t=\pi$ and $\pi/2< \gamma_1t<\pi$. 
    
    If $\gamma_2/3\leq \gamma_1<\gamma_2/2$, we can take $t=\frac{8\pi}{3\gamma_2}$, so that $\gamma_2t=\frac{8\pi}{3}$ and $\frac{8\pi}{9}\leq \gamma_1t<\frac{4\pi}{3}$. 
    
    If $\gamma_1<\gamma_2/3$, then $(\frac{\gamma_2}{\gamma_1}\cdot \frac{2\pi}{3},\frac{\gamma_2}{\gamma_1} \cdot \frac{4\pi}{3})$ has length $>2\pi$, so we can find $t$ such that $\gamma_1t,\gamma_2t\in (\frac{2\pi}{3},\frac{4\pi}{3})$. 
    
    Finally, if $\gamma_2=2\gamma_1$, we can take $t=1.4\pi/\gamma_1$, so that $1+\cos(1.4\pi)+\cos(2.8\pi)<0$.
\end{proof}

\begin{proof}[Proof of \cref{fact:cos-1}]
    Without loss of generality, assume $0<\gamma_1 \leq \gamma_2 \leq \gamma_3 \leq \gamma_4 $. Suppose for the sake of contradiction that $f(t) = 1 + \cos(\gamma_1 t) + \cos(\gamma_2 t) + \cos(\gamma_3 t) + \cos(\gamma_4 t) \geq 0$ for all $t \in \RR$. In particular, we have that $f(t_k) \geq 0$ for all $k \in \ZZ$, with $t_k := \frac{(2k+1) \pi}{\gamma_4}$ . This choice of $t_k$ ensures that $\cos(\gamma_4 t_k) = -1$, so we have $f(t_k) = \cos(\gamma_1 t_k) + \cos(\gamma_2 t_k) + \cos(\gamma_3 t_k)\geq 0$.

    Observe that for any $m>0$, we have
    \[ \sum_{k=0}^{\gamma_4 - 1} \cos(m t_k) =\sum_{k=0}^{\gamma_4 - 1} \cos\left(\frac{m\pi}{\gamma_4}+\frac{2km\pi }{\gamma_4}\right) 
    = \begin{cases} 0 &\text{if } \gamma_4\nmid m \\ (-1)^{m/\gamma_4} &\text{if } \gamma_4\mid m. \end{cases}\]
    By the maximality of $\gamma_4$ and non-negativity of $f(\cdot)$, it follows that 
    \[ 0 \leq \sum_{k=0}^{\gamma_4 - 1}f(t_k) = \sum_{j=1}^{3}\sum_{k=0}^{\gamma_4 - 1} \cos(\gamma_j t_k) = - \gamma_4 |\{i \in [3]: \gamma_i = \gamma_4 \}|\leq 0.\]
    Consequently equality holds in the above chain of inequalities, which implies that we have:
    \begin{itemize}
        \item $f(t_k) = 0$ for all $0 \leq k \leq \gamma_4 - 1$,
        \item $0<\gamma_1, \gamma_2, \gamma_3 < \gamma_4$. 
    \end{itemize}
    Thus, for all $i,j \in [3]$, we have  $|\gamma_i - \gamma_j| < \gamma_4$ and $0 < \gamma_i + \gamma_j < 2 \gamma_4$. This gives 
    \begin{align*}
        0 = \sum_{k=0}^{\gamma_4 - 1}f(t_k)^2 &= \sum_{i,j = 1}^{3}\sum_{k=0}^{\gamma_4 - 1}\cos(\gamma_it_k)\cos(\gamma_jt_k)  
        = \frac{1}{2} \sum_{i,j=1}^{3}\sum_{k=0}^{\gamma_4 - 1}(\cos((\gamma_i + \gamma_j)t_k) + \cos((\gamma_i - \gamma_j)t_k)) \\
        &= \frac{\gamma_4}{2} (| \{(i,j): i,j \in [3], \gamma_i = \gamma_j \}| - |\{(i,j): i,j \in [3], \gamma_i + \gamma_j = \gamma_4| ).
    \end{align*}
Since $| \{(i,j): i,j \in [3], \gamma_i = \gamma_j \}|\geq 3$, we must have $|\{(i,j): i,j \in [3], \gamma_i + \gamma_j = \gamma_4|\geq 3$, which implies that $\gamma_1+\gamma_3=2\gamma_2=\gamma_4$.
In particular, we have 
\begin{align*}
    &\cos(\gamma_3\pi/\gamma_4) = \cos(\pi-\gamma_1\pi/\gamma_4)=-\cos(\gamma_1\pi/\gamma_4),\qquad \cos(\gamma_2\pi/\gamma_4)=\cos(\pi/2)=0.
\end{align*}
Thus, at $t_0=\pi/\gamma_4$, we have
    \begin{itemize}
        \item $f(t_0) = \cos(\gamma_1t_0)+\cos(\gamma_2t_0)+\cos(\gamma_3t_0)=0$, and
        \item $f'(t_0) = -\gamma_1\sin(\gamma_1t_0)-\gamma_2\sin(\gamma_2t_0)-\gamma_3\sin(\gamma_3t_0)<0$ since $0 < \gamma_1t_0, \gamma_2t_0, \gamma_3t_0 <\gamma_4t_0= \pi$.
    \end{itemize}
    It follows that  $f(t_0+ \epsilon) < 0$ for sufficiently small $\epsilon >0$, which is a contradiction as desired.
\end{proof}

In the remainder of this section we prove \cref{prop:cancelling}. In order to handle the more intricate cases when some of the $A_i$ are cancelling, we need to introduce more notation. 

\begin{definition}
    For primes $p\neq \widetilde p$ and multiset $A$, define $V_{p,\widetilde p}(A)=\{\!\!\{(v_{p}(a),v_{\widetilde p}(a)):a\in A\}\!\!\}$.
\end{definition}

\begin{example}
    Suppose $A=\{\!\!\{-1,2,6,-12\}\!\!\}$. Then
    \begin{itemize}
        \item $V_2(A)=\{\!\!\{0,1,1,2\}\!\!\}$, $V_3(A)=\{\!\!\{0,0,1,1\}\!\!\}$,
        \item $V_{2,3}(A)=\{\!\!\{(0,0),(1,0),(1,1),(2,1)\}\!\!\}$.
    \end{itemize} 
\end{example}

\begin{observation}\label{obs:v-additive}
    Consider $a,a'\in\ZZ\setminus\{0\}$ and prime number $p$. 
    
    If $v_{p}(a)\neq v_{p}(a')$, then $v_{p}(a+a')=v_{p}(a-a')=\min\{v_{p}(a),v_{p}(a')\}$.
    
    If $v_{p}(a)= v_{p}(a')$, then $v_{p}(a+a'),v_{p}(a-a')\geq v_{p}(a)$.
\end{observation}

Recall that in Case B, we assume (\cref{setup:b})  that $L$ is of the form
\begin{alignat*}{3}
    L_5 \mathbf x &=x_1+x_2-x_3-x_4 && A_5=\{\!\!\{1,1,-1,-1\}\!\!\}\\
    L_4 \mathbf x &=ax_1+bx_2+cx_3+dx_5 && A_4=\{\!\!\{a,b,c,d\}\!\!\}\\
    L_3 \mathbf x &=(a+c)x_1+(b+c)x_2-cx_4+dx_5\qquad && A_3=\{\!\!\{a+c,b+c,-c,d\}\!\!\}\\
    L_2 \mathbf x &=(a-b)x_1+(b+c)x_3+bx_4+dx_5 && A_2=\{\!\!\{a-b,b+c,b,d\}\!\!\}\\
    L_1 \mathbf x &=(b-a)x_3+(a+c)x_3+ax_4+dx_5 && A_1=\{\!\!\{b-a,a+c,a,d\}\!\!\}
\end{alignat*}
with $a,b,c,d\in\ZZ\setminus\{0\}$, and none of $L_1,\dots,L_4$ is common. By replacing $a,b,c,d$ by their negations if necessary, we may suppose that at least two of $a,b,c,d$ are positive.

\begin{lemma}
    Suppose $L$ is an irredundant $2\times 5$ linear system that falls under Case B in \cref{cor:2*5-classification}. Assume \cref{setup:b}. Then there cannot be three cancelling sets among $A_1,\dots,A_4$.
\end{lemma}
\begin{proof}
    By contradiction, suppose three among these sets are cancelling. Up to isomorphism, we can suppose $A_2,A_3,A_4$ are cancelling and $a>0$. We then have four different cases dependings on the coefficients in $L_1,\dots,L_5$:
    \begin{enumerate}
        \item $a>0$, $b>0$, $c<0$, $d<0$, $a+c>0$, $b+c<0$, $a-b>0$,

        $ab=cd$, $-c(a+c)=(b+c)d$, $b(a-b)=d(b+c)$;

        \item $a>0$, $b>0$, $c<0$, $d<0$, $a+c<0$, $b+c>0$, $a-b<0$,

        $ab=cd$, $-c(b+c)=(a+c)d$, $b(b+c)=d(a-b)$;
        
        \item $a>0$, $c>0$, $b<0$, $d<0$, $a+c>0$, $b+c>0$, $a-b>0$,

        $ac=bd$, $(a+c)(b+c)=-cd$, $(a-b)(b+c)=bd$;

        \item $a>0$, $d>0$, $b<0$, $c<0$, $a+c<0$, $b+c<0$, $a-b>0$,

        $ad=bc$, $(a+c)(b+c)=-cd$, $d(a-b)=b(b+c)$.
    \end{enumerate}
    One can then check by a direct computation that none of these cases has real solutions.
\end{proof}

\begin{corollary}\label{cor:classification-case-b}
    Suppose $L$ is an irredundant $2\times 5$ linear system that falls under Case B in \cref{cor:2*5-classification}. Assume \cref{setup:b}. Then up to isomorphism, at least one of the following occurs: 
    \begin{enumerate}
        \item $A_4$ is cancelling and $V_p(A_4)=\{\!\!\{ 0,v_1,v_2,v_1+v_2 \}\!\!\}$ with $0<v_1<v_2$ for some prime $p$.
        \item $A_4$ is cancelling and $V_p(A_4)=\{\!\!\{ 0,v,v,2v \}\!\!\}$ with $v>0$ for some prime $p$.
        \item $A_4$ is cancelling and $V_{p,\widetilde p}(A_4)=\{\!\!\{ (0,0),(v_1,0),(0,v_2),(v_1,v_2) \}\!\!\}$ with $v_1,v_2>0$ for some primes $p\neq \widetilde p$.   Moreover, none of $A_1,A_2,A_3$ is cancelling.
        \item $A_4$ is cancelling and $V_{p,\widetilde p}(A_4)=\{\!\!\{ (0,0),(v_1,0),(0,v_2),(v_1,v_2) \}\!\!\}$ with $v_1,v_2>0$ for some primes $p\neq \widetilde p$.  Moreover, there is exactly one $i\in[3]$ such that the following two conditions hold:
        \begin{itemize}
            \item[(i)] $A_i$ is cancelling, and
            \item[(ii)] $V_{p',\widetilde p'}(A_i)=\{\!\!\{ (0,0),(w_1,0),(0,w_2),(w_1,w_2) \}\!\!\}$ with $w_1,w_2>0$ for some primes $p'\neq \widetilde p'$.
        \end{itemize}
    \end{enumerate}
\end{corollary}

We now state a reduction step which says that, when $h:\pm G_L\to\CC$ satisfies certain conditions, we can combine it with constructions in \cref{obs:b-1-1} and \cref{prop:b-1-2} appropriately, to obtain Fourier template $g$ with $\sum_{i=1}^5\sigma_{A_i}(g)<0$.

\begin{definition}
    Let $A$ be a multiset of size 4. We say that $A$ is (P2) if it contains two positive and two negative elements, and is (P3) if it contains three positive and one negative elements.
\end{definition}

\begin{proposition}
    \label{obs:b-reduction-2}
    Suppose $L$ is an irredundant $2\times 5$ linear system that falls under Case B in \cref{cor:2*5-classification}. Assume \cref{setup:b}. Moreover, suppose $A_4$ is cancelling, and one of the following holds:
    \begin{enumerate}
        \item One of $A_1,A_2,A_3$ is cancelling, and there exists $h:\pm G_L\to\CC$ such that $h_{A_1},\dots,h_{A_5}$ are $\vec u$-periodic, with $\sum_{i\in[5],\,A_i\text{ cancelling}}\sigma(h_{A_i},\vec u)<0$.
        \item None of $A_1,A_2,A_3$ is cancelling, and there exists $h:\pm G_L\to\CC$ such that $h_{A_1},\dots,h_{A_5}$ are $\vec u$-periodic, with $\sigma(h_{A_5},\vec u)+\sigma(h_{A_4},\vec u)=0$ and $|\{i\in[3]: \text{Re}(\sigma(h_{A_i},\vec u)) \neq 0\}| \in\{1,3\}$.
    \end{enumerate}
Then there exists a Fourier template $g:\ZZ^d\to\CC$ such that $\sum_{i=1}^5\sigma_{A_i}(g)<0$.
\end{proposition}
\begin{remark}
    The caveat here is that, if for some $h:\pm G_L\to\CC$ we have $\sigma(h_{A_5},\vec u)+\sigma(h_{A_4},\vec u)=0$ and $|\{i\in[3]:\text{Re}(\sigma(h_{A_i},\vec u))\neq 0\}|=2$, we might not be able to produce, based on this, a Fourier template $g$ with $\text{Re}\sum_{i=1}^5\sigma(h_{A_i},\vec u)$ \textit{strictly less than} 0. (For example, $1-1+xe^{i\gamma t}-xe^{i\gamma t}+0= 0$ no matter what value $t$ takes.)
\end{remark}
\begin{proof}[Proof of \cref{obs:b-reduction-2}]
We work under \cref{setup:b}. One can check that, since $A_4$ is cancelling, either $A_1,\dots,A_4$ are all (P2), or two of them are (P2) and the other two are (P3).

Recall that the function $h_2$ defined in \cref{obs:b-1-1} satisfies
\begin{align*}
    (h_2)_{A_i}=\begin{cases}
        1 & \text{$A_i$ is (P2)}\\
        -1 & \text{$A_i$  is (P3)}.
    \end{cases}
\end{align*}
The function $h_3$ (with $\theta_1,\dots,\theta_\ell,t$ to be determined) defined in \cref{prop:b-1-2} satisfies
\begin{align*}
    (h_3)_{A_i}=\begin{cases}
        1 & \text{$A_i$ is cancelling}\\
        e^{i\gamma t},\, \gamma\neq 0 & \text{$A_i$ is (P2) but  not cancelling.}
    \end{cases}
\end{align*}

 Suppose condition (1) in \cref{obs:b-reduction-2} occurs.  Then we can pick $\theta_1,\dots,\theta_\ell,t$ so that the corresponding $h_3$ has
\[
\text{Re}\sum_{\substack{\text{$A_i$ is (P2) but}\\\text{not cancelling}}}\sigma((hh_3)_{A_i},\vec u)\leq 0,
\]
which then gives
\[
\text{Re}\sum_{\text{$A_i$ is (P2)}}\sigma((hh_3)_{A_i},\vec u)< 0.
\]
Let $g$ be the finite truncation of $hh_3$ (recall \cref{lem:periodic}) and $g_2$ be the finite truncation of $h_2$. Then, by \cref{prop:join}, we have either $
\sum_{i=1}^5\sigma_{A_i}(g)<0$ or $\sum_{i=1}^5\sigma_{A_i}(gg_2)<0$.

Suppose condition (2) in \cref{obs:b-reduction-2} occurs. If exactly one $i_0\in[3]$ has $\text{Re}(\sigma(h_{A_{i_0}},\vec u))\neq 0$, and this $A_{i_0}$ is (P3), then letting $g$ be the finite truncation of $h$ and $g_2$ be the finite truncation of $h_2$, we have either 
$\sum_{i=1}^5\sigma_{A_i}(g)<0$ or $\sum_{i=1}^5\sigma_{A_i}(gg_2)<0$.

Otherwise, we have 
$$|\{i\in[3]: \text{Re}(\sigma(h_{A_i},\vec u))\neq 0,\, \text{$A_i$ is (P2)} \}|\in\{1,3\},$$ 
so we can pick $\theta_1,\dots,\theta_\ell,t$ such that
\[
\text{Re}\sum_{\text{$A_i$ is (P2)}}\sigma((hh_3)_{A_i},\vec u)< 0.
\]
Again, letting $g$ be the finite truncation of $hh_3$ and $g_2$ be the finite truncation of $h_2$, we have either 
$\sum_{i=1}^5\sigma_{A_i}(g)<0$ or $\sum_{i=1}^5\sigma_{A_i}(gg_2)<0$.
\end{proof}

We will show in \cref{prop:b-2,prop:b-3,prop:b-4,prop:b-5} that, in each of the four cases listed in \cref{cor:classification-case-b}, there exists some Fourier template $g:\ZZ^d\to\CC$ such that $\sum_{i=1}^5\sigma_{A_i}(g)<0$. 

\begin{proposition}\label{prop:b-2}
    Suppose $L$ is an irredundant $2\times 5$ linear system that falls under Case B in \cref{cor:2*5-classification}. Assume \cref{setup:b}. Moreover, suppose $A_4$ is cancelling and $$V_{p_1}(A_4)=\{\!\!\{ 0,v_1,v_2,v_1+v_2 \}\!\!\},\quad \, 0<v_1<v_2<v_1+v_2.$$ Then there exists a Fourier template $g:\ZZ^d\to\CC$ such that $\sum_{i=1}^5\sigma_{A_i}(g)<0$.
\end{proposition}
\begin{proof}Suppose $\text{gcd}(v_1,v_2)=1$. By \cref{obs:v-additive} and the fact that $V_{p_1}(A_4)=\{\!\!\{0,v_1,v_2,v_1+v_2\}\!\!\}$, one can check that $\{V_{p_1}(A_1),V_{p_1}(A_2),V_{p_1}(A_3)\}$ must be one of the following:
\begin{itemize}
    \item $\{\{\!\!\{0,0,0,v_1+v_2\}\!\!\},\{\!\!\{0,v_1,v_1,v_1+v_2\}\!\!\},\{\!\!\{0,v_1,v_2,v_1+v_2\}\!\!\}\}$,
    \item $\{\{\!\!\{0,0,0,v_1\}\!\!\},\{\!\!\{0,v_2,v_2,v_1\}\!\!\},\{\!\!\{0,v_2,v_1+v_2,v_1\}\!\!\}\}$,
    \item $\{\{\!\!\{0,0,0,v_2\}\!\!\},\{\!\!\{0,v_1,v_1,v_2\}\!\!\},\{\!\!\{0,v_1,v_2,v_1+v_2\}\!\!\}\}$,
    \item $\{\{\!\!\{0,v_1,v_1,v_1\}\!\!\},\{\!\!\{0,v_1,v_2,v_2\}\!\!\},\{\!\!\{v_1,v_2,v_1+v_2,0\}\!\!\}\}$.
\end{itemize}
Let
\begin{align*}
    \mathcal V&=\{ \{\!\!\{0,0,0,v_1\}\!\!\},\{\!\!\{0,0,0,v_2\}\!\!\},\{\!\!\{0,0,0,v_1+v_2\}\!\!\},\{\!\!\{0,v_1,v_1,v_1\}\!\!\},\\
    &\qquad \{\!\!\{0,v_1,v_1,v_2\}\!\!\}, \{\!\!\{0,v_1,v_1,v_1+v_2\}\!\!\}, \{\!\!\{0,v_1,v_2,v_2\}\!\!\}, \{\!\!\{0,v_1,v_2,v_1+v_2\}\!\!\}\}
\end{align*}
denote the collection of all 4-multisets that appeared in the above possibilities. Also, let $$W=V_{p_1}(A_4)=\{\!\!\{0,v_1,v_2,v_1+v_2\}\!\!\}.$$

Suppose first that one of $v_1,v_2$ is even. Call this even number $u_e$. Define $\phi:\{0,1,\dots\}\to \RR$ by
\begin{align*}
    \begin{cases}
        \phi(v)=-1 & v\equiv 0,2,\dots, u_e-2\mod 2u_e\\
        \phi(v)=1&\text{otherwise}
    \end{cases}
\end{align*}
and $h=\phi\circ v_{p_1}: \pm G_L\to\RR$.
Let $\vec u=(2u_e,1,\dots,1)$. Then $h_{A_1},\dots,h_{A_5}$ are $\vec u$-periodic, with
\begin{align*}
    \sigma(h_{A_i},\vec u)=\begin{cases}
        2u_e & A_i=A_5\\
        -2u_e & V_{p_1}(A_i)=W\\
        0 & V_{p_1}(A_i)\in \mathcal V\setminus\{W\}.
    \end{cases}
\end{align*}
Since exactly two elements $i\in[4]$ have $V_{p_1}(A_i)=W$, we get that $\sum_{i=1}^5\sigma(h_{A_i},\vec u)=-2u_e<0$, so \cref{lem:periodic} applies.

Suppose $v_1,v_2$ are both odd. In this case, we build $h=\phi\circ v_{p_1}:\pm G_L\to\CC$ that satisfies \cref{obs:b-reduction-2}.  If $v_1+v_2$ is a multiple of 4, define $u_e=v_1+v_2$ and $w_e=v_2-v_1$; if $v_1-v_2$ is a multiple of 4,   define $u_e=v_2-v_1$ and $w_e=v_1+v_2$.

Note that $w_e$ is a multiple of 2 but not 4, and $\text{gcd}(w_e/2,u_e/2)=1$. This means that $w_e$ has order $u_e/2$ in $\ZZ_{u_e}$, so we can define $\phi:\{0,1,\dots\}\to \RR$ by
\begin{align*}
\begin{cases}
      \phi(v)=  -1 & v\equiv 0,1,2w_e,2w_e+1,\dots,-2w_e, -2w_e+1\mod u_e\\
      \phi(v)=   1&\text{otherwise.}
    \end{cases}
\end{align*}
Let $\vec u=(u_e,1,\dots,1)$. Then $h_{A_1},\dots,h_{A_5}$ are $\vec u$-periodic, with
\begin{align*}
    \sigma(h_{A_i},\vec u)=\begin{cases}
        2u_e & A_i=A_5\\
        -2u_e & V_{p_1}(A_i)=\{\!\!\{0,v_1,v_2,v_1+v_2\}\!\!\}\\
        2u_e\text{ or }-2u_e & V_{p_1}(A_i)\in \{ \{\!\!\{0,0,0,v_1+v_2\}\!\!\},\, \{\!\!\{0,v_1,v_1,v_1+v_2\}\!\!\}\}\\
        0 &\text{otherwise.}
    \end{cases}
\end{align*}
This implies the condition of \cref{obs:b-reduction-2}, so we are done.

In general, for $\text{gcd}(v_1,v_2)=w>1$, up to replacing $\phi (v)$ by $\phi (\floor{v/w})$ and the period $\vec u=(u_1,1,\dots,1)$ by $(u_1w_1,1,\dots,1)$, we can still use the previous constructions.
\end{proof}

\begin{proposition}\label{prop:b-3}
    Suppose $L$ is an irredundant $2\times 5$ linear system that falls under Case B in \cref{cor:2*5-classification}. Assume \cref{setup:b}. Moreover, suppose $A_4$ is cancelling and $$V_{p_1}(A_4)=\{\!\!\{ 0,v,v,2v \}\!\!\},\qquad v>0.$$ Then there exists a Fourier template $g:\ZZ^d\to\CC$ such that $\sum_{i=1}^5\sigma_{A_i}(g)<0$.
\end{proposition}
\begin{proof}
    We first assume that $v=1$, i.e., $V_{p_1}(A_4)=\{\!\!\{0,1,1,2\}\!\!\}$. Using \cref{obs:v-additive},
    one can check that $\{V_{p_1}(A_1),V_{p_1}(A_2),V_{p_1}(A_3)\}$ must be one of the following form:
    \begin{itemize}
    \item $\{\{\!\!\{0,0,0,2\}\!\!\},\{\!\!\{0,1,x,2\}\!\!\},\{\!\!\{0,1,x,2\}\!\!\}\}$ for some $x\geq 1$,
    \item $\{\{\!\!\{0,0,0,1\}\!\!\},\{\!\!\{0,1,1,1\}\!\!\},\{\!\!\{0,1,2,1\}\!\!\}\}$,
    \item $\{\{\!\!\{1,x,2,0\}\!\!\},\{\!\!\{1,x,2,0\}\!\!\},\{\!\!\{1,1,2,0\}\!\!\}\}$ for some $x\geq 1$.
\end{itemize}

Suppose $A_i$ is cancelling for some $i\in[3]$. Then we have $V_{p_1}(A_i)=\{\!\!\{0,1,x,2\}\!\!\}$ with $x=1$ or 3. If $x=3$, then we can apply \cref{prop:b-2} to find the desired $g$. If $x=1$, define $h: \pm G_L\to \CC$ by
\begin{align*}
    \begin{cases}
        h(r)=i & v_{p_1}(r)\text{ odd}, \, r>0\\
        h(r)=-i & v_{p_1}(r)\text{ odd}, \, r<0\\
        h(r)=1 & v_{p_1}(r)\text{ even}
    \end{cases}
\end{align*}
so that for every $i\in[5]$, $h_{A_i}$ is $\vec u$-periodic with $\vec u=(2,1,\dots,1)$. Moreover, we have
\begin{align*}
    \sigma(h_{A_i},\vec u)=\begin{cases}
        2 & A_i=A_5\\
        -2 & i\in[4],\,\text{$A_i$ is cancelling.}
    \end{cases}
\end{align*}
This gives $\sum_{\text{$A_i$ cancelling}}\sigma(h_{A_i},\vec u)\leq -2<0$, so \cref{obs:b-reduction-2} applies.

Now suppose that none of $A_1,A_2,A_3$ is cancelling.
In this case, define $h: \pm G_L\to \CC$ by
\begin{align*}
    \begin{cases}
        h(r)=e(1/8) & v_{p_1}(r)\text{ odd}, \, r>0 \text{ or }v_{p_1}(r)\text{ even}, \, r<0\\
        h(r)=e(-1/8) & v_{p_1}(r)\text{ odd}, \, r<0 \text{ or }v_{p_1}(r)\text{ even}, \, r>0.
    \end{cases}
\end{align*}
One can check that
\begin{itemize}
    \item $h_{A_5}=1$ and $h_{A_4}=-1$ on $ \pm G_L$;
    \item if $A_i$ is (P3), and $V_{p_1}(A_i)$ contains an odd number of even elements, then $h_{A_i}$ is a constant function that is either 1 or $-1$;
    \item if $A_i$ is (P2), and $V_{p_1}(A_i)$ contains an even number of even elements, then $h_{A_i}$ is a constant function that is either 1 or $-1$;
    \item otherwise, $\text{Re}(\sigma(h_{A_i},\vec u))=0$.
\end{itemize}
This implies the condition of \cref{obs:b-reduction-2}(2), so we are done.

In general, suppose $v> 1$. If all elements in $V_{p_1}(A_1)\cup V_{p_1}(A_2)\cup V_{p_1}(A_3)$ are divisible by $v$, then we can  use the above constructions  up to replacing $v_{p_1}(r)$ by $\floor{v_{p_1}(r)/v}$. Thus, it suffices to assume that $\{V_{p_1}(A_1),V_{p_1}(A_2),V_{p_1}(A_3)\}$ equals $\{\{\!\!\{0,0,0,2v\}\!\!\},\{\!\!\{0,v,x,2v\}\!\!\},\{\!\!\{0,v,x,2v\}\!\!\}\}$ or 
 $\{\{\!\!\{v,x,2v,0\}\!\!\},\{\!\!\{v,x,2v,0\}\!\!\},\{\!\!\{v,v,2v,0\}\!\!\}\}$ with $v\nmid x$. Define $h:\pm G\to \CC$ by
 \begin{align*}
    \begin{cases}
        h(r)=1 & \floor{v_{p_1}(r)/(2v)}\text{ even}\\
        h(r)=-1 & \floor{v_{p_1}(r)/(2v)}\text{ odd}.
    \end{cases}
\end{align*}
In this case, one can check that
\begin{itemize}
    \item $h_{A_5}=1$ and $h_{A_4}=-1$ on $ \pm G_L$;
    \item $h_{A_1}$, $h_{A_2}$, $h_{A_3}$ are $\vec u$-periodic with $\vec u=(4v,1,\dots,1)$;
    \item $\sigma(h_{A_1},\vec u)$, $\sigma(h_{A_2},\vec u)$, $\sigma(h_{A_3},\vec u)$ are all nonzero.
\end{itemize}
Again, this implies the condition of \cref{obs:b-reduction-2}(2),  so we are done.
\end{proof}

\begin{proposition}\label{prop:b-4}
    Suppose $L$ is an irredundant $2\times 5$ linear system that falls under Case B in \cref{cor:2*5-classification}. Assume \cref{setup:b}. Moreover, suppose $A_4$ is cancelling, none of $A_1,A_2,A_3$ is cancelling, and $$V_{p_1,p_2}(A_4)=\{\!\!\{ (0,0),(v_1,0),(0,v_2),(v_1,v_2) \}\!\!\},\qquad v_1,v_2>0.$$ Then there exists a Fourier template $g:\ZZ^d\to\CC$ such that $\sum_{i=1}^5\sigma_{A_i}(g)<0$.
\end{proposition}
\begin{proof}
     Using \cref{obs:v-additive}, one can check that $\{V_{p_1,p_{2}}(A_1),V_{p_1,p_{2}}(A_2),V_{p_1,p_{2}}(A_3)\}$ must be one of the following form:
\begin{itemize}
    \item $\{\{\!\!\{(v_1,0),(0,0),(x,0),(0,0)\}\!\!\},\{\!\!\{(0,0),(0,v_2),(0,y),(0,0)\}\!\!\},\{\!\!\{(x,0),(0,y),(v_1,v_2),(0,0)\}\!\!\}\}$
    \item $\{\{\!\!\{(0,0),(x,0),(0,y),(v_1,v_2)\}\!\!\},\{\!\!\{(x,0),(0,v_2),(0,0),(v_1,v_2)\}\!\!\},\{\!\!\{(0,y),(0,0),(v_1,0),(v_1,v_2)\}\!\!\}\}$
    \item $\{\{\!\!\{(0,0),(x,0),(0,0),(v_1,0)\}\!\!\},\{\!\!\{(x,0),(0,v_2),(0,y),(v_1,0)\}\!\!\},\{\!\!\{(0,0),(0,y),(v_1,v_2),(v_1,0)\}\!\!\}\}$
    \item $\{\{\!\!\{(0,0),(0,y),(0,0),(0,v_2)\}\!\!\},\{\!\!\{(0,y),(v_1,0),(x,0),(0,v_2)\}\!\!\},\{\!\!\{(0,0),(x,0),(v_1,v_2),(0,v_2)\}\!\!\}\}$
\end{itemize}
where $x,y\geq 0$.

The proof idea is similar to \cref{prop:b-3}. Suppose first that $v_1\mid x$ and $v_2\mid y$. 
In this case, define $h: \pm G_L\to \CC$ by
\begin{align*}
    \begin{cases}
        h(r)=e(1/8) & \floor{v_{p_1}(r)/v_1}\equiv \floor{v_{p_{2}}(r)/v_2} \mod 2, \, r>0 \\
        & \text{ or } \floor{v_{p_1}(r)/v_1}\not\equiv \floor{v_{p_{2}}(r)/v_2} \mod 2, \, r<0\\
        h(r)=e(-1/8) & \floor{v_{p_1}(r)/v_1}\not\equiv \floor{v_{p_{2}}(r)/v_2} \mod 2, \, r>0 \\
        &\text{ or } \floor{v_{p_1}(r)/v_1}\equiv \floor{v_{p_{2}}(r)/v_2} \mod 2, \, r<0.
    \end{cases}
\end{align*}
In this case, one can check that
\begin{itemize}
    \item $h_{A_5}=1$ and $h_{A_4}=-1$;
    \item if $A_i$ is cancelling for some $i\in[3]$, then $A_i=A_4$ and $h_{A_i}=h_{A_4}=-1$;
    \item if $A_i$ is (P3) and $V_{p_1}(A_i)$ contains an odd number of elements in $\{(0,0),(v_1,v_2)\}$ (mod $(2v_1,2v_2)$), then $h_{A_i}$ is a constant function that is either 1 or $-1$;
    \item if $A_i$ is (P2) and $V_{p_1}(A_i)$ contains an even number of elements in $\{(0,0),(v_1,v_2)\}$ (mod $(2v_1,2v_2)$), then $h_{A_i}$ is a constant function that is either 1 or $-1$;
    \item otherwise, $\text{Re}(\sigma(h_{A_i},\vec u))=0$.
\end{itemize}
Thus, in each of the above possibilities of $\{V_{p_1,p_{2}}(A_1),V_{p_1,p_{2}}(A_2),V_{p_1,p_{2}}(A_3)\}$, the condition of \cref{obs:b-reduction-2} holds.

Suppose now that $2x\mid v_1$ and $2y\mid v_2$; moreover, either $x\equiv v_1/2\mod v_1$ or $y\equiv v_2/2\mod v_2$. In this case, define $h: \pm G_L\to \CC$ by $h=h_1h_2$ or $h=h_1h_3$, where
\begin{align*}
    &\begin{cases}
        h_1(r)=1 & \floor{v_{p_1}(r)/v_1}\equiv \floor{v_{p_2}(r)/v_2} \equiv 0\mod 2 \\
         h_1(r)=-1 & \text{otherwise}
    \end{cases}\\
    &\begin{cases}
        h_2(r)=1 & \floor{v_{p_1}(r)/x}+ \floor{v_{p_2}(r)/y} \equiv 0\mod 2 \\
         h_2(r)=i &\floor{v_{p_1}(r)/x}+ \floor{v_{p_2}(r)/y} \equiv 1\mod 2,\, r>0\\
         h_2(r)=-i &\floor{v_{p_1}(r)/x}+ \floor{v_{p_2}(r)/y} \equiv 1\mod 2,\, r<0
    \end{cases}\\
    &\begin{cases}
        h_2(r)=1 & \floor{v_{p_1}(r)/x}+ \floor{v_{p_2}(r)/y} \equiv 1\mod 2 \\
         h_2(r)=i &\floor{v_{p_1}(r)/x}+ \floor{v_{p_2}(r)/y} \equiv 0\mod 2,\, r>0\\
         h_2(r)=-i &\floor{v_{p_1}(r)/x}+ \floor{v_{p_2}(r)/y} \equiv 0\mod 2,\, r<0.
    \end{cases}
\end{align*}
In this case, one can check that
\begin{itemize}
    \item $h_{A_5}=1$ and $h_{A_4}=-1$;
    \item By picking  $h\in\{h_1h_2,h_1h_3\}$ apprpriately, we have $|\{i\in[3]:\text{Re}(\sigma(h_{A_i},\vec u))\neq 0|\in\{1,3\}$.
\end{itemize}

Finally, suppose we have $2v_1\nmid y$ or $2v_2\nmid y$. Set
\begin{align*}
    \begin{cases}
        h(r)=-1&\floor{v_{p_1}(r)/v_1},\floor{v_{p_{2}}(r)/v_2}\text{ both odd}\\
        h(r)=1&\text{otherwise,}
    \end{cases}
\end{align*}
so that  $h_{A_1},\dots,h_{A_5}$ are $\vec u$-periodic with period $\vec u=(2v_1,2v_2,1,\dots,1)$. Moreover, we have $\sigma(h_{A_5},\vec u)=4v_1v_2$, $\sigma(h_{A_4},\vec u)=-4v_1v_2$, and $\sigma(h_{A_3},\vec u),\sigma(h_{A_2},\vec u),\sigma(h_{A_1},\vec u)$ are all nonzero. Thus the condition of \cref{obs:b-reduction-2} always holds, so we are done.
\end{proof}

\begin{proposition}\label{prop:b-5}
    Suppose $L$ is an irredundant $2\times 5$ linear system that falls under Case B in \cref{cor:2*5-classification}. Assume \cref{setup:b}. Moreover, suppose 
    \begin{itemize}
        \item $A_4$ is cancelling and $V_{p_1,p_2}(A_4)=\{\!\!\{ (0,0),(v_1,0),(0,v_2),(v_1,v_2) \}\!\!\}$  with $v_1,v_2>0$.
        \item  $A_{i_0}$ is cancelling for some $i_0\in[3]$, and $V_{p_{j_1},p_{j_2}}(A_i)=\{\!\!\{ (0,0),(w_1,0),(0,w_2),(w_1,w_2) \}\!\!\}$ with $w_1,w_2>0$.
        \item $A_i$ is not cancelling for all $i\in[3]\setminus\{i_0\}$.
    \end{itemize}
    Then there exists a Fourier template $g:\ZZ^d\to\CC$ such that $\sum_{i=1}^5\sigma_{A_i}(g)<0$.
\end{proposition}
\begin{proof}
    Suppose first that $p_1=p_{j_1}$ and $p_2=p_{j_2}$. In this case, we must have $v_1=w_1$ and $v_2=w_2$. Consider $h:\pm G_L\to\CC$ defined by
    \begin{align*}
        \begin{cases}
            h(r)=-1 & \floor{v_{p_1}(r)/v_1} \equiv \floor{v_{p_2}(r)/v_2}\equiv 0\mod 2\\
            h(r)=1 &\text{otherwise}
        \end{cases}
    \end{align*}
    so that $h_{A_4},h_{A_{i_0}}$ are the constant $-1$ function on $\pm G_L$. This implies  \cref{obs:b-reduction-2}(1).

Now suppose $\{p_1,p_2\}\neq \{p_{j_1},p_{j_2}\}$. In this case, since $A_4,A_{i_0}$ are cancelling, $V_{p_1,p_2}(A_{i_0})$ and $V_{p_{j_1},p_{j_2}}(A_4)$ can be partitioned into cancelling pairs.
Consider functions $h,h':\pm G\to\CC$ defined by
    \begin{align*}
        &\begin{cases}
            h(r)=-1 & \floor{v_{p_1}(r)/v_1} \equiv \floor{v_{p_2}(r)/v_2}\equiv 0\mod 2\\
            h(r)=1 &\text{otherwise,}
        \end{cases}\\
        &\begin{cases}
            h'(r)=-1 & \floor{v_{p_{j_1}}(r)/w_1} \equiv \floor{v_{p_{j_2}}(r)/w_2}\equiv 0\mod 2\\
            h'(r)=1 &\text{otherwise.}
        \end{cases}
    \end{align*}
We already know that $h_{A_4}$ and $h_{A_{i_0}}'$ are the constant $-1$ function. Moreover, since $V_{p_1,p_2}(A_{i_0})$ and $V_{p_{j_1},p_{j_2}}(A_4)$ can be partitioned into cancelling pairs, we additionally know that $h'_{A_4}$ and $h_{A_{i_0}}$ are the constant 1 function. Therefore,  $(hh')_{A_4}$ and $(hh')_{A_{i_0}}$ are the constant $-1$ function. Again, this implies the condition of \cref{obs:b-reduction-2}(1), so we are done.
\end{proof}

\subsection{Case D}\label{app:D}

In this subsection, we give the reduction to \cref{eq:reduced-D}. We begin by recalling the setup. Recall that we had the following system $L$, 
\begin{align*}
        &L_5=x_1-x_2+\alpha x_3-\alpha x_4\\
        &L_4=x_1+\beta x_2- x_3-\beta x_5\\
        &L_3=(\alpha+1)x_1+(\alpha\beta -1)x_2-\alpha x_4-\alpha\beta x_5\\
        &L_2=(\beta+1)x_1+(\alpha\beta-1)x_3- \alpha\beta x_4-\beta x_5\\
        &L_1=(\beta+1) x_2+(-1-\alpha)x_3+ \alpha x_4-\beta x_5
    \end{align*}
with $\alpha,\beta\in\QQ$ and $\alpha,\beta>0$, and at least one of $L_1,L_2,L_3$ is common. We then have  one of the following:
\begin{enumerate}
    \item $L_3$ is common, which means that $(\alpha+1)(\alpha\beta-1)=\alpha\cdot\alpha\beta$;
    \item $L_2$ is common, which means that $(\beta+1)(\alpha\beta-1)=\beta\cdot\alpha\beta$;
    \item $L_1$ is common, which means that $\alpha(\beta+1)=\beta(\alpha+1)$.
\end{enumerate}
Note that the first two cases are the same up to switching $\alpha$ and $\beta$. 
The first case implies that $\alpha\beta=\alpha+1$, which gives
\begin{align*}
        &L_5=x_1-x_2+\alpha x_3-\alpha x_4\\
        &L_4=x_1+\frac{\alpha+1}{\alpha}\cdot  x_2- x_3-\frac{\alpha+1}{\alpha}\cdot x_5.
    \end{align*}
The third case implies that $\alpha=\beta$, which gives
\begin{align*}
        &L_5=x_1-x_2+\alpha x_3-\alpha x_4\\
        &L_4=x_1+\alpha x_2- x_3-\alpha x_5.
\end{align*}
From here, we can see that these two cases are the same up to sending $\alpha\mapsto \frac{\alpha+1}{\alpha}$. Therefore, we may restrict our attention to the third case.

\subsection{Case E}\label{app:E}

\begin{proof}[Proof that \cref{eq:E-case 2} is uncommon when $\lambda = 1$]
Recall that the setup is as follows: the elements in $\{L_B:B\in\mathcal C(L)\}$ are
\begin{alignat*}{2}
    &L_5=x_1-x_2+ x_3- x_4 &&A_5=\{\!\!\{1,-1,1,-1\}\!\!\}\\
    &L_1=ax_2+(b-a) x_3+ax_4+cx_5\qquad  &&A_1=\{\!\!\{b-a,a,a,c\}\!\!\}\\
    &L_3=(a-b)x_1+b x_2+b x_4+c x_5 &&A_3=\{\!\!\{a-b,b,b,c\}\!\!\}.
\end{alignat*}

Suppose first that $|a-b|\neq |c|$. Take prime $p$ such that $v:=|v_p(a-b)-v_p(c)|>0$. Then the map $h:\pm G_L\to\CC$ with 
\[
\begin{cases}
    h(r)=1 & \floor{v_p(r)/v}\equiv 0 \mod 2\\
    h(r)=-1 & \floor{v_p(r)/v}\equiv 1 \mod 2
\end{cases}
\]
gives $\text{Re}(h_{A_1})+\text{Re}(h_{A_3})+\text{Re}(h_{A_5})=1-1-1=-1$.

Now suppose $|a-b|= |c|$. Without loss of generality, let $b-a=c$, so we have $A_5=\{\!\!\{1,1,-1,-1\}\!\!\}$, $A_1=\{\!\!\{a,a,c,c\}\!\!\}$, $A_3=\{\!\!\{a+c,a+c,c,-c\}\!\!\}$. If $a,c$ are of the same sign, then we can choose $z\in\CC$ with $|z|=1$ and $\text{Re}(z^4+z^2+1)<0$, and set $h:\pm G_L\to\CC$ by
$$
\begin{cases}
    h(r)=z & r>0\\
    h(r)=\overline{z} &r<0,
\end{cases}$$
so that $\text{Re}(h_{A_5})$, $\text{Re}(h_{A_1})$, $\text{Re}(h_{A_3})$ are constant on $\pm G_L$, with $\text{Re}(h_{A_5})+\text{Re}(h_{A_1})+\text{Re}(h_{A_3})=\text{Re}(z^4+z^2+1)<0$.

If $a,c$ are of different signs, then there exists some prime $p$ that divides exactly one of $|a|$ and $|c|$. Suppose $\{v_p(a),v_p(c)\}=\{v,0\}$ with $v\in\NN$. Again, we choose $z\in\CC$ with $|z|=1$ and $\text{Re}(z^4+z^2+1)<0$. Set $h:\pm G_L\to\CC$ by
$$
h(r)=\begin{cases}
    z & \floor{v_{p}(r)/v}\text{ is even, } r>0\\
    \overline{z} & \floor{v_{p}(r)/v}\text{ is odd, } r>0\\
    \overline{z} & \floor{v_{p}(r)/v}\text{ is even, } r<0\\
    z & \floor{v_{p}(r)/v}\text{ is odd, } r<0.
\end{cases}$$
One can check that $\text{Re}(h_{A_5})$, $\text{Re}(h_{A_1})$, $\text{Re}(h_{A_3})$ are constant on $\pm G_L$, with $\text{Re}(h_{A_5})+\text{Re}(h_{A_1})+\text{Re}(h_{A_3})=\text{Re}(z^4+z^2+1)<0$.
\end{proof}

\end{document}